\documentclass[11pt,fullpage]{article}
\usepackage{amsmath}
\usepackage{amsfonts,amssymb}
\usepackage{amsthm}
\usepackage{graphics,graphicx,pspicture,graphpap}%,picins}
\usepackage{hyperref}
\usepackage[right = 2.5cm, left=2.5cm, top = 2.5cm, bottom =2.5cm]{geometry}
\usepackage{color}
\usepackage{pstricks,pst-node,pst-text,pst-3d}

\numberwithin{equation}{section}

\newcommand{\Real}{\mathbb R}
\newcommand{\Naturals}{\mathbb N}
\newcommand{\Complex}{\mathbb C}
\newcommand{\Integer}{\mathbb Z}
\newcommand{\norm}[1]{\left\lVert#1\right\rVert}
\newcommand{\abs}[1]{\left\vert#1\right\vert}
\newcommand{\set}[1]{\left\{#1\right\}}

\newcommand{\grad}{\nabla}

\newcommand{\K}{\mathcal{K}}

\newcommand{\F}{\mathcal{F}}
\newcommand{\jap}[1]{\langle #1 \rangle}

\newtheorem{theorem}{Theorem}
\theoremstyle{definition}
\newtheorem{remark}{Remark}

\theoremstyle{lemma}

\newtheorem{proposition}{Proposition}

\theoremstyle{definition}
\newtheorem{definition}{Definition}
\theoremstyle{lemma}
\newtheorem{lemma}{Lemma}[section]
\begin{document}

\title{Existence, Uniqueness and Lipschitz Dependence for Patlak-Keller-Segel and Navier-Stokes in $\Real^2$ with Measure-valued Initial Data}

\author{Jacob Bedrossian\footnote{\textit{jacob@cims.nyu.edu}, Courant Institute of Mathematical Sciences. Partially supported by NSF Postdoctoral Fellowship in Mathematical Sciences, DMS-1103765} \, and Nader Masmoudi\footnote{\textit{masmoudi@cims.nyu.edu}, Courant Institute of Mathematical Sciences. Partially supported by NSF  grant  
DMS-0703145 }}

\date{}

\maketitle

\begin{abstract}
We establish a new local well-posedness result in the space of finite Borel measures for mild solutions of the parabolic-elliptic Patlak-Keller-Segel (PKS) model of chemotactic aggregation in two dimensions. 
Our result only requires that the initial measure satisfy the necessary assumption $\max_{x \in \Real^2}\mu(\set{x}) < 8\pi$. 
This work improves the small-data results of Biler \cite{Biler95} and the existence results of Senba and Suzuki \cite{Senba02}.  
Our work is based on that of Gallagher and Gallay \cite{GallagherGallay05}, who prove the uniqueness and log-Lipschitz continuity of the solution map for the 2D Navier-Stokes equations (NSE) %in vorticity-transport form 
with measure-valued initial vorticity. 
We refine their techniques and present an alternative version of their proof which yields existence, uniqueness and \emph{Lipschitz} continuity of the solution maps of both PKS and NSE. 
Many steps are more difficult for PKS than for NSE, particularly on the level of the linear estimates related to the self-similar spreading solutions. 
\end{abstract}

\tableofcontents

\section{Introduction}
The primary focus of this work is establishing a large-data local well-posedness result in the space of finite Borel measures for the parabolic-elliptic Patlak-Keller-Segel model in two dimensions: 
\begin{equation} \label{def:PKS}
\left\{
\begin{array}{l}
  u_t + \grad \cdot (u \grad c) = \Delta u, \\
  - \Delta c = u. 
\end{array}
\right.
\end{equation}
This system is generally considered the fundamental mathematical model for the study of aggregation by chemotaxis of certain microorganisms \cite{Patlak,KS,Hortsmann,HandP}. From now on we will refer to \eqref{def:PKS} as Patlak-Keller-Segel (PKS).
The first equation describes the motion of the microorganism as a random walk with drift up the gradient of the \emph{chemo-attractant} $c$. The second equation describes the production and (instantaneous) diffusion of the chemo-attractant. 
PKS and related variants have received considerable mathematical attention over the years, for example, see the review \cite{Hortsmann} or some of the following representative works \cite{ChildressPercus81,JagerLuckhaus92,Nagai95,Biler95,HerreroVelazquez96,Senba02,Biler06,BlanchetEJDE06,Senba07,Blanchet08,BlanchetCalvezCarrillo08,BlanchetCarlenCarrillo10}. 

An important and well-known property of \eqref{def:PKS} in two dimensions is that it is $L^1$-critical: if $u(t,x)$ is a solution to \eqref{def:PKS} then for all $\lambda \in (0,\infty)$, so is
\begin{equation*}
u_\lambda(t,x) = \frac{1}{\lambda^2}u\left(\frac{t}{\lambda^2}, \frac{x}{\lambda}\right). 
\end{equation*}
It has been known for some time that \eqref{def:PKS} possesses a critical mass: if $\norm{u_0}_1 \leq 8\pi$ then classical solutions exist for all time (see e.g. \cite{Senba02,BlanchetEJDE06,Blanchet08,BlanchetCarlenCarrillo10,BlanchetCalvezCarrillo08}) and if $\norm{u_0}_1 > 8\pi$ then all classical solutions with finite second moment blow up in finite time \cite{JagerLuckhaus92,Nagai95,BlanchetEJDE06} and are known to concentrate at least $8\pi$ mass into a single point at blow-up \cite{Senba02} (see also \cite{HerreroVelazquez96,Senba07}). 
Another important property of \eqref{def:PKS} that plays a decisive role in our work is the existence (and uniqueness) of self-similar spreading solutions for all mass $\alpha \in (0,8\pi)$. 
These are known to be global attractors for the dynamics if the total mass is less than $8\pi$ \cite{BlanchetEJDE06} and for the purposes of our analysis, should be thought of as analogous to the Oseen vortices of the Navier-Stokes equations. 
When studied in higher dimensions, \eqref{def:PKS} is supercritical and the dynamics are quite different, see for example \cite{Biler95,Corrias04,BedrossianIA10,HerreroMedinaVelazquez97,HerreroMedinaVelazquez98}.
Variants of \eqref{def:PKS} involving nonlinear diffusion which are critical in higher dimensions have also been studied \cite{Kowalczyk05,CalvezCarrillo06,SugiyamaDIE06,SugiyamaADE07,Blanchet09,KimYao11} (see also the related \cite{LiebYau87}).
The parabolic-parabolic version of \eqref{def:PKS} has also been analyzed in various contexts (see e.g. \cite{Senba07,CalvezCorrias,HerreroVelazquez97,NaitoSuzukiYoshida02,BilerCorriasDolbeault11}). 
We should also mention that variants of \eqref{def:PKS} have been studied in the context of astrophysics (referred to as \emph{Smoluchowski-Poisson}) as a simplified model for the collapse of overdamped self-gravitating particles undergoing Brownian motion (see e.g. \cite{BilerAccretion95,ChavanisSire04,ChavanisSire06,SireChavanis08}). 
 
The goal of the present work is to prove the most general local well-posedness result known for \eqref{def:PKS}. 
We work with so-called \emph{mild solutions}, motivated by similar notions used in fluid mechanics (other authors have also used this notion for \eqref{def:PKS}). 
See below for the full definition and discussion (Definition \ref{def:MildSolution}), but the main idea is that these solutions satisfy \eqref{def:PKS} as the integral equation
\begin{equation*}
u(t) = e^{t\Delta}\mu - \int_0^t e^{(t-s)\Delta}\grad \cdot (u(s)\grad c(s)) ds, 
\end{equation*}
and satisfy the optimal hypercontractive estimate $\sup_{t \in (0,T)}t^{1/4}\norm{u(t)}_{4/3} < \infty$ (the self-similar spreading solutions show that the rate cannot generally be better).  
We show that there exists a unique mild solution to \eqref{def:PKS} given initial data which is a non-negative, finite Borel measure $\mu$ that satisfies $\max_{x \in \Real^2} \mu(\set{x}) < 8\pi$. Moreover, we also show that the solution map is locally Lipschitz continuous with respect to the total variation norm of the initial data.   
This is the most general well-posedness result possible in this space without considering weaker notions of solutions which can be extended past blow up (such solutions do exist \cite{DolbeaultSchmeiser09}, see below for a discussion).
The mild solutions we construct are smooth for $t > 0$ at least for some short time, which is not possible if there already exists a concentration with critical mass. 
Biler proved \cite{Biler95} using a contraction mapping argument that if the initial measure has a small atomic part then one can construct a unique mild solution. 
Senba and Suzuki \cite{Senba02} construct weak solutions only under the assumption $\max_{x \in \Real^2}\mu(\set{x}) < 8\pi$, however, their solutions are not a priori mild solutions and it is far from clear that more general solutions would necessarily agree with the mild solution.  
Hence, our proof also yields existence, which was to our knowledge open.

Much of our work and motivation is a result of the similarities \eqref{def:PKS} shares with the Navier-Stokes equations in vorticity-transport form
\begin{equation} \label{def:NSE}
\left\{
\begin{array}{l}
  \omega_t + \grad^{\perp} \Psi \cdot \grad \omega = \Delta \omega, \\
   \Delta \Psi = \omega. 
\end{array}
\right.
\end{equation}
Existence of mild solutions to \eqref{def:NSE} with measure-valued initial data was proved earlier in \cite{Cottet86,GigaMiyakawaOsada88} and similar to \eqref{def:NSE}, it is relatively easy to prove well-posedness if the initial data has only very small atoms. 
However, in \cite{GallagherGallay05}, Gallagher and Gallay proved that given arbitrary initial vorticity in the space of finite Borel measures, there is a unique mild solution to \eqref{def:NSE} and the solution map is log-Lipschitz continuous with respect to the total variation norm (see also \cite{GallayWayne05,GallagherGallayLions05} for a proof of uniqueness of the Oseen vortex with point measure initial data).

The proof of Gallagher and Gallay \cite{GallagherGallay05} uses an accurate approximate solution and an intelligent decomposition of the error between the approximate solution and the true solution, which is shown to be very small in an appropriate sense. 
A Gr\"onwall-type estimate is used to prove that if two solutions have the same initial data then they must differ from the approximate solution in the same way and hence are equal. 
However, the argument is not quite a contraction mapping. Consequently, it requires the a priori existence of well-behaved mild solutions and yields log-Lipschitz dependence on initial data, but not Lipschitz.   
Our argument follows the same general principles set forth in \cite{GallagherGallay05}, however, we use a different decomposition which allows stronger results. In particular, unlike \cite{GallagherGallay05}, our argument is a true contraction mapping, 
and this allows us to prove existence of solutions as well as the Lipschitz continuity of the solution maps of both \eqref{def:PKS} and \eqref{def:NSE} (see Theorem \ref{thm:Lipschitz}).

As in \cite{GallagherGallay05}, the approximate solution is constructed by guessing that near a large atomic concentration in the initial data, for short time, the solution to \eqref{def:PKS} or \eqref{def:NSE} should look like a self-similar spreading solution, and elsewhere can be approximated by a linear evolution. 
In order to close a contraction mapping argument, some knowledge about the linearization around the approximate solution is necessary. 
A `brute force' linear analysis is likely intractable, however, it turns out that knowing good spectral properties of all the well-separated pieces of the approximate solution is sufficient to close the argument. 
In particular, we need good spectral properties of the linearization of \eqref{def:PKS} or \eqref{def:NSE} around the self-similar spreading solutions. 
For NSE, the nonlinearity vanishes for radially symmetric data, which is why the Oseen vortices are simply the self-similar solutions to the linear heat equation. 
Moreover, the linearization around the Oseen vortices is relatively easy to analyze, as it is a sum of the Fokker-Planck operator and an operator which is skew-symmetric in an appropriate Hilbert space.
Nothing analogous to these properties hold in the case of the Patlak-Keller-Segel system: the self-similar spreading solutions solve a genuinely nonlinear elliptic system and the spectral properties of the linearization are far from trivial to analyze.
One of the  main tools for dealing with the linearization is a variant of the spectral gap recently obtained by J. Campos and J. Dolbeault \cite{CamposDolbeault12}. 
An independent proof of a weaker version specific to our needs is given in Appendix \S\ref{apx:SpectralGap}.  
This spectral gap needs to be further adapted to the spaces we are working in, similar to what is done in Gallay and Wayne \cite{GallayWayne05} for NSE (see Proposition \ref{prop:SpecT} below).

An additional technicality that appears here is the fact that the velocity field for PKS is not divergence free. 
This makes most of the results of Carlen and Loss \cite{CarlenLoss94} on the fundamental solutions of linear advection-diffusion equations inapplicable. Due to the singular nature of the velocity fields, the linear advection-diffusion equations we study cannot be treated as a perturbation of the heat equation locally in time (see \cite{JiaSverak13} for a related issue) and hence even on the linear level we need to develop tools to carefully deal with questions such as uniqueness and continuity at the initial time. 

Global measure-valued solutions of \eqref{def:PKS} in the sense of Poupaud's weak solutions \cite{Poupaud02}, which make sense even if there are mass concentrations, have been constructed by Dolbeault and Schmeiser in \cite{DolbeaultSchmeiser09} by taking sequences of regularized problems and extracting a measure for $u(t)$ along with an appropriate `defect measure' to make sense of the nonlinear term.  
It appears that the resulting solution depends on the chosen regularization, as the formal dynamics derived by Vel\'azquez \cite{Velazquez04I,Velazquez04II} are different from those constructed by Dolbeault and Schmeiser.
Whether or not measure-valued solutions can be uniquely selected by physically or biologically relevant criteria remains an interesting open question. See  \cite{Velazquez04} for some work in this direction. 

Let us end this introduction by summarizing some of the main difficulties compared to the study 
of the NSE in  \cite{GallagherGallay05}: 

\begin{itemize}

\item[(a)] For the PKS, the vector field is not divergence-free, hence we cannot use 
the results of \cite{CarlenLoss94} on the pointwise decay and localization for 
the fundamental solution of the linear advection-diffusion equation.
 
\item[(b)] For NSE, the existence of a mild solution with strong a priori estimates was already known, a fact which Gallagher and Gallay exploit multiple times in combination with the results of \cite{CarlenLoss94}. 
 For the PKS, the existence of mild solutions with such general initial data was not known. 
 
\item[(c)] The self-similar profiles of \eqref{def:PKS} corresponding to the Oseen vortices are not linear in the mass. 
The critical mass $8\pi$ will appear in many places in our analysis. 

\item[(d)] The linear operators we have to deal with are harder than those that arise in the study of NSE. 
For NSE, due to the divergence-free property, these linear operators are a skew-symmetric perturbation of a Fokker-Planck operator. 
 
\end{itemize} 

In order to overcome difficulties (a) and (b), we had to find a better decomposition of the 
error terms between the solution and the approximate solution.   
This gives better control of the error in norms which permit us to close a contraction mapping argument, allowing also the deduction
of Lipschitz continuity of the solution map with respect to the initial data. To our knowledge, this is a new result even for NSE. 
To overcome difficulty (d), a compactness/rigidity argument is used to prove uniqueness for the singular linear equations and 
a variant of the spectral gap of Campos and Dolbeault \cite{CamposDolbeault12} and known spectral properties of general Fokker-Planck equations both play important roles in the linear analysis.  

\subsection{Results} 
The precise notion of weak solution we are using is that of a \emph{mild solution}, which are motivated by similar notions in fluid mechanics and have also been used previously in the study of PKS (e.g. \cite{Biler95,BlanchetDEF10}). 
\begin{definition}[Mild Solution] \label{def:MildSolution}
Given $\mu \in \mathcal{M}_+(\Real^2)$, we define $u(t)$ to be a \emph{mild solution} to \eqref{def:PKS} with initial data $\mu$ on $[0,T)$ if the following are satisfied:
\begin{itemize}
\item[(i)] $u(t) \rightharpoonup^\star \mu$ as $t \searrow 0$, 
\item[(ii)] $u(t) \in \chi_T$ where
\begin{equation}
\chi_T =  C_w( [0,T]; \mathcal{M}_+(\Real^2)) \cap \set{u(t): \sup_{t \in (0,T)} t^{1/4}\norm{u(t)}_{4/3} < \infty}, \label{def:chiT}
\end{equation}
\item[(iii)] 
$u(t)$ satisfies the following Duhamel integral equation for all $t \in (0,T)$
\begin{equation}
u(t) = e^{t\Delta}\mu - \int_0^t e^{(t-s)\Delta}\grad \cdot (u(s)\grad c(s)) ds, \label{eq:MildSolutionDef} 
\end{equation}
with $-\Delta c(s) = u(s)$ in the sense 
\begin{equation}
c(t,x) = -\frac{1}{2\pi}\int \log\abs{x-y}u(t,y) dy. 
\end{equation}
\end{itemize}  
\end{definition}
\begin{remark} \label{rmk:NonConverge} 
Recall the estimates on the heat kernel, 
\begin{align} 
\norm{e^{t\Delta} f}_p & \lesssim t^{1/p - 1/q}\norm{f}_q, \label{ineq:HeatLpLqEasy} \\ 
\norm{e^{t\Delta}\grad f}_p & \lesssim t^{-1/2 + 1/p - 1/q}\norm{f}_q,  \label{ineq:HeatLpLq}
\end{align} 
which are a consequences of Young's inequality for convolutions. The estimate \eqref{ineq:HeatLpLq} ultimately implies that \eqref{def:chiT} ensures that the Duhamel integral converges in the sense that: 
\begin{align*}
\sup_{t \in (0,T)}t^{1/4}\norm{\int_0^t e^{(t-s)\Delta}\grad \cdot (u(s)\grad c(s)) ds}_{4/3} + \sup_{t \in (0,T)}\norm{\int_0^t e^{(t-s)\Delta}\grad \cdot (u(s)\grad c(s)) ds}_{1} < \infty. 
\end{align*}  
However, if the initial measure has a non-zero atom, the integral does \emph{not} converge to zero in these norms as $t \searrow 0$.
That is, the solution cannot be approximated by the linear heat evolution in the critical norms by choosing $t$ small; the only option would be to impose that the atoms are small (see Theorem \ref{thm:Basics} and the results of Biler \cite{Biler95} below).  
In this general sense, the work here is related to the recent works on 3D NSE in the critical space $L^{3,\infty}$ \cite{JiaSverak12,JiaSverak13}.  
\end{remark}

\begin{remark} 
Here $C_w([0,T];\mathcal{M}_+(\Real^2))$ is the space of $u(t)$ which take values in finite non-negative Borel measures
continuously in time with respect to the weak$^\ast$ topology.  
\end{remark} 

\begin{remark} \label{def:MildSolnLinear}
Often in the sequel we will be studying singular advection-diffusion equations of the form $\partial_t f + \grad \cdot (vf) = \Delta f$ with measure-valued initial data. For these we use a definition of mild solution exactly analogous to Definition \ref{def:MildSolution} except that we will not impose a priori that the solution or initial data is non-negative and of course the velocity field is imposed externally and not derived from the solution itself. 
\end{remark}

In addition to \eqref{ineq:HeatLpLqEasy} and \eqref{ineq:HeatLpLq}, the heat kernel also satisfies the following precise estimate \cite{GigaMiyakawaOsada88}: for all $\mu \in \mathcal{M}(\Real^2)$ and $p \in (1,\infty]$,  
\begin{equation}
\limsup_{t \searrow 0} t^{1-\frac{1}{p}}\norm{e^{t\Delta} \mu}_{p} \lesssim \norm{\mu}_{pp}, \label{ineq:Heatpp}
\end{equation}
where $\norm{\cdot}_{pp}$ denotes the semi-norm on $\mathcal{M}(\Real^2)$ which measures the total variation of the atomic part:
\begin{equation*}
\norm{\mu}_{pp} := \sum_{\set{x \in\Real^2: \abs{\mu(\set{x})} > 0}}\abs{\mu(\set{x})}. 
\end{equation*}
Estimate \eqref{ineq:Heatpp} and related estimates play a key role in our analysis and the work of Gallagher and Gallay, as they show that size conditions for short time results should only depend on the atomic part of the initial data.  
Additional important facts about mild solutions are summarized in the following theorem.
\begin{theorem} \label{thm:Basics}
\begin{itemize}
\item[(i)] Let $u(t)$ be any mild solution to PKS which exists on some time interval $[0,T]$, $T < \infty$.
Then $\sup_{t \in (0,T)}t^{1-1/p}\norm{u(t)}_{p}  < \infty$ for all $p \in [1,\infty]$. 
\item[(ii)] (Biler \cite{Biler95}) There exists some $\epsilon_0 > 0$ such that if $\mu_0 \in \mathcal{M}_+(\Real^2)$ and satisfies 
\begin{equation*}
\limsup_{t \searrow 0} t^{1/4} \norm{e^{t\Delta}\mu_0}_{4/3} < \epsilon_0,
\end{equation*}
 then there exists a unique local-in-time mild solution to \eqref{def:PKS} with initial data $\mu_0$. 
\end{itemize}
\end{theorem}
Part (i), (to our knowledge new), shows that the condition $\sup_{t \in (0,T)}t^{1/4}\norm{u(t)}_{4/3} < \infty$ is equivalent to the $L^\infty$ hypercontractivity estimate $u(t) \lesssim t^{-1}$ (the proof shows that all such estimates are equivalent).
Accordingly, standard parabolic theory implies that all mild solutions are smooth and strictly positive after $t > 0$ until (potentially) critical mass concentration.  
Part (ii) is due to Biler \cite{Biler95} and combined with \eqref{ineq:Heatpp} shows that given a measure with a sufficiently small \emph{atomic} part, one can construct a unique mild solution local in time. 
Part (i) will play a role in the proof of the main results of the paper, although (ii) will not.

We now state our main results. 
For PKS we prove the following existence and uniqueness theorem:  
\begin{theorem} \label{thm:PKSUnique}
Let $\mu \in \mathcal{M}_+(\Real^2)$ with $\max_{x \in \Real^2}\mu(\set{x}) < 8\pi$. Then there exists a unique, local-in-time mild solution $u(t)$ to \eqref{def:PKS} with initial data $\mu$.  
\end{theorem}

As discussed above, our approach also yields the Lipschitz continuity of the solution maps for \eqref{def:PKS} and \eqref{def:NSE}. 
Even for NSE, this is an improvement of the existing result of log-Lipschitz continuity, due to Gallagher and Gallay \cite{GallagherGallay05}. 

\begin{theorem} \label{thm:Lipschitz}
The solution maps of both NSE and PKS in two dimensions are locally Lipschitz continuous with respect to the total variation norm. 
That is, for all $\mu^1,\mu^2 \in \mathcal{M}(\Real^2)$ (in the case of PKS, we assume additionally $\mu^i \in\mathcal{M}_+(\Real^2)$ and $\max \mu^i(\set{x}) < 8\pi$) with associated mild solutions $w^1(t),w^2(t)$, there exists some constant $C_L = C_L(\mu^1) > 0$ and $T = T(\mu^1) > 0$ such that for all $\delta > 0$ sufficiently small, if 
\begin{equation*}
\norm{\mu^1 - \mu^2}_{\mathcal{M}} < \delta,
\end{equation*}
then
\begin{equation*}
\sup_{t \in (0,T)}\left(\norm{w^1(t) - w^2(t)}_{L^1} + t^{1/4}\norm{w^1(t) - w^2(t)}_{4/3}\right) \leq C_L\delta. 
\end{equation*}
Here, $\norm{\cdot}_{\mathcal{M}}$ denotes the total variation norm on finite Borel measures. 
\end{theorem}
Let us briefly discuss the energy structure of \eqref{def:PKS}, which is important for characterizing the self-similar spreading solutions and for analyzing the global behavior, the former being crucially important for our work. 
Formally, the Patlak-Keller-Segel model \eqref{def:PKS} is a gradient flow in the $L^2$ Wasserstein metric for the \emph{free energy} (see \cite{BlanchetCalvezCarrillo08,BlanchetCarlenCarrillo10}), 
\begin{equation}
\F(u) = \int u(x) \log u(x) dx + \frac{1}{4\pi} \int \int u(x)u(y) \log\abs{x-y} dx dy.  
\end{equation}
In particular, if the initial data has finite free energy, then for reasonable notions of weak solution we have the \emph{energy dissipation inequality}, 
\begin{equation*}
\F(u(t)) + \int_0^t \int u(s) \abs{ \grad \log u(s) - \grad c(s)}^2 dx ds \leq \F(u(0)),  
\end{equation*}
for all $t \geq 0$ until blow-up time. 
Using the sharp logarithmic Hardy-Littlewood-Sobolev inequality (see e.g. \cite{CarlenLoss92}) this implies global existence 
of any weak solution which has finite initial free energy provided that the total mass is less than $8\pi$ \cite{Dolbeault04,BlanchetEJDE06}. 
The energy dissipation inequality is actually stronger in similarity variables:
\begin{equation}
\xi = \frac{x}{\sqrt{t}}, \;\;\; \tau = \log t
\end{equation}
and $w(\tau,\xi) := t u(x,t)$. In these variables, \eqref{def:PKS} becomes the following, 
\begin{equation} \label{def:resPKS}
\left\{
\begin{array}{l}
  w_\tau + \grad \cdot (w \grad c) = \Delta w + \frac{1}{2}\grad \cdot (\xi w) \\
  - \Delta c = w,
\end{array}
\right.
\end{equation}
which is formally a gradient flow for the self-similar free energy 
\begin{equation}
\mathcal{G}(w) = \int w(\xi) \log w(\xi) d\xi + \frac{1}{2}\int w(\xi)\abs{\xi}^2 d\xi + \frac{1}{4\pi} \int \int w(\xi)w(\zeta) \log\abs{\xi-\zeta} d\xi d\zeta.  \label{def:Gssfree}
\end{equation}
As the second moment is now part of the energy, uniform control on the entropy $\int w(\xi) \log w(\xi) d\xi$ from below can be obtained and the sharp logarithmic Hardy-Littlewood-Sobolev can then be used to show that any solution to \eqref{def:resPKS} with finite self-similar free energy and mass strictly less than $8\pi$ is uniformly bounded in time. 
In physical variables this is the optimal decay estimate $u(t,x) \lesssim t^{-1}$ as $t \rightarrow \infty$. 

The free energy $\mathcal{G}$ is important to characterize the self-similar solutions of \eqref{def:PKS}. 
Biler et. al. show in \cite{Biler06} that for all $\alpha \in (0,8\pi)$ there exists a unique, radially symmetric self-similar solution with mass $\alpha$, denoted here in self-similar variables by $G_\alpha(\xi)$ (existence had been previously established in \cite{BilerAccretion95,NaitoSuzuki04}). 
These solutions will play the role that the Oseen vortices play in \cite{GallagherGallay05} as the approximation for the solution near the large atomic pieces of the initial data.  
The following proposition collects the important properties of the self-similar solutions, which show that in many ways they are qualitatively similar to the Gaussian Oseen vortices of the NSE.  
While these results are trivial for NSE, they are more difficult for PKS, due to the fully nonlinear nature of the self-similar solutions $G_{\alpha}$.
 Parts (i-iii) are not new, but we sketch some aspects of the proof in Appendix \S\ref{apx:PropSelfSim} for the readers' convenience, as they are not all located in one place in the literature. 
Parts (iv) and (v) seem to be new and are both crucial in deducing the Lipschitz dependence in Theorem \ref{thm:Lipschitz}.
Part (iv) is also necessary to show that many constants and linear estimates are uniform as $\alpha \searrow 0$, 
which is necessary to prove Theorem \ref{thm:PKSUnique}.
We should point out that the result of (v) depends on the variant of the spectral study \cite{CamposDolbeault12} in Appendix \S\ref{apx:SpectralGap}. 
For any $f \in L^1$, $f^\star$ denotes the Riesz symmetric decreasing re-arrangement (see \cite{LiebLoss} for more information on this symmetrization technique).    
In what follows we denote the polynomial weighted $L^2$ space, 
\begin{equation*}
L^p(m) := \set{f \in L^p : \jap{\xi}^{m}f(\xi) \in L^p}, 
\end{equation*}
with the convention $\jap{\xi} := (1 + \abs{\xi}^2)^{1/2}$. We also define $L_0^p(m) = \set{f \in L^p(m): \int f dx = 0}$. 
Note that for $m > 2$, $L^2(m) \hookrightarrow L^1$. 
For any $m > 2$ we have the following, which will be useful later 
\begin{align}
t^{1/4}\norm{\frac{1}{t}f(\log t, \frac{\cdot -z}{\sqrt{t}})}_{4/3} = \norm{f(\log t,\cdot)}_{4/3} \lesssim \norm{f(\log t,\cdot)}_{L^2(m)}. \label{ineq:L2mInject}
\end{align}
Now we may state Proposition \ref{prop:Galpha}.
\begin{proposition}[Properties of the self-similar solutions] \label{prop:Galpha}
Let $\alpha \in (0,8\pi)$.
\begin{itemize} 
\item[(i)] There exists a stationary solution to \eqref{def:resPKS}, denoted $G_\alpha$, which is smooth, strictly positive, satisfies $G_\alpha = G^\star_\alpha$, $\norm{G_\alpha}_1 = \alpha$ and denoting $c_\alpha = -\Delta^{-1}G_\alpha$ we have
\begin{subequations} 
\begin{align} 
G_\alpha(\xi) & \sim \frac{\alpha}{\int e^{c_\alpha(\zeta) - \abs{\zeta}^2/4} d\zeta} \abs{\xi}^{-\frac{\alpha}{2\pi}}e^{-\abs{\xi}^2/4} \textup{ as } \xi \rightarrow \infty, \label{eq:Galpha_limit} \\ 
\grad G_\alpha(\xi) & \sim -\frac{\alpha}{2\int e^{c_\alpha(\zeta) - \abs{\zeta}^2/4} d\zeta}\left(\frac{\alpha \xi}{\pi\abs{\xi}^2} + \xi\right) \abs{\xi}^{-\frac{\alpha}{2\pi}}e^{-\abs{\xi}^2/4} \textup{ as } \xi \rightarrow \infty \label{eq:GradGalpha_limit} 
\end{align} 
\end{subequations} 
in the sense of asymptotic expansion. 
Moreover, for all $p \in [1,\infty]$ and $m \geq 0$ we have 
\begin{align} 
\norm{G_\alpha}_{L^p(m)} + \norm{\grad G_\alpha}_{L^p(m)}  \lesssim_{m,p,\alpha} 1. \label{ineq:GalphaBoundsLarge}
\end{align} 
\item[(ii)] $G_\alpha(\xi)$ is the unique stationary solution of \eqref{def:resPKS} with finite self-similar energy \eqref{def:Gssfree}. Moreover, $G_\alpha$ is the unique minimizer of the self-similar free energy.
\item[(iii)] In physical variables, $\frac{1}{t}G_\alpha\left(\frac{x}{\sqrt{t}}\right)$ is the unique mild solution with finite self-similar free energy with initial data $\alpha \delta$, where $\delta$ denotes the Dirac delta mass.  
\item[(iv)] For $\alpha$ sufficiently small, the following estimate holds for all $p \in [1,\infty]$ and $m > 4$,  
\begin{align} 
\norm{G_\alpha}_{L^p(m)} + \norm{\grad G_\alpha}_{L^p(m)}  \lesssim_{m,p} \alpha. \label{ineq:GalphaBounds} 
\end{align} 
\item[(v)] For all $K < 8\pi$ and for $\alpha, \beta \leq K$, the estimate $\norm{G_\alpha - G_\beta}_{L^p} +
   \norm{G_\alpha - G_\beta}_{L^2(m)} \lesssim_{p,m,K} \abs{\alpha - \beta}$ holds for all $1 \leq p 
   \leq \infty$ and $m > 4$.
\end{itemize}
\end{proposition}
\begin{remark}
The proof of Proposition \ref{prop:Galpha} (iv) primarily shows that for $\alpha$ sufficiently small, $\norm{G_\alpha - \alpha G}_{L^2(m)} + \norm{G_\alpha - \alpha G}_{p} \lesssim \alpha^2$ where $G(\xi) = (4\pi)^{-1/2}e^{-\abs{\xi}^2/4}$ is the standard Gaussian. 
\end{remark}

Due to the a priori estimates and the uniqueness of $G_\alpha$, a compactness argument shows that if $\int u(t) dx = \alpha$, then $\lim_{t \rightarrow \infty} \norm{u(t,x) - \frac{1}{t}G_\alpha(\frac{x}{\sqrt{t}})}_p = 0 $ for $p \in [1,\infty]$. 
These results are naturally analogous to the well-known results for the heat equation and for the 2D Navier-Stokes equations \cite{GallayWayne05}. 
The spectral gap-type inequality deduced by J. Campos and J. Dolbeault in \cite{CamposDolbeault12} can also be used to deduce an exponential estimate on the rate of convergence; see also e.g. \cite{GallayWayne05,BlanchetDEF10}.

\begin{remark}
An obvious question arises about whether or not Theorems \ref{thm:PKSUnique} and \ref{thm:Lipschitz} can be extended to more general models than PKS and NSE. 
If the nonlocal velocity law is a linear combination of the Biot-Savart law for NSE and the chemotactic gradient law for PKS then this generalization should be more or less straightforward since the $G_\alpha$ will still be the self-similar solutions. 
However, if the velocity law is no longer homogeneous, for example if the equation for the chemo-attractant is replaced by $-\Delta c + c = u$ or $-\grad \cdot(a(x)\grad c) = u$, then there are no longer exact self-similar solutions. 
If $G_\alpha$ are still good short-time approximations for the evolution of atomic initial data, as should be the case in the examples just mentioned, then the stated results of Theorems \ref{thm:PKSUnique} and \ref{thm:Lipschitz} can likely be proved with similar arguments after some additional approximations.
Such cases should also include models in which the chemo-attractant and/or the density $u(t,x)$ is subjected to an 
external drift, provided the drift is sufficiently regular.
If $G_\alpha$ no longer provide a good short-time approximation to atomic initial data, more substantial changes would have to be made. 
\end{remark}

\subsubsection*{Notation and Conventions}
We denote the $L^p(dx)$ norms by $\norm{u}_p := \norm{u}_{L^p}$. 
If a measure other than Lebesgue measure is used to define the norm, this is denoted by $L^p(d\mu)$ (note that this is different than the definition of the polynomial weighted space $L^2(m)$).  

To avoid clutter in computations, function arguments (time and space) will be omitted whenever they are obvious from context.  
In formulas we will sometimes use the notation $C(p,k,M,..)$ to denote a generic constant, which may be different from line to line or even term to term in the same computation. 
Moreover, to further reduce clutter in formulas, we make very frequent use of the notation $f\lesssim_{p,k,...} g$ 
to denote $f\leq C(p,k,..)g$. 
We will generally suppress the dependencies which are not relevant for the estimate at hand and simply write $f \lesssim g$. 
In most cases, universal constants from functional inequalities and parameters which are not important for the discussion are omitted.
We will also usually suppress dependence from uniform estimates which have already been established, although we often alert the reader to the estimate being used.

Self-similar solutions of mass $\alpha$ are denoted $G_{\alpha}$, and when the mass is given by $\alpha_i$, we will often shorten this to $G_i$. Similarly, the velocity field associated with the self similar solutions are denoted $v^{G_\alpha}$ or $v^{G_i}$. 

\section{Preliminaries}
The following proposition collects the basic properties of the nonlocal velocity law of \eqref{def:PKS}, which are essentially analogous to the properties of the Biot-Savart law for NSE. 

\begin{proposition}[Properties of the nonlocal velocity law] \label{prop:Vel}
Define 
\begin{equation*}
B(x) = -\frac{x}{2\pi\abs{x}^2}. 
\end{equation*}
Then
\begin{itemize}
\item[(i)] Let $\frac{1}{q} = \frac{1}{p} - \frac{1}{2}$ for some $p \in (1,2)$. Then, 
\begin{equation}
\norm{B \ast u}_{q} \lesssim \norm{u}_p. \label{ineq:VelLp}
\end{equation}
\item[(ii)] Let $p \in (1,\infty)$. Then, 
\begin{equation*}
\norm{\grad B \ast u}_{p} \lesssim \norm{u}_p. 
\end{equation*}
Moreover, $\grad \cdot B \ast u = -u$. 
\item[(iii)] If $u \in L^2(m)$ for some $m \in (0,1)$ or $u \in L^2_0(m)$ for some $m \in (1,2)$. Then for all $q \in (2,\infty)$
\begin{equation}
\norm{\jap{\xi}^{m - \frac{2}{q}}B \ast u}_{q} \lesssim \norm{u}_{L^2(m)}. \label{ineq:VelWeightedLp}
\end{equation}
\end{itemize}
\end{proposition}
  
\subsection{Outline for the proof of Theorem \ref{thm:PKSUnique}} 
For $\epsilon > 0$ chosen small later we define the decomposition of the initial data
\begin{equation*}
\mu = \sum_{i = 1}^N \alpha_i \delta_{z_i} + \mu_0, 
\end{equation*}
for $\delta_{z_i}:= \delta(z - z_i)$, $z_i \in \Real^2$ and $\alpha_i > 0$ chosen such that $\norm{\mu_0}_{pp} < \epsilon$. 
If the measure $\mu$ contains only finitely many point masses then $N$ is finite and independent of $\epsilon$ for $\epsilon$ sufficiently small. However, in general there may be infinitely many point masses and in this case it is important to note that $N$ is fixed large when $\epsilon > 0$ is fixed small. 
Define the minimal distance between any two concentrations (which is generally forced small when $\epsilon$ is chosen small): 
\begin{equation*}
d := \min( \abs{z_i - z_j}) > 0.
\end{equation*} 
The goal of this decomposition is to construct an accurate approximate solution 
and use a perturbation argument to build a true solution which is very close to the approximate one.       
Analogously to \cite{GallagherGallay05}, we construct a mild solution $u(t)$ via a decomposition of the form
\begin{equation}
u(t,x) = \tilde{w}_0(t,x) + \sum_{i = 1}^N\alpha_i\frac{1}{t}\tilde{w}_i\left(\log t, \frac{x - z_i}{\sqrt{t}}\right) + \frac{1}{t}G_{\alpha_i}\left(\frac{x - z_i}{\sqrt{t}}\right), \label{def:decompu}
\end{equation}
with the terms $\tilde{w}_0$, $\tilde{w}_i$ defined below. However, our definition of $\tilde{w}_i$ is different than in \cite{GallagherGallay05}. 

In what follows we will explain the decomposition formally, assuming that we have a well-behaved mild solution already. 
In reality, we will construct this solution using the decomposition. 
The term $\tilde{w}_0$ is defined as the solution associated with the (approximately) non-atomic portion of the initial data, which formally satisfies the initial value problem  
\begin{equation} \label{def:w0Original}
\left\{
\begin{array}{l}
  \partial_t \tilde{w}_0 + \grad \cdot (\tilde{w}_0 v) = \Delta \tilde{w}_0 \\ 
  \tilde{w}_0(0) = \mu_0, 
\end{array}
\right. 
\end{equation}
where still $v = B \ast u$ is given by the nonlocal velocity law associated with the full solution. 
Since $\mu_0$ has a small atomic part, \eqref{ineq:Heatpp} suggests that for short time $t^{1/4}\norm{\tilde{w}_0(t)}_{4/3}$ will be small. Of course, it will take some work (Proposition \ref{prop:SNProperties}) to make this convincing, as $v(t,x)$ is very singular at time zero. 
On the other hand, the part of the solution associated with the large atomic parts of the initial data is not small in any relevant sense, so further decomposition is necessary. 
Consider the solutions $w_i(t,x)$ of the advection-diffusion equation in physical variables: 
\begin{equation*}
\left\{
\begin{array}{l}
  \partial_t w_i + \grad \cdot (w_i v) = \Delta w_i \\ 
  w_i(0) = \alpha_i \delta_{z_i}. 
\end{array}
\right. 
\end{equation*}
In \cite{GallagherGallay05}, the authors consider the difference between $w_i$ and the self-similar solution of mass $\alpha_i$ centered at $z_i$. This quantity turns out to be small as $t^{-1}G_{\alpha_i}((x- z_i)t^{-1/2})$ is an accurate approximation for $w_i$ for short time, however, $w_i(t,x) - t^{-1}G_{\alpha_i}((x-z_i)t^{-1/2})$ proves difficult to correctly control in a contraction argument.
Hence, we choose a different decomposition which still satisfies 
\begin{equation}
\sum_{i = 1}^N w_i(t,x) = \sum_{i = 1}^N \frac{1}{t}G_{\alpha_i}\left(\frac{x - z_i}{\sqrt{t}}\right) + \sum_{i = 1}^N \alpha_i\frac{1}{t}\tilde{w}_i\left(\log t, \frac{x-z_i}{\sqrt{t}} \right) \label{def:decomp}
\end{equation}  
and while each $\tilde{w}_i$ will be localized around $z_i$, $\tilde{w}_i \neq w_i-G_{\alpha_i}$ (although the proof will show they are close to being equal). In particular, \eqref{def:decomp} cannot be decoupled into separate expressions for $\tilde{w}_i$. 

Applying \eqref{def:decomp} to $v = B \ast u$ implies a more precise PDE for $\tilde{w}_0$:  
\begin{equation}  \label{def:w0}
\left\{
\begin{array}{l}
  \partial_t \tilde{w}_0 + \grad \cdot \left(\tilde{w}_0 \sum_{j = 1}^N \frac{1}{\sqrt{t}}v^{G_j}\left(\frac{x - z_j}{\sqrt{t}}\right)\right) + \grad \cdot(\tilde{w}_0 \tilde{v}_0) + \grad \cdot \left(\tilde{w}_0 \sum_{j = 1}^N\alpha_j \frac{1}{\sqrt{t}}\tilde{v}_j\left(\log t,\frac{x-z_j}{\sqrt{t}} \right)\right) = \Delta \tilde{w}_0 \\ 
  \tilde{w}_0(0) = \mu_0,
\end{array}
\right.  
\end{equation}
where $\tilde{v}_0 = B \ast \tilde{w}_0$ and $\tilde{v}_j = B \ast \tilde{w}_j$. 
For future convenience define 
\begin{equation*}
v_j(t,x) := \alpha_j\frac{1}{\sqrt{t}}\tilde{v}_j\left(\log t, \frac{x-z_j}{\sqrt{t}} \right). 
\end{equation*}

We now turn to the definition of $\tilde{w}_j$ for $j \geq 1$, which is somewhat more technical. 
For notational clarity define
\begin{align}
v^{W_i}  (\tau, \xi) & := \sum_{j = 1, j \neq i}^N \alpha_j \tilde{v}_j(\tau,\xi - (z_j - z_i)e^{-\tau/2}), \nonumber \\ 
v^{g_i} (\tau, \xi)  & := \sum_{j = 1, j \neq i}^N v^{G_j}(\xi - (z_j - z_i)e^{-\tau/2}), \label{def:vgi}
\end{align}
the velocity fields induced by the perturbations $j \neq i$ in the coordinate system of $\tilde{w}_i$ and the velocity fields induced by the self-similar solutions of the $j\neq i$ concentrations written in the coordinate system of $\tilde{w}_i$. 
Let $\phi(x)$ be a smooth, non-negative, radially symmetric, non-increasing function such that $\phi(x) = 1$ for $\abs{x} < d/2$ and $\phi(x) = 0$ for $\abs{x} > 3d/4$. 
We will define $\tilde{w}_i$ to be a solution of the following: 
\begin{equation} \label{eq:wi}
%\left\{
\begin{array}{l} 
 \partial_\tau \tilde{w}_i + \grad \cdot (\tilde{w}_i v^{G_i}) + \grad \cdot (G_i \tilde{v}_i) \\
  \hspace{.5cm} + \grad \cdot \Big[\tilde{w}_i\sum_{j \neq i} \big(1-\phi(\xi e^{\tau/2} + z_i - z_j)\big)v^{G_j} \big(\xi - (z_j - z_i)e^{-\tau/2} \big)\Big] \\ 
  \hspace{.5cm} + \grad \cdot (\sum_{i \neq j}\frac{\alpha_j}{\alpha_i}\tilde{w}_j(\xi - (z_j - z_i)e^{-\tau/2})\phi(\xi e^{\tau/2})v^{G_i}) \\
  \hspace{.5cm} + \grad \cdot (\frac{1}{\alpha_i}G_i v^{g_i}) 
  \hspace{.5cm} + \grad \cdot (\frac{1}{\alpha_i}G_i v^{W_i}) 
  \hspace{.5cm} + \grad \cdot (\frac{1}{\alpha_i}G_ie^{\tau/2}\tilde{v}_0(e^\tau,\xi e^{\tau/2} + z_i))\\ 
  \hspace{.5cm}+ \grad \cdot (\tilde{w}_i\tilde{v}_i + \tilde{w}_i v^{W_i}) 
  \hspace{.5cm} + \grad \cdot (\tilde{w}_i e^{\tau/2}\tilde{v}_0(e^\tau, \xi e^{\tau/2} + z_i)) \\
  = \Delta \tilde{w}_i + \frac{1}{2}\grad \cdot(\xi \tilde{w}_i), \\
  \lim_{\tau \rightarrow -\infty}\tilde{w}_i(\tau) = 0.  
\end{array}
%\right.
\end{equation}
It is in the second and third line where our definition differs from \cite{GallagherGallay05}. Our definition more naturally treats the dangerous terms when making estimates, but destroys the advection-diffusion structure of \eqref{eq:wi} and the coupling makes it trickier to prove that all mild solutions can be decomposed in this manner (proved below in Proposition \ref{prop:EquivSolutions}). 
%\begin{equation} \label{eq:wi}
%\left\{
%\begin{array}{l}
% \partial_\tau \tilde{w}_i + \grad \cdot (\tilde{w}_i v^{G_i}) + \grad \cdot (G_i \tilde{v}_i)
%  + \grad \cdot (\tilde{w}_i\sum_{j \neq i} (1-\phi(\xi e^{\tau/2} + z_i - z_j))v^{G_j}(\xi - (z_j - z_i)e^{-\tau/2}))  \\
%  + \grad \cdot (\sum_{i \neq j}\frac{\alpha_j}{\alpha_i}\tilde{w}_j(\xi - (z_j - z_i)e^{-\tau/2})\phi(\xi e^{\tau/2})v^{G_i})
%  + \grad \cdot (\frac{1}{\alpha_i}G_i \sum_{j \neq i} v^{G_j}(\xi - (z_j - z_i)e^{-\tau/2}))  \\
%  + \grad \cdot (\frac{1}{\alpha_i}G_i \sum_{j \neq i, j \geq 1}\alpha_j \tilde{v}_j(\xi - (z_j - z_i)e^{-\tau/2}))
%  + \grad \cdot (\frac{1}{\alpha_i}G_ie^{\tau/2}\tilde{v}_0(\xi e^{\tau/2} + z_i,e^{\tau})) \\
%  + \grad \cdot (\tilde{w}_i\sum_{j = 1}^N \alpha_j \tilde{v}_j(\xi - (z_j - z_i)e^{-\tau/2}))
%  + \grad \cdot (\tilde{w}_i e^{\tau/2}\tilde{v}_0(\xi e^{\tau/2} + z_i,e^{\tau}))
%  = \Delta \tilde{w}_i + \frac{1}{2}\grad \cdot(\xi \tilde{w}_i), \\
%  \lim_{\tau \rightarrow -\infty}\tilde{w}_i(\tau) = 0.
%\end{array}
%\right.
%\end{equation}
We re-write the equations for the perturbations $\tilde{w}_0,\tilde{w}_i$ as the corresponding Duhamel integral equations. 
Given some $\nu \in \mathcal{M}_+(\Real^2)$, define $f(t) = S_N(t,s)\nu$ to be the mild solution to the following singular PDE 
\begin{align}
\partial_t f + \grad \cdot \left( f \sum_{j = 1}^N \frac{1}{\sqrt{t}}v^{G_j} \left(\frac{x - z_j}{\sqrt{t}} \right) \right) & = \Delta f \label{eq:SNDefinition} \\ 
f(s) & = \nu. 
\end{align}
We prove that mild solutions to \eqref{eq:SNDefinition} are well-defined and collect the important properties in Proposition \ref{prop:SNProperties} below. 
Hence we may re-write \eqref{def:w0} as the formally equivalent Duhamel integral equation
\begin{equation}
\tilde{w}_0(t) = S_N(t,0)\mu_0 - \int_0^t S_N(t,s) \left[\grad \cdot(\tilde{w}_0(s) \tilde{v}_0(s)) + \grad \cdot (\tilde{w}_0(s) \sum_{j = 1}^Nv_j(s)) \right] ds. \label{def:w0Duhamel}
\end{equation}
We now turn to the perturbations $\tilde{w}_i$. 
Define the following linear operator, which is the linearization of the transport term around the self-similar solution,    
\begin{align*}
\Lambda_\alpha f & := \grad \cdot (G_\alpha v) + \grad \cdot(f v^{G_\alpha}), \\ 
v & = B \ast f,      
\end{align*}
and define the Fokker-Planck operator 
\begin{equation}
Lf := \Delta f + \frac{1}{2}\grad \cdot (\xi f). 
\end{equation}
Denote by $\mathcal{T}_\alpha(\tau) := e^{\tau(L - \Lambda_\alpha)}$ the linear propagator for the PDE
\begin{equation*}
\partial_\tau f = Lf - \Lambda_\alpha f.   
\end{equation*}
The important properties of $\mathcal{T}_\alpha$ are collected in Proposition \ref{prop:SpecT} below. 
We may now write \eqref{eq:wi} as the formally equivalent Duhamel integral equation
\begin{align}
\tilde{w}_i(\tau) & = -\int_{-\infty}^\tau \mathcal{T}_{\alpha_i}(\tau - \tau^\prime)(\grad \cdot (\tilde{w}_i\sum_{j \neq i} (1-\phi(\xi e^{\tau^\prime/2} + z_i - z_j))v^{G_j}(\xi - (z_j - z_i)e^{-\tau^\prime/2})) ) d\tau^\prime \nonumber \\ 
   & \hspace{.5cm} -\int_{-\infty}^\tau \mathcal{T}_{\alpha_i}(\tau - \tau^\prime)(\grad \cdot (\sum_{i \neq j}\frac{\alpha_j}{\alpha_i}\tilde{w}_j(\xi - (z_j - z_i)e^{-\tau^\prime/2})\phi(\xi e^{\tau^\prime/2})v^{G_i})) d\tau^\prime \nonumber  \\
  & \hspace{.5cm} -\int_{-\infty}^\tau \mathcal{T}_{\alpha_i}(\tau-\tau^\prime)(\grad \cdot (\frac{1}{\alpha_i}G_i v^{g_i}))d\tau^\prime  \nonumber
%   & \hspace{.5cm} 
-\int_{-\infty}^\tau \mathcal{T}_{\alpha_i}(\tau-\tau^\prime)(\grad \cdot (\frac{1}{\alpha_i}G_i v^{W_i})) d\tau^\prime \nonumber  \\
  & \hspace{.5cm} -\int_{-\infty}^\tau \mathcal{T}_{\alpha_i}(\tau-\tau^\prime)(\grad \cdot (\frac{1}{\alpha_i}G_ie^{\tau^\prime/2}\tilde{v}_0(e^{\tau^\prime},\xi e^{\tau^\prime/2} + z_i))) d\tau^\prime \nonumber \\
  & \hspace{.5cm} -\int_{-\infty}^\tau \mathcal{T}_{\alpha_i}(\tau-\tau^\prime)( \grad \cdot (\alpha_i\tilde{w}_i\tilde{v}_i + \tilde{w}_i v^{W_i}))d\tau^\prime \nonumber \\ 
  & \hspace{.5cm} -\int_{-\infty}^\tau \mathcal{T}_{\alpha_i}(\tau-\tau^\prime)(\grad \cdot (\tilde{w}_i e^{\tau^\prime/2}\tilde{v}_0(e^{\tau^\prime},\xi e^{\tau^\prime/2} + z_i))) d\tau^\prime.  \label{def:wiDuhamel}
\end{align}  
The primary effort of proving Theorem \ref{thm:PKSUnique} goes into showing that the system of integral equation \eqref{def:w0Duhamel},\eqref{def:wiDuhamel} has a unique solution in the relevant spaces, which is done using a contraction mapping argument.  
The perturbations $\set{\tilde{w}_i}_{i = 0}^N$ are normed with $M[\tilde{w}](t)$ defined as follows, which differs from the norm used in \cite{GallagherGallay05} by the presence of the constant $K_0$ to be chosen later.
Let
\begin{equation*}
M_0[\tilde{w}](t) = M_0[\tilde{w}_0](t) = \sup_{0 < s < t} s^{1/4}\norm{\tilde{w}_0(s)}_{4/3},  
\end{equation*}
and for $1 \leq i \leq N$, 
\begin{equation*}
M_i[\tilde{w}](t) = M_i[\tilde{w}_i](t) = \sup_{-\infty < \tau < \log(t)}\norm{\tilde{w}_i(\tau)}_{L^2(m)},
\end{equation*}
then define
\begin{equation*}
M[\tilde{w}](t) = \max\left( K_0 M_0[\tilde{w}_0](t), \max_{1\leq i \leq N} M_i[\tilde{w}_i](t)\right)
\end{equation*}
for some large constant $K_0 \geq 1$ to be chosen later. We use $K_0$ to enforce more control over $\tilde{w}_0$ than the other perturbations, which is very important for dealing with the potentially disruptive effect of $\tilde{w}_0$ on $\tilde{w}_i$, $i \geq 1$.

By construction, the unique solution to the system \eqref{def:w0Duhamel},\eqref{def:wiDuhamel} can be re-constituted into a mild solution $u(t)$ of \eqref{def:PKS} via \eqref{def:decompu}.
However, it is not a priori clear that every mild solution can be represented as a solution to the system. 
This nontrivial fact is stated in the following proposition. The proof mainly depends on a compactness argument and that $G_\alpha$ are the unique self-similar solutions, analogous to Proposition 4.5 in \cite{GallagherGallay05}.
However, unlike \cite{GallagherGallay05}, an additional step is required to construct $\tilde{w}_j$ which satisfy \eqref{def:wiDuhamel} since the cross-terms in \eqref{def:wiDuhamel} couple all the $\tilde{w}_j$ in a more subtle manner than in \cite{GallagherGallay05}. 
\begin{proposition}[Equivalence of formulations] \label{prop:EquivSolutions}
Suppose $u(t)$ is a mild solution of \eqref{def:PKS}. Then $u(t)$ can be decomposed as in \eqref{def:decompu} with $\tilde{w}_0$,$\tilde{w}_i$ satisfying the integral equations \eqref{def:w0Duhamel},\eqref{def:wiDuhamel}.  
\end{proposition}
By Proposition \ref{prop:EquivSolutions}, any mild solution must correspond to the unique solution of the system \eqref{def:w0Duhamel},\eqref{def:wiDuhamel}, which would complete the proof of Theorem \ref{thm:PKSUnique}. 

\begin{remark}
Alternatively, in order to complete the proof of Theorem \ref{thm:PKSUnique} we could fall back to a proof which more closely matches Gallagher and Gallay and use their decomposition to show that any second mild solution must agree with the one constructed with the integral equations \eqref{def:w0Duhamel},\eqref{def:wiDuhamel}. 
This should work, however, we prefer to give a more self-contained proof by going through Proposition \ref{prop:EquivSolutions}.  
\end{remark}

The rest of the paper is organized as follows. In Section \S\ref{sec:Requis} we state the main linear estimates which 
are required for the proof of both Theorem \ref{thm:PKSUnique} and \ref{thm:Lipschitz}. 
In Section \S\ref{sec:PKSUnique} we prove Theorem \ref{thm:PKSUnique} in several steps. 
In Section \S\ref{sec:Lipschitz} we establish Theorem \ref{thm:Lipschitz}, the proof of which is closely related 
to the main steps of Theorem \ref{thm:PKSUnique}. 
In Appendix \S\ref{apx:LinearEstimates} we establish the linear estimates stated in \S\ref{sec:Requis} and in Appendix \S\ref{apx:PropSelfSim} we sketch the proof of Proposition \ref{prop:Galpha}.
Finally in Appendix \S\ref{apx:SpectralGap} we include an independent proof of a version of the spectral gap estimate due to Campos and Dolbeault. 

\subsection{Requisite Linear Estimates} \label{sec:Requis}
We briefly recall some known properties of the linear propagator of the Fokker-Planck equation $S(\tau) := e^{\tau L}$ in $L^2(m)$, studied in \cite{GallayWayne02}. The following proposition can also be found in \cite{GallagherGallay05}. 

\begin{proposition}[Properties of $S(\tau)$] \label{prop:Stau}
Fix $m > 1$. Then, 
\begin{itemize}
\item[(i)] $S(\tau)$ defines a strongly continuous semigroup on $L^2(m)$ and for all $w \in L^2(m)$, 
\begin{equation}
\norm{S(\tau)w}_{L^2(m)} \lesssim \norm{w}_{L^2(m)}, \;\;\; \norm{\grad S(\tau) w}_{L^2(m)} \lesssim \frac{1}{a(\tau)^{1/2}}\norm{w}_{L^2(
m)},  \label{ineq:SgradHyper}
\end{equation}
for all $\tau > 0$ and where $a(\tau) = 1- e^{-\tau}$. 
\item[(ii)] If $m > 2$ and $w \in L_0^2(m)$, then
\begin{equation}
\norm{S(\tau)w}_{L^2(m)} \lesssim e^{-\tau/2}\norm{w}_{L^2(m)}, \;\;\; \forall \tau>0.
\end{equation}
\item[(iii)] If $q \in [1,2]$ then for all $w \in L^q(m)$ and $\tau > 0$, 
\begin{align}
\norm{S(\tau)w}_{L^2(m)} \lesssim \frac{1}{a(\tau)^{\frac{1}{q} - \frac{1}{2}}}\norm{w}_{L^q(m)} \\ 
\norm{\grad S(\tau) w}_{L^2(m)} \lesssim \frac{1}{a(\tau)^{\frac{1}{q}}}\norm{w}_{L^q(m)}. \label{ineq:gradSDecay}
\end{align}
\end{itemize}
Note that
\begin{equation}
\grad S(\tau) = e^{\tau/2}S(\tau) \grad. \label{eq:SgradCommute} 
\end{equation}
\end{proposition}

The following proposition is of crucial importance. It is the analogue to Proposition 4.6 in \cite{GallagherGallay05} and Proposition 4.12 in \cite{GallayWayne05} but the proof deviates in several key places due to the different nature of the linear operator.  
Indeed, Recall that analyzing the spectral properties of the linearization around $G_\alpha$ is more difficult for PKS than for NSE, as the operator $\Lambda_\alpha$ is not skew-symmetric in any relevant Hilbert space.  As mentioned previously, the key tool used is a variant of the spectral gap-type results recently obtained by J. Campos and J. Dolbeault \cite{CamposDolbeault12}. 
This  spectral gap must be adapted to the polynomial weighted spaces $L^2(m)$, a procedure analogous to what is done in \cite{GallayWayne05}, which we carry out in Appendix \S\ref{apx:LinearEstimates}.

\begin{proposition} \label{prop:SpecT} 
Fix $\alpha \in (0,8\pi)$ and $m > 2$.
\begin{itemize}
\item[(i)] $\mathcal{T}_\alpha(\tau)$ defines a strongly continuous semigroup which is bounded on $L^2(m)$ and satisfies 
\begin{align}
\norm{\mathcal{T}_\alpha(\tau) f}_{L^2(m)} \lesssim_\alpha \norm{f}_{L^2(m)}, \label{ineq:TBounded}\\ 
\norm{\grad\mathcal{T}_\alpha(\tau)f}_{L^2(m)} \lesssim_\alpha a(\tau)^{-1/2}\norm{f}_{L^2(m)}, \label{ineq:TGradientHyper}
\end{align}
where $a(\tau) := 1 - e^{-\tau}$.
\item[(ii)] For some $\nu = \nu(\alpha) \in (0,1/2)$ which depends on $\alpha$ and
for all $f \in L^2_0(m)$,
\begin{equation}
\norm{\mathcal{T}_\alpha(\tau) f}_{L^2(m)} \lesssim_\alpha e^{-\nu \tau}\norm{f}_{L^2(m)}. \label{ineq:SpecGapT}
\end{equation}
\item[(iii)] If $q \in (1,2]$ then $\mathcal{T}_\alpha(\tau) \grad$ is a bounded operator from $L^q(m)$ to $L^2_0(m)$ and there exists a $\nu \in (0,1/2)$ (the same $\nu$ as in (ii)) such that, 
\begin{equation}
\norm{\mathcal{T}_{\alpha}(\tau)\grad f}_{L^2(m)} \lesssim_{\alpha} \frac{e^{-\nu\tau}}{a(\tau)^{1/q}}\norm{f}_{L^q(m)}. \label{ineq:TGradDecay}
\end{equation}  
\end{itemize}
Though $\nu$ and all of the implicit constants depend on $\alpha$, as $\alpha \searrow 0$, $\nu \approx 1/2$ and the constants are uniformly bounded by Proposition \ref{prop:Galpha} (iv) as $\mathcal{T}_\alpha(\tau)$ can be treated as a perturbation of $S(\tau)$ (see Remark \ref{rmk:controlledSpec} in Appendix \S\ref{apx:SpectralGap}). 
\end{proposition} 

The following is the analogue of Proposition 4.3 in \cite{GallagherGallay05}, but the proof must deviate from the corresponding one for NSE in a non-trivial manner, as the underlying linear operator no longer has as nice structure (carried out in Appendix \S\ref{sec:SNPrp}). 
The first step is a general lemma (Lemma \ref{lem:GenLinear}) which exhibits at least one well-behaved mild solution to a class of singular advection-diffusion equations including \eqref{eq:SNDefinition}.
Next, uniqueness is proved for the $N=1$ case by a compactness/rigidity argument that requires the monotonicity of $v^{G}$ to localize potential pathologies in the solution as well as a decay estimate of Carlen and Loss \cite{CarlenLoss92}.
The extension to $N > 1$ is straightforward following a similar argument of Gallagher and Gallay \cite{GallagherGallay05}.  
The proof of (iii) below uses spectral properties of linear Fokker-Planck equations with general confining potentials. 
\begin{proposition} \label{prop:SNProperties}
There exists some $t_0$ sufficiently small such that 
\begin{itemize}
\item[(i)] $S_N(t,s)$ defines a weak$^\star$ continuous linear propagator (see Remark \ref{def:weakstarcont} below) on $\mathcal{M}(\Real^2)$ and for all $p \in [1,\infty]$ and $\nu \in \mathcal{M}(\Real^2)$ we have 
\begin{equation}
\norm{S_N(t,s)\nu}_{L^p} \lesssim \frac{1}{(t-s)^{1-1/p}}\norm{\nu}_{\mathcal{M}}, \;\;\; 0 \leq s < t < s + t_0. \label{ineq:SNHypercon} 
\end{equation}
\item[(ii)] For all $p \in (1,\infty]$ and $\nu \in \mathcal{M}_+(\Real^2)$ (uniformly in $s$), 
\begin{equation}
\limsup_{t \searrow s} (t-s)^{1-1/p}\norm{S_N(t,s)\nu}_p \lesssim \norm{\nu}_{pp}. \label{ineq:SNpp}
\end{equation}
\item[(iii)] There exists some $\lambda_0 \in (0,1/2)$ independent of $\epsilon$ such that the following holds: for all $p \in [1,\infty]$, for all $\gamma \in (0,\lambda_0)$ and for all $w \in L^1(\Real^2)$, 
\begin{equation} 
\norm{S_N(t,s)\grad \cdot w}_{p} \lesssim_{p,\gamma} \frac{1}{(t-s)^{3/2 - 1/p}}\left( \frac{t}{s} \right)^{\gamma + 1/2 - \lambda_0}\norm{w}_{1}, \;\;\; 0 < s < t < s + t_0. \label{ineq:SNgrad}
\end{equation}
\end{itemize}
All of the implicit constants above are independent of $t$,$s$,$\epsilon,N$ and $d$. Moreover, it will suffice to choose $t_0$ such that  
\begin{equation*}
t_0 \leq d^2\min(1,K),
\end{equation*}
for some $K$ which is independent of $\epsilon,N$ and $d$. 
\end{proposition}  
\begin{remark} \label{def:weakstarcont}
By `weak$^\star$ continuous linear propagator' we mean that $S_N(t,s)$ is linear and if $\mu_n \subset \mathcal{M}(\Real^2)$ satisfy $\sup_{n}\norm{\mu_n}_{\mathcal{M}(\Real^2)} < \infty$ and $\mu_n \rightharpoonup^\star \mu$ and $t_n,s_n \subset (0,T)$, $0 \leq s_n \leq t_n$, with $s_n \rightarrow \bar{s} \in [0,T)$ and $t_n \rightarrow \bar{t} \in [0,T)$ we have 
\begin{align*}
S_N(t_n,s_n)\mu_n & \rightharpoonup^\star S_N(\bar t,\bar{s})\mu. 
\end{align*}
\end{remark} 

\section{Existence and Uniqueness: Proof of Theorem \ref{thm:PKSUnique}} \label{sec:PKSUnique}
We proceed in several steps. First we prove the contraction mapping argument which establishes the existence and uniqueness
of solutions to the integral equations \eqref{def:w0Duhamel} and \eqref{def:wiDuhamel} which is the core of the proof.
Next we establish Theorem \ref{thm:Basics} (i) which is necessary to establish Proposition \ref{prop:EquivSolutions}, which is carried out last. Finally we briefly summarize the full argument at the end of the section.  

\subsection{Contraction Mapping} \label{sec:ContractMap}
We will construct our solution to \eqref{def:w0Duhamel} and \eqref{def:wiDuhamel} in the following ball (for $\epsilon > 0$ and $T>0$ to be chosen small later): define $\rho := \set{\rho_i}_{i = 0}^N$, 
\begin{equation}
B_{\epsilon,T} = \set{\rho(t):M[\rho_0 - S_N(t,0)\mu_0,\rho_{i \geq 1}](T) < \epsilon}. \label{def:BepT}
\end{equation} 
Note that the ball is centered around $S_N(t,0)\mu_0$, although given Proposition \ref{prop:SNProperties} (ii), this is a minor detail.
For any $\set{\tilde{w}_{i}}_{i = 0}^N$, the corresponding $u(t,x)$ constructed by \eqref{def:decompu} will be in $\chi_T$ (but is not small due to the presence of the large atomic pieces). 
It might be useful to bear in mind that the approximate solution we are perturbing around is, 
\begin{equation*}
u_{app}(t,x) = S_N(t,0)\mu_0 + \sum_{i = 1}^N\frac{1}{t}G_{\alpha_i}\left(\frac{x - z_i}{\sqrt{t}}\right), 
\end{equation*}
although we will not  make explicit note of this in the remainder of the paper.
 
Let $\tilde{w} = F[\rho]$ be the nonlinear solution map which takes $\rho$ to $\tilde{w} := \set{\tilde{w}_i}_{i = 0}^N$ defined by the following procedure. 
In what follows, define 
\begin{align*}
\tilde{v}_0^\rho(t,x) & : = B \ast \rho_0, \\ 
\tilde{v}_j^\rho(\tau,\xi) & := B \ast \rho_j, \\ 
v_j^\rho(t,x) & : = \frac{1}{\sqrt{t}}\tilde{v}_j^\rho\left(\log t, \frac{x- z_j}{\sqrt{t}} \right), \\ 
v^{R_i}(\tau,\xi) & := \sum_{j \neq i} \alpha_j \tilde{v}^\rho_j(\tau,\xi - (z_j - z_i)e^{-\tau/2}). 
\end{align*}
Given $\rho$, we define $\tilde{w}_0$ by
\begin{equation*}
\tilde{w}_0(t) = S_N(t,0)\mu_0 - \int_0^tS_N(t,s)\left[\grad \cdot (\rho_0(s) \tilde{v}^{\rho}_0(s)) + \grad\cdot(\rho_0(s)\sum_{i = 1}^N \alpha_i v^{\rho}_i(s)) \right] ds. 
\end{equation*}
Similarly we define $\tilde{w}_i$ by 
\begin{align*}
\tilde{w}_i(\tau) & = -\int_{-\infty}^\tau \mathcal{T}_{\alpha_i}(\tau - \tau^\prime)(\grad \cdot (\rho_i\sum_{j \neq i} (1-\phi(\xi e^{\tau^\prime/2} + z_i - z_j))v^{G_j}(\xi - (z_j - z_i)e^{-\tau^\prime/2})) ) d\tau^\prime  \\
   &\hspace{.5cm} -\int_{-\infty}^\tau \mathcal{T}_{\alpha_i}(\tau - \tau^\prime)(\grad \cdot (\sum_{i \neq j}\frac{\alpha_j}{\alpha_i}\rho_j(\xi - (z_j - z_i)e^{-\tau^\prime/2})\phi(\xi e^{\tau^\prime/2})v^{G_i})) d\tau^\prime  \\
  & \hspace{.5cm} -\int_{-\infty}^\tau \mathcal{T}_{\alpha_i}(\tau-\tau^\prime)(\grad \cdot (\frac{1}{\alpha_i}G_i v^{g_i}))d\tau^\prime
%   & \hspace{.5cm} 
-\int_{-\infty}^\tau \mathcal{T}_{\alpha_i}(\tau-\tau^\prime)(\grad \cdot (\frac{1}{\alpha_i}G_i v^{R_i})) d\tau^\prime  \\
  & \hspace{.5cm} -\int_{-\infty}^\tau \mathcal{T}_{\alpha_i}(\tau-\tau^\prime)(\grad \cdot (\frac{1}{\alpha_i}G_ie^{\tau^\prime/2}\tilde{v}^\rho_0(e^{\tau^\prime},\xi e^{\tau^\prime/2} + z_i))) d\tau^\prime \\
  & \hspace{.5cm} -\int_{-\infty}^\tau \mathcal{T}_{\alpha_i}(\tau-\tau^\prime)( \grad \cdot (\rho_i \alpha_i \tilde{v}^\rho_i + \rho_i v^{R_i}))d\tau^\prime \\ 
  & \hspace{.5cm} -\int_{-\infty}^\tau \mathcal{T}_{\alpha_i}(\tau-\tau^\prime)(\grad \cdot (\rho_i e^{\tau^\prime/2}\tilde{v}^\rho_0(e^{\tau^\prime},\xi e^{\tau/2} + z_i))) d\tau^\prime.
\end{align*}
By definition, a fixed point of $F$ corresponds to a solution of the system \eqref{def:w0Duhamel},\eqref{def:wiDuhamel}, which in turn corresponds to a mild solution of \eqref{def:PKS}.  

The application of the contraction mapping theorem requires that $F$ maps $B_{\epsilon,T}$ to itself and that $F$ defines a locally Lipschitz mapping, which we prove in separate propositions. Note that linear terms in the definition of $F$ can be treated essentially the same in the two propositions, and it is in the linear terms that our arguments significantly differ from \cite{GallagherGallay05}. 
\begin{proposition}\label{prop:Fmapping}
For $\epsilon > 0$ sufficiently small, $T >0$ can be chosen sufficiently small such that $F:B_{\epsilon,T} \rightarrow B_{\epsilon,T}$. 
\end{proposition}
\begin{proof}
First we control $\tilde{w}_0$ to show that for $T$ chosen sufficiently small (depending on $\epsilon$)
\begin{align}
K_0M_0[\tilde{w}_0 - S_N(t,0)\mu_0](t) \leq K_1K_0 M_0[\rho_0](t) M[\rho](t) \leq K_1 \left(M[\rho](t)\right)^2, \label{ineq:w0PropFmapping}
\end{align}
for some constant $K_1 > 0$ which is independent of $\epsilon$. 
This is essentially the analogue of Proposition 5.1 in \cite{GallagherGallay05}, and we approach it in a similar way. 
Indeed,
\begin{align}
K_0t^{1/4}\norm{\tilde{w}_0(t) - S_N(t,0)\mu_0}_{4/3} &\leq K_0 t^{1/4}\int_0^t \norm{S_N(t,s)\grad \cdot (\tilde{\rho}_0(s)\tilde{v}^\rho_0(s))}_{4/3}ds \nonumber \\ 
& \hspace{4pt} + K_0t^{1/4}\int_0^t \norm{S_N(t,s)\grad \cdot (\tilde{\rho}_0(s)\sum_{j=1}^N\alpha_jv_j^\rho(s))}_{4/3}ds. \label{ineq:FmapSN}
\end{align}
To control the first term: by Proposition \ref{prop:SNProperties}, H\"older's inequality and \eqref{ineq:VelLp} we have for $t$ sufficiently small (so that Proposition \ref{prop:SNProperties} holds), 
\begin{align*}
\norm{S_N(t,s)\grad \cdot (\rho_0(s)\tilde{v}^\rho_0(s))}_{4/3} & \lesssim \frac{1}{(t-s)^{3/4}}\left(\frac{t}{s}\right)^{\gamma + 1/2 - \lambda_0} \norm{\rho_0(s) \tilde{v}^\rho_0(s)}_1 \\ 
& \leq \frac{1}{(t-s)^{3/4}}\left(\frac{t}{s}\right)^{\gamma + 1/2 - \lambda_0} \norm{\rho_0(s)}_{4/3} \norm{\tilde{v}^\rho_0(s)}_4 \\
& \lesssim \frac{1}{(t-s)^{3/4}}\left(\frac{t}{s}\right)^{\gamma + 1/2 - \lambda_0} \norm{\rho_0(s)}_{4/3}^2.  
\end{align*}
Hence, 
\begin{align*}
K_0 t^{1/4}\int_0^t \norm{S_N(t,s)\grad \cdot (\tilde{\rho}_0(s)\tilde{v}^\rho_0(s))}_{4/3}ds & \lesssim K_0t^{1/4}\int_0^t\frac{1}{(t-s)^{3/4}}\left(\frac{t}{s}\right)^{\gamma + 1/2 - \lambda_0} \norm{\rho_0(s)}_{4/3}^2 ds \\ 
& \lesssim K_0t^{1/4} \left(M_0[\rho_0](t)\right)^2\int_0^t\frac{1}{(t-s)^{3/4}}\left(\frac{t}{s}\right)^{\gamma + 1/2 - \lambda_0} \frac{1}{s^{1/2}} ds \\ 
& \lesssim K_0\left(M_0[\rho_0](t)\right)^2. 
\end{align*}
To control the second term in \eqref{ineq:FmapSN} we proceed similarly (again for $t$ sufficiently small), 
\begin{align*}
K_0t^{1/4}\int_0^t \norm{S_N(t,s)\grad \cdot (\rho_0(s)\sum_{j=1}^N\alpha_jv_j^\rho(s))}_{4/3}ds \lesssim & \\ & \hspace{-4cm}  \sum_{j = 1}^N \alpha_j K_0t^{1/4}\int_0^t\frac{1}{(t-s)^{3/4}}\left(\frac{t}{s}\right)^{\gamma + 1/2 - \lambda_0} \norm{\rho_0(s)}_{4/3} s^{-1/4}\norm{\rho_i(\log s)}_{4/3} ds.  
\end{align*}
Recalling \eqref{ineq:L2mInject}, 
\begin{align*}
K_0t^{1/4}\int_0^t \norm{S_N(t,s)\grad \cdot (\tilde{\rho}_0(s)\sum_{j=1}^N\alpha_jv_j^\rho(s))}_{4/3}ds \lesssim & \\ & \hspace{-4cm} \sum_{j = 1}^N \alpha_j K_0 t^{1/4} M_0[\rho_0](t) M_j[\rho_j](t) \int_0^t\frac{1}{(t-s)^{3/4}}\left(\frac{t}{s}\right)^{\gamma + 1/2 - \lambda_0} \frac{1}{s^{1/2}} ds. 
\end{align*}
Putting the estimates together proves \eqref{ineq:w0PropFmapping}. 

The significantly more delicate challenge is controlling $\tilde{w}_i$ for $1 \leq i \leq N$, which is stated in the following lemma. 
\begin{lemma} \label{lem:Fmap}
For all $\epsilon > 0$ the following estimate holds for $i \in \set{1,...,N}$: 
\begin{align*}
M_i[\tilde{w}](t) & \leq \delta_2(t) + \eta(t)M[\rho](t) + K_2M_0[\rho_0](t) + K_3M_i[\rho](t)M[\rho](t), \\ 
& \leq \delta_2(t) + \eta(t)M[\rho](t) + \frac{K_2}{K_0}M[\rho](t) + K_3M_i[\rho](t)M[\rho](t),
\end{align*}
where $\delta_2(t)$ and $\eta(t)$ depend on $\epsilon$ but go to zero as $t \searrow 0$ while $K_2$ and $K_3$ are independent of $\epsilon$. 
\end{lemma} 
\begin{proof}
Write \eqref{def:wiDuhamel} as
\begin{equation*}
\tilde{w}_i(\tau) = -\sum_{k = 1}^6 F_{i,k}(\tau),
\end{equation*}
where 
\begin{align*}
F_{i,1}(\tau) & = \sum_{j \neq i}\int_{-\infty}^\tau \mathcal{T}_{\alpha_i}(\tau - \tau^\prime) \grad \cdot (\frac{1}{\alpha_i}G_i v^{G_j}(\xi - (z_j-z_i)e^{-\tau^\prime/2})) d\tau^\prime \\ 
F_{i,2}(\tau) & = \sum_{j \neq i}\int_{-\infty}^\tau \mathcal{T}_{\alpha_i}(\tau - \tau^\prime) \grad \cdot (\frac{1}{\alpha_i}G_i \alpha_j \tilde{v}^\rho_j(\xi - (z_j-z_i)e^{-\tau^\prime/2})) d\tau^\prime \\ 
F_{i,3}(\tau) & = \int_{-\infty}^\tau \mathcal{T}_{\alpha_i}(\tau - \tau^\prime) \grad \cdot (\frac{1}{\alpha_i}G_i e^{\tau^\prime/2}\tilde{v}_0^\rho(e^{\tau^\prime},\xi e^{\tau^\prime/2} + z_i)) d\tau^\prime \\
F_{i,4}(\tau) = F_{i,4}^{(1)}(\tau) + F_{i,4}^{(2)}(\tau) & = \int_{-\infty}^\tau \mathcal{T}_{\alpha_i}(\tau - \tau^\prime) \grad \cdot (\rho_i\sum_{j \neq i} (1-\phi(\xi e^{\tau^\prime/2} + z_i - z_j))v^{G_j}(\xi - (z_j - z_i)e^{-\tau^\prime/2}))d\tau^\prime  \\
& \hspace{.5cm} +  \int_{-\infty}^\tau \mathcal{T}_{\alpha_i}(\tau - \tau^\prime)\grad \cdot (\sum_{j \neq i}\frac{\alpha_j}{\alpha_i}\rho_j(\xi + (z_i - z_j)e^{-\tau^\prime/2})\phi(\xi e^{\tau^\prime/2})v^{G_i}(\xi)) d\tau^\prime \\ 
F_{i,5}(\tau) & = \sum_{j = 1}^N \int_{-\infty}^\tau \mathcal{T}_{\alpha_i}(\tau - \tau^\prime) \grad \cdot (\rho_i(\tau^\prime)\alpha_j\tilde{v}_j^{\rho}(\xi - (z_j-z_i)e^{-\tau^\prime/2})) d\tau^\prime \\ 
F_{i,6}(\tau) & = \int_{-\infty}^\tau \mathcal{T}_{\alpha_i}(\tau - \tau^\prime) \grad \cdot (\rho_i e^{\tau^\prime/2}\tilde{v}_0^\rho(e^{\tau^\prime},\xi e^{\tau^\prime/2} + z_i)) d\tau^\prime. 
\end{align*}
The first is controlled analogously to the corresponding term in \cite{GallagherGallay05}. Using \eqref{ineq:TGradDecay},  
\begin{align*}
\norm{F_{i,1}(\tau)}_{L^2(m)} & \lesssim \sum_{j \neq i}\int_{-\infty}^\tau \frac{e^{-\nu(\tau - \tau^\prime)}}{a(\tau - \tau^\prime)^{1/2}}\norm{v^{G_j}(\xi - (z_j - z_i)e^{-\tau^\prime/2})\jap{\xi}^{m}\frac{G_i}{\alpha_i} }_{2} d\tau^\prime \\
 & \lesssim \sum_{j \neq i}\sup_{\xi}\left(\jap{\xi}^{-1}v^{G_j}(\xi - (z_j - z_i)e^{-\tau/2})\right) \int_{-\infty}^\tau \frac{e^{-\nu(\tau - \tau^\prime)}}{a(\tau - \tau^\prime)^{1/2}}\norm{\frac{G_i}{\alpha_i}}_{L^2(m+1)} d\tau^\prime \\
& \lesssim e^{\tau/2}, 
\end{align*}
with an implicit constant which depends on $d$, and hence $\epsilon$. %independent of $\epsilon$ and $N$. 
The last inequality follows from Proposition \ref{prop:Galpha} (iv) and \eqref{ineq:GradcalphaAsymptotic} which imply $\jap{\xi}^{-1}v^{G_j}(\xi - (z_j - z_i)e^{-\tau/2}) \leq C(\alpha_j)d^{-1}e^{\tau/2}$ (Proposition \ref{prop:Galpha} (iv) shows that the constants are uniform as $\alpha_j \searrow 0$). 
To control the next term we also proceed analogously to \cite{GallagherGallay05}, 
\begin{align*}
\norm{F_{i,2}}_{L^2(m)} \lesssim \sum_{j \neq i}\alpha_j\int_{-\infty}^\tau \frac{e^{-\nu(\tau - \tau^\prime)}}{a(\tau - \tau^\prime)^{1/2}}\norm{\tilde{v}^\rho_j(\xi - (z_j - z_i)e^{-\tau^\prime/2})\jap{\xi}^m \frac{G_{i}}{\alpha_i}}_{2} d\tau^\prime. 
\end{align*}
Continuing with $q \in (2/\nu,\infty)$ and $\gamma \in (2/q,1)$,  
\begin{align*}
\norm{\tilde{v}^\rho_j(\xi - (z_j - z_i)e^{-\tau^\prime/2})\jap{\xi}^m \frac{G_i}{\alpha_i}}_{L^2} & \leq \norm{\jap{\xi}^{\gamma - \frac{2}{q}}\tilde{v}^\rho_j}_q \norm{\jap{\xi - (z_j - z_i)e^{-\tau^\prime/2}}^{\frac{2}{q} - \gamma}\jap{\xi}^m \frac{G_i}{\alpha_i}}_{\frac{2q}{q-2}}.  
\end{align*}
The first factor can be controlled via the weighted estimate on the nonlocal velocity law \eqref{ineq:VelWeightedLp} which implies (since $m > \gamma$):
\begin{equation*}
\norm{\jap{\xi}^{\gamma - \frac{2}{q}}\tilde{v}^\rho_j}_q \lesssim \norm{\rho_j}_{L^2(\gamma)} \leq \norm{\rho_j}_{L^2(m)}.
\end{equation*}
The second factor is controlled by the localization of $G_\alpha$ given in Proposition \ref{prop:Galpha}. In particular, since $d > 0$ we have, 
\begin{equation*}
\norm{\jap{\xi - (z_j - z_i)e^{-\tau^\prime/2}}^{2/q - \gamma}\jap{\xi}^m \frac{G_i}{\alpha_i}}_{\frac{2q}{q-2}} \lesssim_d e^{\tau(\gamma/2 - 1/q)}.
\end{equation*}
Therefore, 
\begin{align*}
\norm{F_{i,2}}_{L^2(m)} & \lesssim_{d} \sum_{j \neq i} \alpha_j \int_{-\infty}^\tau\frac{e^{-\nu(\tau - \tau^\prime)}}{a(\tau - \tau^\prime)^{1/2}}\norm{\rho_j(\tau^\prime)}_{L^2(m)} e^{\tau^\prime \left(\gamma/2 - 1/q\right)} d\tau^\prime \\ 
& \lesssim_d M[\rho](e^\tau) e^{\tau(\gamma/2 - 1/q)}. 
\end{align*}
Now we confront the next term using the $L^p$ estimate on the nonlocal velocity law \eqref{ineq:VelLp}, \eqref{ineq:TGradDecay} and Proposition \ref{prop:Galpha},  
\begin{align*}
\norm{F_{i,3}}_{L^2(m)} & \lesssim \int_{-\infty}^\tau \frac{e^{-\nu(\tau - \tau^\prime)}}{a(\tau - \tau^\prime)^{1/2}}\norm{e^{\tau^\prime/2}\tilde{v}^\rho_0(e^{\tau^\prime},\xi e^{\tau^\prime/2} + z_i)}_{L^4}\norm{\jap{\xi}^{m}\frac{G_i}{\alpha_i}}_{L^4} d\tau^\prime \\ 
& \lesssim  M_0[\rho](e^\tau)\int_{-\infty}^\tau \frac{e^{-\nu(\tau - \tau^\prime)}}{a(\tau - \tau^\prime)^{1/2}} d\tau^{\prime} \\ 
& \leq K_2M_0[\rho](e^\tau).  
\end{align*}
Note that $K_2$ is independent of $d$ and $\epsilon$, since by Proposition \ref{prop:Galpha} (iv), for $\alpha \rightarrow 0$, $\norm{\jap{\xi}^m\frac{G_{\alpha}}{\alpha}}_{L^4}$ remains bounded.  
Now we turn to $F_{i,4}$, which is an important difference between the work here and \cite{GallagherGallay05}. Dealing with the first term:
\begin{align*}
\norm{F_{i,4}^{(1)}}_{L^2(m)} \lesssim \int_{-\infty}^\tau \frac{e^{-\nu(\tau - \tau^\prime)}}{a(\tau - \tau^{\prime})^{1/2}} \norm{\rho_i\sum_{j \neq i} (1-\phi(\xi e^{\tau^\prime/2} + z_i - z_j))v^{G_j}(\xi + (z_i - z_j)e^{-\tau^\prime/2}))}_{L^2(m)} d\tau^\prime.
\end{align*} 
By the definition of the cut off $\phi$, the integrand of the $L^2$ norm is only non-zero if $\abs{\xi e^{\tau^\prime/2} + z_i - z_j} > d/2$, 
which of course is equivalent to $\abs{\xi + (z_i - z_j)e^{-\tau^\prime/2}} > e^{-\tau^\prime/2}d/2$. Since $v^{G_j}$ decays like $(\alpha_j/2\pi)\abs{\xi}^{-1}$, 
\begin{align*}
\norm{F_{i,4}^{(1)}}_{L^2(m)} & \lesssim \int_{-\infty}^\tau \frac{e^{-\nu(\tau - \tau\prime)}}{a(\tau - \tau^{\prime})^{1/2}} \norm{\rho_i\sum_{j \neq i} (1-\phi(\xi e^{\tau^\prime/2} + z_i - z_j))v^{G_j}(\xi + (z_i - z_j)e^{-\tau^\prime/2}))}_{L^2(m)} d\tau^\prime \\ 
& \lesssim \int_{-\infty}^\tau \frac{e^{-\nu(\tau - \tau^\prime)}}{a(\tau - \tau^{\prime})^{1/2}}\left(\frac{e^{\tau^\prime/2}}{d}\sum_{i \neq j} \alpha_j\right) \norm{\rho_i(\tau^\prime)}_{L^2(m)} d\tau^\prime \\ 
& \lesssim \frac{1}{d}e^{\tau/2}M[\rho_i](e^{\tau}). 
\end{align*} 
Now we consider the second term, which is concentrated around $z_i$, 
\begin{align*}
\norm{F_{i,4}^{(2)}}_{L^2(m)} \lesssim \int_{-\infty}^\tau \frac{e^{-\nu(\tau - \tau^\prime)}}{a(\tau - \tau^{\prime})^{1/2}}\norm{\sum_{i \neq j}\frac{\alpha_j}{\alpha_i}\rho_j(\xi - (z_j - z_i)e^{-\tau^\prime/2})\phi(\xi e^{\tau^\prime/2})v^{G_i}(\xi)}_{L^2(m)} d\tau^\prime.
\end{align*}
Since $\abs{\frac{1}{\alpha}\jap{\xi}^{1}v^{G_\alpha}(\xi)}$ is bounded uniformly (for $\alpha < 8\pi$ obviously) by Proposition \ref{prop:Galpha} and \eqref{ineq:GradcalphaAsymptotic}, 
\begin{align*}
\frac{1}{\alpha_i}\norm{\rho_j(\xi - (z_j - z_i)e^{-\tau^\prime/2})\phi(\xi e^{\tau^\prime/2})v^{G_i}(\xi) \jap{\xi}^m }_2 \lesssim \norm{\rho_j(\xi - (z_j - z_i)e^{-\tau^\prime/2})\phi(\xi e^{\tau^\prime/2})\jap{\xi}^{m-1} }_2. 
\end{align*}
Due to the cut-off, the integrand is only supported where $\abs{\xi} \leq 3e^{-\tau^\prime/2}d/4$ and this implies that
\begin{align*}
\abs{\xi - (z_j - z_i)e^{-\tau^\prime/2}} > e^{-\tau^\prime/2}d/4.
\end{align*}
Hence, 
\begin{align*}
\frac{\jap{\xi}^{m-1}}{\jap{\xi -(z_j - z_i)e^{-\tau^\prime/2}}^m} 
& \lesssim \frac{\jap{d e^{-\tau^\prime/2}}^{m-1}}{\jap{d e^{-\tau^\prime/2}}^{m}} \\ & \lesssim \frac{1}{\jap{d e^{-\tau^\prime/2}}}.
\end{align*}
Therefore we can translate the coordinate system in the $L^2(m)$ norm and we get from the above: 
\begin{align*}
\norm{\rho_j(\xi - (z_j - z_i)e^{-\tau^\prime/2})\phi(\xi e^{\tau^\prime/2})\jap{\xi}^{m-1} }_2 & = \norm{\rho_j(\xi - (z_j - z_i)e^{-\tau^\prime/2})\phi(\xi e^{\tau^\prime/2})\jap{\xi}^{m-1}\frac{\jap{\xi -(z_j - z_i)e^{-\tau^\prime/2}}^m}{\jap{\xi -(z_j - z_i)e^{-\tau^\prime/2}}^m}}_2 \\  & \lesssim \frac{1}{\jap{de^{-\tau^\prime/2}}}\norm{\rho_j}_{L^2(m)}.  
\end{align*}
Putting the previous two estimates together we ultimately have 
\begin{align*}
\norm{F_{i,4}^{(1)}(\tau) + F_{i,4}^{(2)}(\tau)}_{L^2(m)} \lesssim \left(\frac{1}{d}e^{\tau/2} + \frac{1}{\jap{de^{-\tau/2}}}\right)M[\rho](e^{\tau}),  
\end{align*}
where the implicit constant does not depend on $\epsilon$ or $d$.
Estimating the first nonlinear term, for $1 < p < 2$ using \eqref{ineq:TGradDecay} 
\begin{align*}
\norm{F_{i,5}(\tau)}_{L^2(m)} \lesssim \sum_{j = 1}^N \alpha_j \int_{-\infty}^\tau\frac{e^{-\nu(\tau - \tau^\prime)}}{a(\tau - \tau^\prime)^{1/p}} \norm{\tilde{v}^\rho_j(\xi - (z_j - z_i)e^{-\tau^\prime/2},\tau^\prime)\rho_i(\tau^\prime)}_{L^p(m)} d\tau^\prime.  
\end{align*}
The norm can be estimated with H\"older's inequality and the $L^p$ estimate for the nonlocal velocity law \eqref{ineq:VelLp} (and that $L^2(m)$ injects into $L^p$, $p < 2$ since $m > 2$): 
\begin{align*}
\norm{\tilde{v}^\rho_j(\xi - (z_j - z_i)e^{-\tau^\prime/2},\tau^\prime)\rho_i(\tau^\prime)}_{L^p(m)} \leq \norm{\jap{\xi}^m\rho_i}_{2}\norm{\tilde{v}_j^\rho}_{2p/(2-p)} \lesssim \norm{\rho_i}_{L^2(m)}\norm{\rho_j}_p \lesssim \norm{\rho_i}_{L^2(m)}\norm{\rho_j}_{L^2(m)}. 
\end{align*}
Hence, 
\begin{equation*}
\norm{F_{i,5}}_{L^2(m)} \lesssim M_i[\rho](e^\tau)M[\rho](e^\tau), 
\end{equation*}
with an implicit constant which is independent of $\epsilon$. 
The last nonlinear term we deal with similarly, 
\begin{align*}
\norm{F_{i,6}(\tau)}_{L^2(m)} & \lesssim \int_{-\infty}^\tau\frac{e^{-\nu(\tau - \tau^\prime)}}{a(\tau - \tau^\prime)^{3/4}}\norm{\rho_i e^{\tau^\prime/2}\tilde{v}_0^\rho(e^{\tau^\prime},\xi e^{\tau^\prime/2} + z_i)}_{L^{4/3}(m)} d\tau^\prime \\ 
& \lesssim \int_{-\infty}^\tau\frac{e^{-\nu(\tau - \tau^\prime)}}{a(\tau - \tau^\prime)^{3/4}}\norm{\rho_i(\tau^\prime)}_{L^2(m)} \norm{e^{\tau^\prime/2}\tilde{v}_0^\rho(e^{\tau^\prime},\xi e^{\tau^\prime/2} + z_i)}_{4} d\tau^\prime.  
\end{align*}
By definition, since $t^\prime = e^{\tau^\prime}$, and using the $L^p$ estimate for the nonlocal velocity law \eqref{ineq:VelLp} we have
\begin{align*}
\norm{e^{\tau^\prime/2}\tilde{v}_0^\rho(e^{\tau^\prime},\xi e^{\tau^\prime/2} + z_i)}_{4} & = (t^\prime)^{1/4}\norm{\tilde{v}_0^\rho(t^\prime)}_{4} \\ 
& \lesssim (t^\prime)^{1/4}\norm{\rho_0(t^\prime)}_{4/3} = M_0[\rho_0](t^\prime). 
\end{align*}
Therefore, 
\begin{equation*}
\norm{F_{i,6}(\tau)}_{L^2(m)} \lesssim M_0[\rho](e^\tau)M[\rho_i](e^\tau),  
\end{equation*}
with an implicit constant which is independent of $\epsilon$. 
This completes the estimate of $\tilde{w}_i$. 
\end{proof}
In order to prove that $F:B_{\epsilon,T} \rightarrow B_{\epsilon,T}$ we need to first fix $\epsilon$ small, then $T$ small and $K_0$ large. 
The reader is advised to note that fixing $\epsilon$ small in turn generally fixes $N$ large and $d$ small.
Now, if we fix $\epsilon$, then restrict $T$ such that $\delta_2(t) < \epsilon/4$, $\eta(t) < \epsilon/4$, and $K_0$ such that $K_0 > 4K_2$ then
\begin{equation*}
K_0M_0[\tilde{w} - S_N(t,0)\mu_0](t) \leq K_1K_0M_0[\rho]M[\rho] < K_1\epsilon^2,  
\end{equation*} 
and 
\begin{align*}
M[\tilde{w}_i](t) & \leq \frac{\epsilon}{4} + \frac{\epsilon}{4}M[\rho] + \frac{1}{4}M[\rho] + K_3 M_i[\rho] M[\rho] \\ 
& \leq 3\epsilon/4 +  K_3 \epsilon^2. 
\end{align*}
Finally, the result follows by choosing $\epsilon$ small. 
\end{proof}

Proving that $F$ is a contraction does not pose any significant new challenges.
The linear terms, which were the most difficult to deal with in Proposition \ref{prop:Fmapping}, are treated exactly the same. 
The only variation is in the treatment of nonlinear terms, but these do not pose a significant issue and can be dealt with as in \cite{GallagherGallay05}. Hence we only sketch the proof.
\begin{proposition}\label{prop:Fcontraction}
For $\epsilon > 0$ sufficiently small, $T$ can be chosen sufficiently small such that $F$ is a contraction on $B_{\epsilon,T}$. That is, if $\rho^1,\rho^2 \in B_{\epsilon,T}$ and $w^1 = F(\rho^1), w^2 = F(\rho^2)$ then,  
\begin{equation*}
M[\tilde{w}^1 - \tilde{w}^2](t) \leq \frac{1}{2}M[\rho^1 - \rho^2](t). 
\end{equation*}
\end{proposition}
\begin{proof} 
We first estimate $M_0[\tilde{w}^1 - \tilde{w}^2](t)$. By definition, keeping notation analogous with above,  
\begin{align}
\tilde{w}_0^1(t) - \tilde{w}_0^2(t)  & = -\int_0^t S_N(t,s)\grad \cdot \left(\tilde{v}_0^{\rho_1}(s)\rho_0^1(s) - \tilde{v}_0^{\rho_2}(s)\rho_0^2(s) \right) ds \nonumber \\ 
& \hspace{.5cm} - \sum_{j=1}^N\alpha_j\int_0^t S_N(t,s)\grad \cdot \left(v_j^{\rho_1}(s)\rho_0^1(s) - v_j^{\rho_2}(s)\rho_0^2(s) \right) ds. \label{def:w01w02Lip}
\end{align}
All the terms are dealt with essentially the same, but consider the interactions of $\rho_0$ with $\rho_i$, $i \geq 1$. 
Write 
\begin{equation}
v^{\rho_1}_j \rho_0^1 - v^{\rho_2}_j \rho_0^2 = (v_j^{\rho_1} - v_j^{\rho_2})\rho_0^1 + v_j^{\rho_2}(\rho_0^1 - \rho_0^2). \label{eq:vrhoi}
\end{equation}
Estimating the first set of terms, using \eqref{ineq:SNgrad}, H\"older's inequality, \eqref{ineq:VelLp} and \eqref{ineq:L2mInject},  
\begin{align*}
t^{1/4}\norm{\int_0^t S_N(t,s)\grad \cdot \left(\sum_{j = 1}^N\alpha_j(v_j^{\rho_1}(s) - v_j^{\rho_2}(s))\right)\rho_0^1(s) ds}_{4/3} &\\ & \hspace{-4cm} \lesssim  \sum_{j = 1}^N t^{1/4}\alpha_j\int_0^t \frac{1}{(t-s)^{3/4}}\left(\frac{t}{s}\right)^{\gamma+1/2 - \lambda_0}\norm{v_{j}^{\rho_1}(s) - v_j^{\rho_2}(s)}_4\norm{\rho^1_0(s)}_{4/3} ds \\ 
& \hspace{-4cm}\lesssim M[\rho_1 - \rho_2](t) M_0[\rho_0^1]t^{1/4}\int_0^t \frac{1}{(t-s)^{3/4}}\left(\frac{t}{s}\right)^{\gamma + 1/2 - \lambda_0}\frac{1}{s^{1/2}} ds. 
\end{align*}
Similarly, 
\begin{align*}
t^{1/4}\norm{\int_0^t S_N(t,s)\grad \cdot \left(\sum_{j = 1}^N\alpha_jv_j^{\rho_2}(s) \left(\rho_0^1(s) - \rho_0^2(s)\right) \right)ds}_{4/3} &\\ & \hspace{-4cm} \lesssim  \sum_{j = 1}^N t^{1/4} \alpha_j\int_0^t \frac{1}{(t-s)^{3/4}}\left(\frac{t}{s}\right)^{\gamma + 1/2 - \lambda_0}\norm{v_j^{\rho_2}(s)}_4\norm{\rho^1_0(s) - \rho_0^2(s)}_{4/3} ds \\ 
& \hspace{-4cm} \lesssim M[\rho_2](t) M_0[\rho_0^1 - \rho_0^2](t)t^{1/4}\int_0^t \frac{1}{(t-s)^{3/4}}\left(\frac{t}{s}\right)^{\gamma + 1/2 - \lambda_0}\frac{1}{s^{1/2}} ds. 
\end{align*}
The terms involving $\tilde{v}_0^{\rho_i}$ are treated similarly (easier in fact) so we omit the details.  
Using these estimates together with \eqref{def:w01w02Lip} implies that there exists some constant $K_5$ independent of $\epsilon$ such that for all $\epsilon >0$ sufficiently small, we can choose $T$ sufficiently small such that if $t \in (0,T)$,
\begin{align}
K_0 M_0[\tilde{w}_0^1 - \tilde{w}_0^2](t) & \leq K_5M[\rho_1 - \rho_2](t)\left(K_0M_0[\rho^1](t) + M[\rho^2](t)\right) \nonumber \\ 
& \leq K_5M[\rho_1 - \rho_2](t)\left(M[\rho^1](t) + M[\rho^2](t)\right). \label{ineq:M0Contract}  
\end{align}
Now we turn to the contraction estimate on the perturbations around the self-similar solutions. 
Similar to \cite{GallagherGallay05}, define $G_{i,k} = F_{i,k}^1 - F_{i,k}^2$, where $F^j_{i,k}$ is defined as in Proposition \ref{prop:Fmapping} corresponding to $\rho^j$. The source term satisfies $G_{i,1} = 0$, 
whereas the proof of Lemma \ref{lem:Fmap} immediately implies that the linear terms satisfy
\begin{equation}
G_{2,k} + G_{3,k} + G_{4,k} \lesssim \eta(t)M[\rho^1 - \rho^2](t) + \frac{K_2}{K_0}M[\rho^1 - \rho^2](t), \label{ineq:GlinearContract}
\end{equation}
where $\eta(t)$ and $K_2$ are the same as those in Lemma \ref{lem:Fmap}.  
Finally, the nonlinear terms $G_{4,k},G_{5,k}$ can be treated easily by combining \eqref{eq:vrhoi} with the arguments of Lemma \ref{lem:Fmap} to prove that there exists some $K_6$ independent of $\epsilon$ such that  
\begin{equation}
G_{4,k} + G_{5,k} \leq K_6\left(M[\rho^1](t) + M[\rho^2](t)\right)M[\rho^1 - \rho^2](t). \label{ineq:GnonlinContract}
\end{equation} 
Together, \eqref{ineq:M0Contract}, \eqref{ineq:GlinearContract} and \eqref{ineq:GnonlinContract}  imply Proposition \ref{prop:Fcontraction} by first choosing $\epsilon$ sufficiently small: 
\begin{equation*}
\epsilon < \frac{1}{16}\min\left(\frac{1}{K_6},\frac{1}{K_5}\right),
\end{equation*}
choose $K_0 > 8 K_2$ and then choose $T$ such that \eqref{ineq:SNgrad} holds and that $\eta(t) < \frac{1}{8}$ (the parameters of course should also be chosen such that Proposition \ref{prop:Fmapping} holds).  
\end{proof} 

\subsection{Proof of Theorem \ref{thm:Basics} \textit{(i)}}
We now prove Theorem \ref{thm:Basics} \textit{(i)}, which is an important property of mild solutions and also plays a necessary role in the proof of Proposition \ref{prop:EquivSolutions}.  
\begin{proof}
By standard theory on the continuation of classical solutions, it suffices to assume $T > 0$ is sufficiently small. We first prove that the $\norm{u(t)}_{4/3}\lesssim t^{-1/4}$ estimate can be bootstrapped up to the estimate $\norm{u(t)}_{3} \lesssim t^{-2/3}$. Then we show separately that this additional estimate implies the $L^\infty$ estimate. 
This kind of two step bootstrap approach is common when applying similar methods \cite{JagerLuckhaus92,Kowalczyk05,BlanchetEJDE06,Blanchet09,CalvezCarrillo06,Corrias04,BRB10,BedrossianIA10}. 
It turns out to be a little easier to prove the $L^3$ estimate in self-similar variables. Hence define, 
\begin{equation*}
w(\tau,\xi) = t u(t,x), \;\;\; \tau = \log t, \;\; \xi = \frac{x}{\sqrt{t}}. 
\end{equation*}
The hypercontractive estimates $\norm{u(t)}_{p} \lesssim t^{\frac{1}{p}-1}$ are all equivalent to $\norm{w(\tau)}_p \lesssim 1$, 
which is one of the reasons the self-similar variables simplify the argument. 
We use an argument which takes advantage of the a priori control on the vertical distribution of mass implied by $\norm{w(\tau)}_{4/3}\lesssim 1$. Indeed, estimates on the vertical distribution of mass have long been known to be a key controlling quantity for PKS (see e.g. \cite{JagerLuckhaus92,BlanchetEJDE06,CalvezCarrillo06}) so it is natural that such control also produces hypercontractive estimates, as already seen in \cite{BlanchetEJDE06}.
Let $k \geq 1$ be some constant which will be chosen later and define $w_k = (w-k)_+$. 
Let $\tau_0 \in (-\infty, \log(T)-1)$ be arbitrary and define for $\tau \in [\tau_0,\tau_0+1]$, $p(\tau) = 4/3 + (5/3)(\tau-\tau_0)$. 
Note that while $p$ varies with $\tau$, it lies in $p \in [4/3,3]$ and $\dot{p} = 5/3$, so for making most estimates we can treat $p$ as basically constant. 
We will show that there exists some constant $C_1$ independent of $\tau_0$ such that
\begin{equation}
\norm{w(1+\tau_0)}^3_{3} = \norm{w(1+\tau_0)}_{p(1 + \tau_0)}^{p(1 + \tau_0)} \leq C_1, \label{ineq:L3bound}
\end{equation}
which proves the desired claim. For $\tau \in (\tau_0, \tau_0 + 1)$ we now compute the following (defining $-\Delta c = w$), 
\begin{align*}
\frac{d}{d\tau}\int w_k(\tau)^{p(\tau)} d\xi & =  \dot{p}\int w_k^p \log w_k d\xi + p\int w_k^{p-1}\left(Lw - \grad \cdot (w \grad c) \right) d\xi \\ 
& = \dot{p}\int w_k^p \log w_k d\xi - \frac{4(p-1)}{p}\int\abs{\grad w_k^{p/2}}^2 d\xi \\ & \hspace{.5cm} + \frac{p}{2}\int w_k^{p-1}\grad \cdot (\xi w) d\xi - p\int w_k^{p-1}\grad\cdot (w \grad c) d\xi. 
\end{align*}
Using $w_k^lw = w_k^{l+1} + kw_k^l$ and $-\Delta c = w$, the above expands into
\begin{align*}
\frac{d}{d\tau}\int w_k(\tau)^{p(\tau)} d\xi & = \dot{p}\int w_k^p \log w_k d\xi - \frac{4(p-1)}{p}\int\abs{\grad w_k^{p/2}}^2 d\xi + (p-1)\int w_k^{p+1} d\xi \\ & \hspace{.5cm} + C(k,p)\int w_k^{p} d\xi + C(k,p)\int w_k^{p-1} d\xi, 
\end{align*}
for some constants that depend on $k$ and $p$ that we will not need the precise values of. Since $\dot{p} = 5/3 > 0$ and $\log w_k \leq w_k$, 
\begin{align*}
\frac{d}{d\tau}\int w_k(\tau)^{p(\tau)} d\xi  
& \leq -\frac{4(p-1)}{p}\int\abs{\grad w_k^{p/2}}^2 d\xi + (p+\frac{2}{3})\int w_k^{p+1} d\xi + C(k,p)\int w_k^{p} d\xi + C(k,p)\int w_k^{p-1} d\xi. 
\end{align*}
We may interpolate all of the lower order terms between $L^{p+1}$ and $L^1$ and use weighted Young's inequality to deduce 
\begin{equation*}
C(k,p)\norm{w_k}_{p}^p \leq C(k,p)\norm{w_k}_{p+1}^{\frac{p^2 - 1}{p}}\norm{w_k}_1^{\frac{1}{p}} \leq \frac{1}{6}\norm{w_k}_{p+1}^{p+1} + C(k,p)\norm{w_k}_1,  
\end{equation*} 
where the constant in the last inequality is different than the first, but we do not need to track such details. A similar inequality holds for $p-1$ and hence, defining $M = \norm{w}_1$, 
\begin{align*}
\frac{d}{d\tau}\int w_k(\tau)^{p(\tau)} d\xi & \leq -\frac{4(p-1)}{p}\int\abs{\grad w_k^{p/2}}^2 d\xi + (p+1)\int w_k^{p+1} d\xi + C(k,p)M. 
\end{align*} 
Similar to \cite{CalvezCarrillo10}, we apply the Gagliardo-Nirenberg inequality
\begin{equation*}
\norm{w_k}_{p+1}^{p+1} \leq K(p)\norm{w_k}_1\int \abs{\grad w_k^{p/2}}^2 d\xi,    
\end{equation*}
to the first term, which implies, 
\begin{align*}
\frac{d}{d\tau}\int w_k(\tau)^{p(\tau)} d\xi & \leq -\frac{4(p-1)}{pK(p)}\frac{\norm{w_k}_{p+1}^{p+1}}{\norm{w_k}_1} + (p+1)\norm{w_k}_{p+1}^{p+1} + C(k,p)M. 
\end{align*} 
Now we use uniform vertical control imposed by the estimate $\norm{w}_{4/3} \lesssim 1$, which implies, 
\begin{equation*}
\norm{w_k}_1 \leq \int_{w > k} w(\xi) d\xi \leq \frac{1}{k^{1/3}}\int \abs{w(\xi)}^{4/3} d\xi \lesssim \frac{1}{k^{1/3}}. 
\end{equation*}
Applying this to the time evolution of $\norm{w_k}_p$, we have that for for some constant which is uniformly bounded for $p$ on $\tau \in [\tau_0, \tau_0 + 1]$, 
\begin{equation*}
\frac{d}{d\tau}\int w_k(\tau)^{p(\tau)} d\xi  \leq \left(p+1-C(p)k^{1/3}\right)\norm{w_k}_{p+1}^{p+1} + C(k,p)M. 
\end{equation*}
Hence, for $k$ chosen sufficiently large we have, 
\begin{equation*}
\frac{d}{d\tau}\int w_k(\tau)^{p(\tau)} d\xi \leq -\norm{w_k}_{p+1}^{p+1} + C(k,p)M. 
\end{equation*}
Since $\norm{w_k}_{p}^{p} \leq \norm{w_k}_{p+1}^{p+1} + \norm{w_k}_1 \leq \norm{w_k}_{p+1}^{p+1} + M$ we finally have
\begin{equation*}
\frac{d}{d\tau}\int w_k(\tau)^{p(\tau)} d\xi  \leq -\norm{w_k}_{p}^{p} + C_\star,   
\end{equation*} 
where $C_\star$ is some constant which depends only on $M$, as $p \in [4/3,3]$ and $k$ has been fixed. 
Integrating implies 
\begin{equation*}
\norm{w_k(\tau_0 + 1)}_{3}^3 \leq \max\left( \norm{w_k(\tau_0)}_{4/3}^{4/3},C_\star\right) \leq \max\left( \norm{w(\tau_0)}_{4/3}^{4/3},C_\star\right). 
\end{equation*}
Using now the inequality $\norm{w}_{3} \lesssim \norm{w_k}_3 + k^{2}\norm{w}_1$ we finally get \eqref{ineq:L3bound}. 
Notice that the a priori estimates on $w(\tau)$ imply that $\norm{B \ast w(\tau)}_\infty \lesssim 1$, 
which in physical variables is equivalent to $\norm{v(t)}_\infty \lesssim t^{-1/2}$.  

To bootstrap the $L^3$ estimate to $L^\infty$ we return to the original variables, although the reader may wish to note the parallel between the following argument and the one just finished. 
Let $t_k = 2^{-k}$ and consider the dyadic intervals $[t_k,t_{k-1}]$. In the following computations it is important to keep in mind that $t_{k-1} = 2t_k = 4t_{k+1}$, and are hence all comparable.   
On each dyadic interval, 
\begin{equation*}
u(t) = e^{(t-t_k)\Delta}u(t_k) - \int_{t_k}^t e^{(t-s)\Delta}\grad \cdot (u(s) v(s))ds.
\end{equation*}
By the $L^{4/3}$ and $L^3$ bounds on $u$, \eqref{ineq:HeatLpLqEasy}, \eqref{ineq:HeatLpLq} and $\norm{v(t)}_\infty \lesssim t^{-1/2}$, 
\begin{align*}
\norm{u(t)}_\infty & \lesssim (t-t_k)^{-3/4}\norm{u(t_k)}_{4/3} +\int_{t_k}^t \norm{ e^{(t-s)\Delta}\grad \cdot (u(s) v(s))}_{\infty}ds \\ 
& \lesssim t_k^{-1/4}(t-t_k)^{-3/4} + \int_{t_k}^t (t-s)^{-5/6} \norm{u(s) v(s)}_3 ds \\ 
& \lesssim  t_k^{-1/4}(t-t_k)^{-3/4} + \int_{t_k}^t (t-s)^{-5/6} \norm{u(s)}_3 \norm{v(s)}_\infty ds \\
& \lesssim  t_k^{-1/4}(t-t_k)^{-3/4} + \int_{t_k}^t (t-s)^{-5/6}s^{-2/3}s^{-1/2} ds \\
& \lesssim  t_k^{-1/4}(t-t_k)^{-3/4} + t_k^{-7/6}(t-t_k)^{1/6} \\ 
& \lesssim  t_k^{-1/4}(t-t_k)^{-3/4} + t_k^{-1}. 
\end{align*} 
Hence this implies $\norm{u(t_{k-1})}_{\infty} \approx t_{k}^{-1} \approx t_{k-1}^{-1}$. 
Re-doing the above computation using this information to deal with the first term (and the maximum principle $\norm{e^{t\Delta}f}_\infty \lesssim \norm{f}_\infty$) we see that indeed $\norm{u(t)}_{\infty} \lesssim t^{-1}$.  
\end{proof}

\subsection{Proof of Proposition \ref{prop:EquivSolutions}}
In this section we prove the equivalence of the integral equations \eqref{def:w0Duhamel},\eqref{def:wiDuhamel} with general mild solutions.  
Let $u(t,x)$ be a mild solution of \eqref{def:PKS} with initial data $\mu$ with associated nonlocal velocity $v(t) = B \ast u(t)$.
By Theorem \ref{thm:Basics}, $u(t,x)$ necessarily satisfies the following a priori estimates for all $p \in [1,\infty]$ and $q \in (2,\infty]$, 
\begin{equation}
\norm{u(t)}_{p} \lesssim t^{1/p-1}, \;\;\; \norm{v(t)}_q = \norm{B \ast u(t)}_{q} \lesssim t^{\frac{1}{q} - \frac{1}{2}}. \label{ineq:PropEquivApriori}
\end{equation}
Suppose for any $\epsilon > 0$ we write 
\begin{equation*}
\mu = \mu_0 + \sum_{j = 1}^N \alpha_j \delta_{z_j},   
\end{equation*}
with as always $\norm{\mu_0}_{pp} < \epsilon$. 
Define $w_i$, $1 \leq i \leq N$ as non-negative mild solutions to 
\begin{equation} \label{def:wilinear}
\left\{
\begin{array}{l}
  \partial_t w_i + \grad \cdot (w_i v) = \Delta w_i \\ 
  w_i(0) = \alpha_i \delta_{z_i}, 
\end{array}
\right. 
\end{equation}
which also satisfy the following for all $\gamma > 1$ and $t \leq 1$,  
\begin{align} 
\norm{w_i(t)}_p \lesssim t^{1/p-1}, \;\;\; \int e^{\frac{\abs{x}^2}{4\gamma t}}w_i(t,x) dx \lesssim_\gamma 1.  \label{ineq:linbds}
\end{align}
Existence of such solutions is proved below by Lemma \ref{lem:GenLinear} in Appendix \S\ref{subsec:SNExistence}; the only condition which does not immediately follow from \eqref{ineq:PropEquivApriori}  and parabolic regularity is \eqref{def:vtight}. 
However, this can be derived directly from \eqref{ineq:PropEquivApriori}, the tightness implied by $u(t) \in C_w([0,T];\mathcal{M}(\Real^2))$ and the nonlocal velocity law.  
 
With these $w_i$ we may then write $\tilde w_0 = u - \sum_{i = 1}^N w_i$ which is a mild solution to \eqref{def:w0Original}
that also satisfies the a priori estimate $\norm{\tilde w_0(t)}_p \lesssim t^{1/p-1}$ (but we cannot conclude that it is non-negative). 
Note that as of yet we cannot assert this decomposition is unique (although the ensuing proof will show that it is) as we do not yet have the necessary structure for $v(t,x)$. 

For all $i \in \set{1,...,N}$, we write $w_i$ in self-similar coordinates (without re-naming), $\tau = \log t$, $\xi = \frac{x - z_i}{\sqrt{t}}$, which then satisfy
\begin{equation*}
\partial_\tau w_i + \grad \cdot (v_i w_i) + \grad \cdot (\tilde R_i w_i) = Lw_i, 
\end{equation*}
for $\tilde R_i$ given as follows: define $v_i = B \ast w_i$ then as in \cite{GallagherGallay05},
\begin{equation} 
\tilde R_i(\tau,\xi) = e^{\tau/2}\tilde{v}_0(e^\tau,\xi e^{\tau/2} + z_i) + \sum_{j = 1, j \neq i}^N v_j(\tau,\xi - (z_j - z_i)e^{-\tau/2}). \label{def:RiRemVel}
\end{equation}
The steps to proving Proposition \ref{prop:EquivSolutions} are the following: 
\begin{itemize}
\item[(i)] Show that $w_i(\tau,\xi) \rightarrow G_{\alpha_i}(\xi)$ in $L^2(m)$ for $m > 2$ as $\tau \rightarrow -\infty$.   
\item[(ii)] Use (i) to construct $\tilde{w}_i$ which satisfy \eqref{def:wiDuhamel}. 
\item[(iii)] Show that $\tilde{w}_0$ satisfies \eqref{def:w0Duhamel}. 
\end{itemize}
Once we have completed (i) and (ii), (iii) follows from the weak$^\star$ continuity of $S_N(t,s)$, indeed, once we have constructed suitable $\tilde{w}_i$, we may write 
\begin{equation*}
\tilde{w}_0(t) = S_N(t,t_0)\tilde{w}_0(t_0) - \int_{t_0}^t S_N(t,s) \left[\grad \cdot(\tilde{w}_0(s) \tilde{v}_0(s)) + \grad \cdot (\tilde{w}_0(s) \sum_{j = 1}^Nv_j(s)) \right] ds, 
\end{equation*}
and pass to the limit $t_0 \searrow 0$ using Proposition \ref{prop:SNProperties}. 

Now we concentrate on the more involved procedure of proving (i) and (ii).
Part (i) uses an energy/compactness argument analogous to the approach of Gallagher and Gallay \cite{GallagherGallay05}. 
The idea is as follows: a compactness argument shows that $w_i(\tau)$ is precompact in $L^2(m)$ as $\tau \rightarrow -\infty$ and the uniqueness properties of the self-similar solution stated in Proposition \ref{prop:Galpha} will imply that the $\alpha$-limit set can only consist of $\set{G_{\alpha_i}}$.  
The first lemma is the compactness. 
\begin{lemma} \label{lem:precompact} 
For all $i \in \set{1,...,N}$, $\set{w_i(\tau)}$ is precompact in $L^2(m)$ as $\tau \rightarrow -\infty$. 
\end{lemma}
\begin{proof}
In the rescaled variables, \eqref{ineq:linbds} implies the Gaussian localization estimate
\begin{equation}
\int w_i(\tau, \xi ) e^{\frac{\abs{\xi}^2}{4\gamma}} d\xi \lesssim_\gamma 1, \label{ineq:GaussianLocalize}
\end{equation} 
for all $\gamma > 1$. Combined with the a priori estimates \eqref{ineq:PropEquivApriori}, $w_i(\tau)$ is then uniformly bounded in $L^2(m)$ for all $m$. 
As $H^1(m+1) \hookrightarrow\hookrightarrow L^2(m)$ by the Rellich-Khondrashov embedding theorem, it suffices to prove that $\grad w_i(\tau)$ is uniformly bounded in $L^2(m)$ for all $m > 2$, 
for which we proceed similar to what is done to prove analogous statements in \cite{GallagherGallay05,GallayWayne05}, with the necessary alterations to deal with the divergence of the velocity field. 
One can show using an argument similar to \S\ref{sec:SpecTi} that the uniform bound in $L^2(m)$ implies 
that at least $\norm{\grad w_i(\tau)}_{L^2(m)}$ is locally integrable in $\tau$, hence we may proceed with an a priori estimate.
Write $w_i(\tau)$ in integral form: for some $-\infty < \tau_0 < \log(T)$, 
\begin{equation*}
w_i(\tau) = S(\tau - \tau_0)w_i(\tau_0) - \int_{\tau_0}^\tau S(\tau - \tau^\prime) \grad \cdot (v_i(\tau^\prime) w_i(\tau^\prime) + \tilde R_i(\tau^\prime)w_i(\tau^\prime)) d\tau^\prime. 
\end{equation*}
Using \eqref{ineq:SgradHyper}, for some $p \in (1,2)$ we have,  
\begin{align*}
\norm{\grad w_i(\tau)}_{L^2(m)} & \lesssim \frac{\norm{w_i(\tau_0)}_{L^2(m)}}{a(\tau - \tau_0)^{1/2}} + \int_{\tau_0}^\tau a(\tau-\tau^\prime)^{-1/p}\norm{\grad \cdot (v_i(\tau^\prime) w_i(\tau^\prime) + \tilde R_i(\tau^\prime)w_i(\tau^\prime))}_{L^p(m)} d\tau^\prime \\ 
& \lesssim \frac{\norm{w_i(\tau_0)}_{L^2(m)}}{a(\tau - \tau_0)^{1/2}} + \int_{\tau_0}^\tau a(\tau-\tau^\prime)^{-1/p}\norm{(v_i(\tau^\prime) + \tilde R_i(\tau^\prime))\cdot \grad w_i(\tau^\prime)}_{L^p(m)} d\tau^\prime \\ & \hspace{.5cm} + \int_{\tau_0}^\tau a(\tau-\tau^\prime)^{-1/p} \norm{w_i(\tau^\prime)(\grad \cdot v_i(\tau^\prime) + \grad \cdot \tilde R_i(\tau^\prime))}_{L^p(m)} d\tau^\prime. 
\end{align*}
The second term can be controlled as in \cite{GallagherGallay05}: 
\begin{align*}
\norm{(v_i(\tau^\prime) + \tilde R_i(\tau^\prime))\cdot \grad w_i(\tau^\prime)}_{L^p(m)} \leq \norm{\grad w_i(\tau^\prime)}_{L^2(m)}\norm{v_i(\tau^\prime) + \tilde R_i(\tau^\prime)}_{\frac{2p}{2-p}}. 
\end{align*}
The latter factor is bounded as follows, recalling the definition of $\tilde R_i$ \eqref{def:RiRemVel}. For $j \in \set{1,..,,N} $ using \eqref{ineq:VelLp} and H\"older's inequality, 
\begin{align*}
\norm{v_j(\tau^\prime)}_{\frac{2p}{2-p}} \lesssim \norm{w_j(\tau^\prime)}_{p} \lesssim_m \norm{w_j(\tau^\prime)}_{L^2(m)} \lesssim_m 1,  
\end{align*}
and for the approximately non-atomic part, using the a priori estimate \eqref{ineq:PropEquivApriori} 
\begin{align*}
\norm{e^{\tau^\prime/2}\tilde v_0(e^{\tau^\prime}, \xi e^{\tau^\prime/2} + z_i)}_{\frac{2p}{p-2}} & = e^{\tau^\prime(\frac{1}{2} - \frac{p-2}{2p})}\norm{\tilde v_0(e^{\tau^\prime},x)}_{\frac{2p}{p-2}} \lesssim 1.   
\end{align*}
We now turn to the second term, which involves the divergence of the velocity fields. 
Recall that for PKS, the nonlocal law is given by $v = \grad (-\Delta)^{-1}w$ and hence for all $j \in \set{1,...,N}$,
\begin{align*}
\norm{w_i(\tau^\prime)\grad \cdot v_j(\tau^\prime,\xi - (z_j-z_i)e^{-\tau^\prime/2})}_{L^p(m)} & = \norm{w_i(\tau^\prime) w_j(\tau^\prime,\xi - (z_j - z_i)e^{-\tau^\prime/2})}_{L^p(m)} \\ 
& \leq \norm{w_j(\tau^\prime)}_\infty \norm{w_i(\tau^\prime)}_{L^p(m)} \lesssim 1, 
\end{align*}
by the a priori estimates \eqref{ineq:PropEquivApriori} and \eqref{ineq:GaussianLocalize}. 
Similarly, for the approximately non-atomic term (using the $L^\infty$ estimate of \eqref{ineq:PropEquivApriori}), 
\begin{align*}
\norm{w_i(\tau^\prime) \grad \cdot e^{\tau^\prime/2}\tilde v_0(e^{\tau^\prime},e^{\tau^\prime/2}\xi + z_i)}_{L^p(m)} & = e^{\tau^\prime}\norm{w_i(\tau^\prime) \tilde w_0(e^{\tau^\prime},e^{\tau^\prime/2}\xi + z_i)}_{L^p(m)} \\
&\leq e^{\tau^\prime}\norm{\tilde w_0(e^{\tau^\prime})}_\infty \norm{w_i(\tau^\prime)}_{L^p(m)} \lesssim 1. 
\end{align*}
Putting everything together we have 
\begin{align*}
\norm{\grad w_i(\tau)}_{L^2(m)} \lesssim \frac{1}{a(\tau - \tau_0)^{1/2}} + \int_{\tau_0}^\tau a(\tau - \tau^\prime)^{-1/p}\left( \norm{\grad w_i(\tau^\prime)}_{L^2(m)} + K \right) d\tau^\prime,  
\end{align*}
for some constant $K$ (we have also used that $\norm{w_i(\tau_0)}_{L^2(m)}$ is uniformly bounded). 
Therefore, for some constants $C_i$, 
\begin{align*}
a(\tau - \tau_0)^{1/2}\norm{\grad w_i(\tau)}_{L^2(m)} & \leq C_1 + C_2\int_{\tau_0}^\tau\frac{a(\tau - \tau_0)^{1/2}}{a(\tau - \tau^\prime)^{1/p}a(\tau^\prime-\tau_0)^{1/2}}a(\tau^\prime-\tau_0)^{1/2}\norm{\grad w_i(\tau^\prime)}_{L^2(m)}d\tau^\prime \\ & \hspace{.5cm} + K\int_{\tau_0}^\tau\frac{a(\tau - \tau_0)^{1/2}}{a(\tau - \tau^\prime)^{1/p}}ds. 
\end{align*}
Hence by choosing $\tau \in (\tau_0,\tau_0 + \bar{T})$ for $\bar{T}$ sufficiently small: 
\begin{equation*}
\sup_{\tau_0 < \tau < \tau_0 + \bar{T}}C_2\int_{\tau_0}^\tau \frac{a(\tau-\tau_0)^{1/2}}{a(\tau - \tau^\prime)^{1/p}a(\tau^\prime-\tau_0)^{1/2}} d\tau^\prime = \sup_{0 < \tau < \bar{T}}C_2\int_{0}^\tau \frac{a(\tau)^{1/2}}{a(\tau - \tau^\prime)^{1/p}a(\tau^\prime)^{1/2}} d\tau^\prime \leq \frac{1}{4}, 
\end{equation*}
and similarly, 
\begin{equation*}
\sup_{0 < \tau < \bar{T}}K\int_{0}^\tau \frac{a(\tau)^{1/2}}{a(\tau - \tau^\prime)^{1/p}} d\tau^\prime \leq 1, 
\end{equation*}
we have that $\norm{\grad w_i(\tau)}_{L^2(m)} \leq 2(C_1+1)a(\tau - \tau_0)^{-1/2}$ for $\tau \in (\tau_0,\tau_0 + \bar{T})$. 
 However, $\tau_0$ was arbitrary and $\bar{T}$ was independent of $\tau_0$ so $\norm{\grad w_i(\tau)}_{L^2(m)}$ must be uniformly bounded.
\end{proof}

By the precompactness just proved, the orbit $\set{w_i(\tau)}_{\tau < \log(T)}$ has a non-trivial $\alpha$-limit set $\mathcal{A}$, which we show is invariant under the self-similar PKS \eqref{def:resPKS}.
This is essentially equivalent to showing that as $\tau \rightarrow -\infty$ the remainder $\tilde R_i(\tau)$ becomes negligible, due to not being localized around $z_i$. The primary difficulty is the presence of the velocity field coming from the approximately non-atomic part $\tilde{w}_0$, which requires some care to properly deal with (as in \cite{GallagherGallay05}).
\begin{lemma} \label{lem:RiDecay} 
For any $i \in \set{1,...,N}$, $m > 2$ and $p \in (1,2)$, 
\begin{equation*}
\lim_{\tau \rightarrow -\infty}\norm{\tilde R_i(\tau)w_i(\tau)}_{L^p(m)} = 0.
\end{equation*}
\end{lemma} 
\begin{proof} 
Following a similar procedure as \cite{GallagherGallay05},  
\begin{align*}
\norm{\tilde R_i(\tau) w_i(\tau) \jap{\xi}^{m}}_p \leq \norm{\jap{\xi}^{-1}\tilde R_i(\tau)}_{\frac{2p}{p-2}}\norm{w_i(\tau)}_{L^2(m+1)} \lesssim \norm{\jap{\xi}^{-1}\tilde R_i(\tau)}_{\frac{2p}{p-2}}.
\end{align*}
The portion of $\tilde R_i(\tau)$ coming from the other concentrations is relatively easy to handle as we know a priori they are localized away from the $i$-th concentration by the uniform bound on $\norm{w_j}_{L^2(m)}$. Indeed, define $q = \frac{2p}{p-2}$ and choose $\nu \in (0,1-2/q)$ (as in \cite{GallagherGallay05}). Then by \eqref{ineq:VelWeightedLp}, 
\begin{align*}
\norm{\jap{\xi}^{-1}v_j(\xi - (z_j - z_i)e^{-\tau/2},\tau)}_q & \leq \norm{\jap{\xi}^{-1}\jap{\xi - (z_j - z_i)e^{-\tau/2}}^{-\nu}}_\infty\norm{\jap{\xi}^\nu v_j(\xi)}_{q} \\ & \lesssim e^{\nu \tau/2}\norm{w_j(\tau)}_{L^2(m)} \lesssim  e^{\nu \tau/2}. 
\end{align*} 
Dealing with the approximately non-atomic part is more difficult and is the content of the following lemma. 
\begin{lemma}
For all $q \in (2,\infty]$,
\begin{equation*}
\lim_{\tau \rightarrow -\infty}\norm{\jap{\xi}^{-1}e^{\tau/2}\tilde{v}_0(e^\tau,\xi e^{\tau/2} + z_j)}_{L^q_\xi}   = \lim_{t \rightarrow 0} t^{\frac{1}{2} - \frac{1}{q}}\norm{\tilde{v}_0(t,x)\jap{\frac{\abs{x-z_j}}{\sqrt{t}}}^{-1}}_{L_x^q} = 0. \label{ineq:vanishingv0} 
\end{equation*}
\end{lemma}
\begin{proof}
We proceed similar to Lemma 4.2 in \cite{GallagherGallay05} but with several changes due to the lack of incompressibility and the unavailability of \cite{CarlenLoss94}. 
Without loss of generality, we can assume $z_j = 0$.
Recall that $\tilde{w}_0(t,x)$ satisfies the linear advection-diffusion equation (as a mild solution)
\begin{equation}
\partial_t \tilde{w}_0 + \grad \cdot ( v \tilde{w}_0) = \Delta \tilde{w}_0, \label{eq:w_0linear}
\end{equation}
where $v = B \ast u$, with initial data $\tilde{w}_0(0) = \mu_0$.
The main difficulty posed by the lack of incompressibility at this step is that we no longer have the results of Carlen and Loss \cite{CarlenLoss94} to provide pointwise estimates on the fundamental solution of \eqref{eq:w_0linear}. 
However, such precise pointwise control is not necessary.

Since $\mu_0(\set{0}) = 0$,  for all $\delta > 0$, there exists some $r > 0$ such that $\mu_0(B_{4r}) < \delta$.
Write $w^{(2)}$ as the non-negative mild solution to \eqref{eq:w_0linear} with initial data $\mu_0 \mathbf{1}_{\Real^2 \setminus B_{4r}}$ constructed in Lemma \ref{lem:GenLinear}. Next define $w^{(1)} = \tilde w_0 - w^{(2)}$ which is a mild solution to \eqref{eq:w_0linear} with initial data $\mu_0 \mathbf{1}_{B_{4r}}$. 
In what follows, denote $v^{(1)}$ and $v^{(2)}$ the corresponding velocity fields determined from $v^{(i)} = B \ast w^{(i)}$. 
Although we cannot immediately conclude $w^{(1)}$ is non-negative we still have the a priori estimates: for all $p \in [1,\infty]$ and $q \in (2,\infty]$, 
\begin{align*}
\limsup_{t \searrow 0} t^{1-\frac{1}{p}}\norm{w^{(1)}(t)}_p \lesssim \norm{\mu_0 \mathbf{1}_{B_{4r}}}_{\mathcal{M}} & \lesssim \delta, \\ 
\limsup_{t \searrow 0} t^{\frac{1}{2}-\frac{1}{q}}\norm{v^{(1)}(t)}_q \lesssim \norm{\mu_0 \mathbf{1}_{B_{4r}}}_{\mathcal{M}} & \lesssim \delta,  
\end{align*}

Let $\phi(x)$ be a smooth, radially symmetric, non-increasing cut-off function which is one for $\abs{x} \leq  2r$ and zero for $\abs{x} > 4r$. 
Now, further decompose $w^{(2)}(t,x) = \phi(x)w^{(2)}(t,x) + (1-\phi(x))w^{(2)}(t,x) := w^{(3)}(t,x) + w^{(4)}(t,x)$ and $v^{(3)} = B \ast w^{(3)}$ and $v^{(4)} = B \ast w^{(4)}$.  
One can compute the rate at which mass flows into the origin by (using the definition of distribution solution), 
\begin{align*}
\frac{d}{dt}\int w^{(3)} (t,x) dx = \frac{d}{dt}\int w^{(2)}(t,x) \phi(x) dx  & = \int w^{(2)}(t,x)\Delta \phi(x) + w^{(2)}(t,x) v(t,x) \cdot \grad \phi(x) dx \\ 
& \lesssim \norm{w^{(2)}(t)}_1\frac{1}{r^2} + \norm{w^{(2)}}_1\norm{v(t,x)}_\infty\frac{1}{r} \\ 
  & \lesssim \frac{1}{r^2} + \frac{\norm{w^{(2)}}_1}{r t^{1/2}}.
\end{align*}
Since the RHS is integrable at $t = 0$ and $\int w^{(2)}(0,x) \phi(x) dx = 0$, we have that 
\begin{equation*}
\lim_{t \rightarrow 0} \int w^{(3)}(t,x) dx = \lim_{t \rightarrow 0} \int w^{(2)}(t,x) \phi(x) dx = 0. 
\end{equation*}
By interpolation against the a priori $L^\infty$ bound $\norm{w^{(2)}(t)}_\infty \lesssim t^{-1}$, for $p \in (1,\infty)$,  
\begin{equation*}
t^{1-1/p}\norm{w^{(3)}(t)}_p \leq t^{1-1/p}\norm{w^{(3)}(t)}_\infty^{1-1/p}\norm{w^{(3)}(t)}^{1/p}_1 \lesssim \norm{w^{(3)}(t)}^{1/p}_1 \rightarrow 0. 
\end{equation*}
Therefore by \eqref{ineq:VelLp}, for all $q \in (2,\infty]$, 
\begin{equation*}
\lim_{t\rightarrow 0} t^{\frac{1}{2}-\frac{1}{q}}\norm{v^{(3)}(t)}_{q} = 0. 
\end{equation*}
It remains to control $v^{(4)}(t,x)$. For $\abs{x} < r$ we have, 
\begin{align*}
t^{\frac{1}{2} - \frac{1}{q}}\abs{v^{(4)}(t,x)} \lesssim t^{\frac{1}{2} - \frac{1}{q}} \int_{\abs{y} > 2r} \frac{w^{(4)}(t,y)}{\abs{x-y}} dy \lesssim \frac{t^{\frac{1}{2} - \frac{1}{q}}}{r}\norm{w^{(4)}(t)}_1 \rightarrow 0.  
\end{align*}
For $\abs{x} > r$ we have,
\begin{align*} 
t^{\frac{1}{2} - \frac{1}{q}}\norm{v^{(4)}(t,x)\mathbf{1}_{\abs{x} > r}\jap{\frac{\abs{x}^2}{t}}^{-1}}_q \leq \jap{\frac{r^2}{t}}^{-1}t^{\frac{1}{2} - \frac{1}{q}}\norm{v^{(4)}(t,x)}_q \rightarrow 0, 
\end{align*}
by \eqref{ineq:PropEquivApriori}.
Putting the estimates together, we have shown that for all $\delta > 0$,  
\begin{equation*}
\limsup_{t \rightarrow 0} t^{\frac{1}{2} - \frac{1}{q}}\norm{\tilde v_0(t,x) \jap{\frac{\abs{x}^2}{t}}^{-1}} \lesssim \delta, 
\end{equation*}
which proves the claim by choosing $\delta$ arbitrarily small. 
\end{proof} 
This completes the proof of Lemma \ref{lem:RiDecay}.  
\end{proof} 

The following is a direct consequence of Lemma \ref{lem:RiDecay} and the proof proceeds analogously to Lemma 6.3 in \cite{GallagherGallay05} so we omit it. 
\begin{lemma} 
$\mathcal{A}$ is invariant under the self-similar PKS \eqref{def:resPKS}.
\end{lemma}

By standard considerations, $\mathcal{A}$ consists only of compact, entire orbits of finite and constant self-similar energy $\mathcal{G}$.
However, by Proposition \ref{prop:Galpha}, the $G_\alpha$  are the unique functions of this type as the self-similar free energy is strictly decreasing on all other sets. 
Therefore we have proved: 
\begin{lemma}
$\mathcal{A} = \set{G_{\alpha_i}(\xi)}$ and hence for all $m > 2$,
\begin{equation}
\lim_{\tau \rightarrow -\infty}\norm{w_i(\tau) - G_{\alpha_i}}_{L^2(m)} = 0. \label{eq:wilimit}
\end{equation}
\end{lemma}

We now proceed to step (ii) and construct suitable $\tilde{w}_j$.  
As noted above, unlike in \cite{GallagherGallay05}, \eqref{eq:wilimit} does not immediately imply Proposition \ref{prop:EquivSolutions}. 
Indeed, our decomposition is of the following form (denoting $w_j$ in $t,x$ coordinates again),  
\begin{equation}
\sum_{j = 1}^N w_j(t,x) = \sum_{j = 1}^N \frac{1}{t}G_{\alpha_j}\left(\frac{x-z_j}{\sqrt{t}}\right) + \alpha_j \frac{1}{t} \tilde{w}_j\left(\log t,\frac{x-z_j}{\sqrt{t}}\right), \label{def:widecomp}
\end{equation}
but cannot be decoupled into an equation for individual $j$. Therefore, it is still not obvious how to construct the set of $\tilde{w}_j$ for an arbitrary mild solution. 
Define 
\begin{equation}
\alpha_i\frac{1}{t}\bar{w}_i\left(\log t,\frac{x - z_i}{\sqrt{t}}\right)  = w_i(t,x)- \frac{1}{t}G_{\alpha_i}\left( \frac{x-z_i}{\sqrt{t}}\right),
\end{equation}
which by \eqref{eq:wilimit} satisfies $\lim_{\tau \rightarrow -\infty}\norm{\bar{w}_i(\tau)}_{L^2(m)} = 0$ in self-similar variables. 
Note again, that unlike the analogous perturbations in \cite{GallagherGallay05}, $\tilde{w}_i \neq \bar{w}_i$. 
However, it turns out that the perturbations $\tilde{w}_i$ are not very far from $\bar{w}_i$, so we will produce suitable $\tilde{w}_j$ using a contraction mapping argument around $\bar{w}_i$ in self-similar variables.  
Precisely, we will construct correctors to $\bar{w}_j$, denoted $R_j$, which will be used to define $\tilde{w}_j = \bar{w}_j + R_j$. 
To do so, we construct solutions $\set{R_i(\tau,\xi)}_{i = 1}^N$ to the following system of integral equations 
\begin{align}
R_i(\tau) & = -\bar{w}_i(\tau) -\int_{-\infty}^\tau \mathcal{T}_{\alpha_i}(\tau - \tau^\prime)(\grad \cdot ((\bar{w}_i + R_i)\sum_{j \neq i} (1-\phi(\xi e^{\tau^\prime/2} + z_i - z_j))v^{G_j}(\xi - (z_j - z_i)e^{-\tau^\prime/2})) ) d\tau^\prime \nonumber \\ 
   & \hspace{.5cm} -\int_{-\infty}^\tau \mathcal{T}_{\alpha_i}(\tau - \tau^\prime)(\grad \cdot (\sum_{i \neq j}\frac{\alpha_j}{\alpha_i}(\bar{w}_j + R_j)(\xi - (z_j - z_i)e^{-\tau^\prime/2})\phi(\xi e^{\tau^\prime/2})v^{G_i})) d\tau^\prime \nonumber  \\
  & \hspace{.5cm} -\int_{-\infty}^\tau \mathcal{T}_{\alpha_i}(\tau-\tau^\prime)(\grad \cdot (\frac{1}{\alpha_i}G_i \sum_{j \neq i} v^{G_j}(\xi - (z_j - z_i)e^{-\tau^\prime/2})))d\tau^\prime  \nonumber \\
   & \hspace{.5cm} -\int_{-\infty}^\tau \mathcal{T}_{\alpha_i}(\tau-\tau^\prime)(\grad \cdot (\frac{1}{\alpha_i}G_i \sum_{j \neq i, j \geq 1}\alpha_j \tilde{v}_j(\xi - (z_j - z_i)e^{-\tau^\prime/2}))) d\tau^\prime \nonumber  \\
  & \hspace{.5cm} -\int_{-\infty}^\tau \mathcal{T}_{\alpha_i}(\tau-\tau^\prime)(\grad \cdot (\frac{1}{\alpha_i}G_ie^{\tau^\prime/2}\tilde{v}_0(e^{\tau^\prime},\xi e^{\tau^\prime/2} + z_i))) d\tau^\prime \nonumber \\
  & \hspace{.5cm} -\int_{-\infty}^\tau \mathcal{T}_{\alpha_i}(\tau-\tau^\prime)( \grad \cdot ((R_i + \bar{w}_i)\sum_{j = 1}^N \alpha_j \tilde{v}_j(\xi - (z_j - z_i)e^{-\tau^\prime/2})))d\tau^\prime \nonumber \\ 
  & \hspace{.5cm} -\int_{-\infty}^\tau \mathcal{T}_{\alpha_i}(\tau-\tau^\prime)(\grad \cdot ((R_i + \bar{w}_i) e^{\tau^\prime/2}\tilde{v}_0(e^{\tau^\prime},\xi e^{\tau^\prime/2} + z_i))) d\tau^\prime, \label{def:RjDuhamel}
\end{align}
where now $\tilde{v}_j = B\ast (\bar{w}_j + R_j)$ and $\tilde{v}_0(t,x) := B \ast \tilde{w}_0$. 
By construction, $\tilde{w}_j = \bar{w}_j + R_j$ solves \eqref{def:wiDuhamel}. 
To simplify \eqref{def:RjDuhamel}, first notice that $\bar{w}_i$ is a mild solution to
\begin{equation*}
\partial_\tau \bar{w}_i + \Lambda_{\alpha_i} \bar{w}_i  = L\bar{w}_i - \alpha_i\grad \cdot (\bar{w}_i \bar{v}_i) - \grad \cdot \left((\bar{w}_i + \frac{1}{\alpha_i}G_i)(e^{\tau/2}\tilde{v}_0(e^\tau,\xi e^{\tau/2} + z_i) + \bar{V}_i + v^{g_i})\right), 
\end{equation*}
where $v^{g_i}$ is defined as in \eqref{def:vgi} and if $\bar{v}_j = B \ast \bar{w}_j$ then we define 
\begin{equation*}
\bar{V}_i := \sum_{j \neq i} \alpha_j \bar{v}_j(\tau, \xi - (z_j - z_i)e^{-\tau/2}).  
\end{equation*}
Since $\bar{w}_i(\tau) \rightarrow 0$ in $L^2(m)$ as $\tau \rightarrow -\infty$ we can write $\bar{w}_i$ as a solution to the integral equation
\begin{equation}
\bar{w}_i(\tau) = -\int_{-\infty}^\tau \mathcal{T}_{\alpha_i}(\tau - \tau^\prime)\grad \cdot \left[\alpha_i\bar{w}_i\bar{v}_i + (\bar{w}_i + \frac{1}{\alpha_i}G_i)\left(e^{\tau^\prime/2}v_0(e^{\tau^\prime},\xi e^{\tau^\prime/2} + z_i) + \bar{V}_i + v^{g_i} \right) \right]d\tau^\prime. \label{eq:barwiDuhamel}
\end{equation}
Also write $v^{R_i} = B \ast R_i$ and 
\begin{equation*}
v^{r_i}(\tau,\xi) := \sum_{j \neq i} \alpha_j v^{R_j}(\tau,\xi - (z_j - z_i)e^{-\tau/2}). 
\end{equation*}
Now applying \eqref{eq:barwiDuhamel} to \eqref{def:RjDuhamel} gives
\begin{align}
R_i(\tau) & = -\int_{-\infty}^\tau \mathcal{T}_{\alpha_i}(\tau - \tau^\prime)(\grad \cdot ((\bar{w}_i + R_i)\sum_{j \neq i} (1-\phi(\xi e^{\tau^\prime/2} + z_i - z_j))v^{G_j}(\xi - (z_j - z_i)e^{-\tau^\prime/2})) ) d\tau^\prime \nonumber \\ 
   & \hspace{.5cm} -\int_{-\infty}^\tau \mathcal{T}_{\alpha_i}(\tau - \tau^\prime)(\grad \cdot (\sum_{i \neq j}\frac{\alpha_j}{\alpha_i}(\bar{w}_j + R_j)(\xi - (z_j - z_i)e^{-\tau^\prime/2})\phi(\xi e^{\tau^\prime/2})v^{G_i})) d\tau^\prime \nonumber  \\
   & \hspace{.5cm} -\int_{-\infty}^\tau \mathcal{T}_{\alpha_i}(\tau-\tau^\prime)(\grad \cdot (\frac{1}{\alpha_i}G_i v^{r_i})) d\tau^\prime % \nonumber  \\
%   & \hspace{.5cm} 
 -\int_{-\infty}^\tau \mathcal{T}_{\alpha_i}(\tau-\tau^\prime)( \grad \cdot (R_i (v^{r_i}+\bar{V}_i + \bar v_i)))d\tau^\prime \nonumber \\ 
   & \hspace{.5cm} -\int_{-\infty}^\tau \mathcal{T}_{\alpha_i}(\tau-\tau^\prime)( \grad \cdot (\bar{w}_i v^{r_i}))d\tau^\prime -\int_{-\infty}^\tau \mathcal{T}_{\alpha_i}(\tau-\tau^\prime)( \grad \cdot ( (\bar{w}_i + R_i)v^{R_i}))d\tau^\prime  \nonumber \\
  & \hspace{.5cm} 
-\int_{-\infty}^\tau \mathcal{T}_{\alpha_i}(\tau-\tau^\prime)(\grad \cdot (R_i e^{\tau^\prime/2}\tilde{v}_0(e^{\tau^\prime},\xi e^{\tau/2} + z_i))) d\tau^\prime \nonumber \\ 
& \hspace{.5cm} + \int_{-\infty}^\tau \mathcal{T}_{\alpha_i}(\tau - \tau^\prime) \grad \cdot (\bar w_i v^{g_i}) d\tau^\prime
. \label{def:RjForRealDuhamel}
\end{align}
\begin{lemma} 
For $\epsilon>0$ sufficiently small and for all $T$ sufficiently small (as always depending on $\epsilon$), the system of integral equations \eqref{def:RjForRealDuhamel} has a unique solution $\set{R_i(\tau)}_{i = 1}^N$ in $C((-\infty, \log T);L^2(m))$. 
Moreover, for $\tau \leq \log T$ sufficiently small, $R_i(\tau)$ satisfies for all $\gamma > 1$,  
\begin{align}
\max_{1 \leq i \leq N}\norm{R_i(\tau)}_{L^2(m)} \lesssim_\gamma e^{-\frac{d^2}{256\gamma}e^{-\tau}}. \label{ineq:2xExp}
\end{align} 
\end{lemma}
\begin{remark} 
The decay \eqref{ineq:2xExp} is natural when one considers that the $R_i$ correct for long-range interactions between concentrations which are Gaussian localized and are being separated by the coordinate system at a rate of $e^{-\tau/2}$. 
\end{remark} 
\begin{proof} 
The proof will be a contraction mapping argument. 
In order to see \eqref{ineq:2xExp}, we show that the source terms are all double exponentially small. 
First, we will identify a cancellation between the first term and the last which makes the total contribution much smaller. Specifically, we see that 
\begin{align}
\mathcal{R}(\tau) & := \int_{-\infty}^\tau \mathcal{T}_{\alpha_i}(\tau - \tau^\prime) \grad \cdot (\bar w_i v^{g_i}) d\tau^\prime  \nonumber \\ & \quad  -\int_{-\infty}^\tau \mathcal{T}_{\alpha_i}(\tau - \tau^\prime)(\grad \cdot (\bar{w}_i\sum_{j \neq i} (1-\phi(\xi e^{\tau^\prime/2} + z_i - z_j))v^{G_j}(\xi - (z_j - z_i)e^{-\tau^\prime/2})) ) d\tau^\prime  \nonumber \\
&  = \int_{-\infty}^\tau \mathcal{T}_{\alpha_i}(\tau - \tau^\prime)(\grad \cdot (\bar{w}_i\sum_{j \neq i} \phi(\xi e^{\tau^\prime/2} + z_i - z_j)v^{G_j}(\xi - (z_j - z_i)e^{-\tau^\prime/2})) ) d\tau^\prime. \nonumber 
\end{align}    
Then, for $p \in (1,2)$ using Proposition \ref{prop:SpecT} (iii), 
\begin{align*} 
\norm{\mathcal{R}(\tau)}_{L^2(m)} & \lesssim \int_{-\infty}^\tau \frac{e^{-\nu(\tau-\tau^\prime)}}{a(\tau-\tau^\prime)^{1/p}}\norm{\bar{w}_i\sum_{j \neq i} \phi(\xi e^{\tau^\prime/2} + z_i - z_j)v^{G_j}(\xi - (z_j - z_i)e^{-\tau^\prime/2})}_{L^p(m)} d\tau^\prime \\ 
& \lesssim \int_{-\infty}^\tau \frac{e^{-\nu(\tau-\tau^\prime)}}{a(\tau-\tau^\prime)^{1/p}}\norm{\bar{w}_i\sum_{j \neq i} \phi(\xi e^{\tau^\prime/2} + z_i - z_j)v^{G_j}(\xi - (z_j - z_i)e^{-\tau^\prime/2})}_{L^p(m)} d\tau^\prime \\ 
& \lesssim \int_{-\infty}^\tau \frac{e^{-\nu(\tau-\tau^\prime)}}{a(\tau-\tau^\prime)^{1/p}}\sum_{j \neq i}\norm{\bar{w}_i \phi(\xi e^{\tau^\prime/2} + z_i - z_j)}_{L^2(m)} \norm{v^{G_j}(\xi - (z_j - z_i)e^{-\tau^\prime/2})}_{2p/(2-p)} d\tau^\prime.  
\end{align*} 
By \eqref{ineq:GaussianLocalize} and \eqref{ineq:PropEquivApriori}, 
\begin{align*} 
\norm{\bar{w}_i\sum_{j \neq i} \phi(\xi e^{\tau^\prime/2} + z_i - z_j)}^2_{L^2(m)} & \leq \int_{\abs{\xi} \geq \frac{d}{4}e^{-\tau^\prime/2}}\jap{\xi}^{2m}\abs{\bar{w}_i(\tau^\prime,\xi)}^2 d\xi \\
 & \lesssim e^{-\frac{d^2}{256\gamma}e^{-\tau^\prime}}\int_{\abs{\xi} \geq \frac{d}{4}e^{-\tau^\prime/2}} e^{\frac{\abs{\xi}^2}{16\gamma}} \jap{\xi}^{2m}\abs{\bar{w}_i(\tau^\prime,\xi)} d\xi \\ 
& \lesssim_{\gamma,m} e^{-\frac{d^2}{256\gamma}e^{-\tau^\prime}} \int_{\abs{\xi} \geq \frac{d}{4}e^{-\tau^\prime/2}} e^{\frac{\abs{\xi}^2}{4\gamma}} \abs{\bar{w}_i(\tau^\prime,\xi)} d\xi \lesssim e^{-\frac{d^2}{256\gamma}e^{-\tau^\prime}}.
\end{align*} 
Since $\norm{v^{G_j}(\xi - (z_j - z_i)e^{-\tau^\prime/2})}_{2p/(2-p)} \lesssim \abs{\alpha_j}$ by Propositions \ref{prop:Galpha} and \ref{prop:Vel}, it follows that 
the total contribution of this term is double-exponentially small uniformly in $N$.
The remaining source term is treated analogously, which only requires the following new estimate, 
\begin{align*} 
\int \abs{\bar{w}_j(\xi - (z_j - z_i)e^{-\tau^\prime/2})}^2 \phi(\xi e^{\tau^\prime/2}) \jap{\xi}^{2m} d\xi & \leq \int_{\abs{\xi}e^{\tau^\prime/2} \leq 3d/4}\abs{\bar{w}_j(\xi - (z_j - z_i)e^{-\tau^\prime/2})}^2 \jap{\xi}^{2m} d\xi \\ 
& \hspace{-2cm}\lesssim    \jap{d e^{-\tau^\prime/2}}^{2m}e^{\frac{-d^2}{64\gamma} e^{-\tau^\prime}}\int_{\abs{\xi + (z_j-z_i)e^{-\tau^\prime/2}} \leq 3d/4 e^{-\tau^\prime/2}} \abs{\bar{w}_j(\xi)} e^{\frac{\abs{\xi}^2}{4\gamma}} d\xi \\
& \hspace{-2cm}\lesssim_{\gamma,m}  e^{\frac{-d^2}{256\gamma} e^{-\tau^\prime}}. 
\end{align*}
We now apply a contraction mapping argument which mirrors the one used to prove Propositions \ref{prop:Fmapping} and \ref{prop:Fcontraction}. 
All of the linear terms involving $\bar{w}_i$ are harmless since by \eqref{eq:wilimit}, $\norm{\bar{w}_i(\tau)}_{L^2(m)} \rightarrow 0$ as $\tau \rightarrow -\infty$ and hence can be made small by choosing $T$ small.  
The second to last term in \eqref{def:RjForRealDuhamel} is being treated here as linear (rather than nonlinear as in Proposition \ref{prop:Fmapping} and \ref{prop:Fcontraction}), as $v_0$ is being considered as an external field.
Due to \eqref{eq:wilimit}, it follows that 
\begin{align*} 
\norm{v(t) - \sum_{i = 1}^N \frac{1}{\sqrt{t}}v^{G_i}\left(\frac{\cdot - z_i}{\sqrt{t}}\right)}_{4} \lesssim \epsilon t^{-1/4},
\end{align*} 
which implies for $\epsilon$ sufficiently small, \eqref{def:w0Original} can be written as a perturbation of \eqref{eq:SNDefinition}.
In particular, we may write $\tilde{w}_0$ as a Duhamel integral involving $S_N(t,s)$
\begin{align*} 
\tilde w_0(t) = S_N(t,0)\mu_0 - \int_0^t S_N(t,s)\grad \cdot \left[\tilde w_0(s)\left(v(t) - \sum_{i = 1}^N \frac{1}{\sqrt{s}}v^{G_i}\left(\frac{\cdot - z_i}{\sqrt{s}}\right) \right)\right] ds, 
\end{align*}
and consider the Duhamel integral as small in the critical norm, from which one can show that $\tilde w_0$ is uniquely determined and enjoys all of the properties in Lemma \ref{lem:GenLinear}.
Therefore, it follows from \eqref{ineq:linearHyperLoc} and \eqref{ineq:VelLp} that
\begin{align} 
\limsup_{t \rightarrow 0} t^{\frac{1}{2} - \frac{1}{q}}\norm{\tilde{v}_0(t)}_q \lesssim \epsilon. \label{ineq:v0limsup}
\end{align} 
For $p \in (1,2)$ using Proposition \ref{prop:SpecT} (iii)
\begin{align*}
\norm{\int_{-\infty}^\tau \mathcal{T}_{\alpha_i}(\tau-\tau^\prime)(\grad \cdot (R_i(\tau^\prime) e^{\tau^\prime/2}\tilde{v}_0(e^{\tau^\prime},\xi e^{\tau^\prime/2} + z_i))) d\tau^\prime}_{L^2(m)} & \\ & \hspace{-5cm} \lesssim \int_{-\infty}^\tau \frac{e^{-\nu(\tau - \tau^\prime)}}{a(\tau - \tau^\prime)^{1/p}} \norm{R_i(\tau^\prime) e^{\tau^\prime/2}\tilde{v}_0(e^{\tau^\prime},\xi e^{\tau^\prime/2} + z_i)}_{L^p(m)} d\tau^\prime \\ 
& \hspace{-5cm} \lesssim \int_{-\infty}^\tau \frac{e^{-\nu(\tau - \tau^\prime)}}{a(\tau - \tau^\prime)^{1/p}} \norm{e^{\tau^\prime/2}\tilde{v}_0(e^{\tau^\prime},\xi e^{\tau^\prime/2} + z_i)}_{2p/(2-p)}\norm{R_i(\tau^\prime)}_{L^2(m)} d\tau^\prime.
\end{align*} 
However, (writing $t = e^{\tau^\prime}$),  
\begin{equation*}
\norm{e^{\tau^\prime/2}\tilde{v}_0(e^{\tau^\prime},\xi e^{\tau^\prime/2} + z_i)}_{2p/(2-p)} = t^{\frac{1}{2} - \frac{2-p}{2p}}\norm{\tilde{v}_0(t,x)}_{2p/(2-p)}.
\end{equation*}
By Propositions \ref{prop:SNProperties} and \ref{prop:Vel},
\begin{equation*}
\limsup_{t \rightarrow 0} t^{\frac{1}{2} - \frac{2-p}{2p}}\norm{v_0(t,x)}_{2p/(2-p)} \lesssim \norm{\mu_0}_{pp} < \epsilon. 
\end{equation*}
Hence, for $\tau < \log(t)$ for $t$ sufficiently small, we have,  
\begin{equation*}
\norm{\int_{-\infty}^\tau \mathcal{T}_{\alpha_i}(\tau-\tau^\prime)(\grad \cdot (R_i(\tau^\prime) e^{\tau^\prime/2}v_0(\xi e^{\tau^\prime/2} + z_i,e^{\tau^\prime}))) d\tau^\prime}_{L^2(m)} \lesssim \epsilon M[R_i](e^\tau).  
\end{equation*}

The rest of the terms can be handled using the same techniques as the proofs of Propositions \ref{prop:Fmapping} and \ref{prop:Fcontraction}. Accordingly, the lemma follows from the contraction mapping theorem. 
\end{proof}

There is one last remaining detail: due to the nonlinear terms in \eqref{def:RjForRealDuhamel} it is not immediately obvious that one re-constitutes the original $\sum_{i}w_i(t,x)$ by summing the $\tilde{w}_j$. From
\begin{equation*}
\sum_{i = 1}^N \frac{1}{t}G_{\alpha_i}\left(\frac{x-z_i}{\sqrt{t}}\right) + \frac{\alpha_i}{t}\tilde{w}_i\left(\log t,\frac{x-z_i}{\sqrt{t}}\right) = \sum_{i = 1}^N \frac{1}{t}\alpha_i R_i\left(\log t,\frac{x - z_i}{\sqrt{t}}\right) + w_i(t,x),   
\end{equation*}
it suffices to prove the following lemma: 
\begin{lemma} \label{lem:sumR}
For all $t \in (0,T)$, 
\begin{equation*}
\sum_{i = 1}^N \frac{1}{t} \alpha_i R_i\left(\log t, \frac{x- z_i}{\sqrt{t}}\right) = 0. 
\end{equation*}
\end{lemma} 
\begin{proof} 
Denote 
\begin{align*} 
F(t,x) := \sum_{i = 1}^N \frac{1}{t} \alpha_i R_i\left(\log t, \frac{x- z_i}{\sqrt{t}}\right). 
\end{align*} 
Multiplying both sides of \eqref{def:RjForRealDuhamel} by $\alpha_i$ and changing coordinates,  
we see that $F$ is a mild solution to the following PDE (using that the source terms in \eqref{def:RjForRealDuhamel} cancel after summation): 
\begin{align*} 
\partial_t F &  + \grad \cdot \left(F \sum_{i = 1}^N \frac{1}{\sqrt{t}}v^{G_i}\left(\frac{x - z_i}{\sqrt{t}}\right)\right)  
+ \grad \cdot \left( \sum_{i = 1}^N \frac{1}{t}G_{i}\left(\frac{x - z_i}{\sqrt{t}}\right) v^F \right) \\ 
& + \grad \cdot \left(F \sum_{j=1}^N \alpha_j\frac{1}{\sqrt{t}} \tilde v_j\left(\log t, \frac{x-z_i}{\sqrt{t}}\right)\right) + \grad \cdot \left( \sum_{i =1 }^N \frac{1}{t}\bar w_{i}\left(\frac{x - z_i}{\sqrt{t}}\right) v^F\right) + \grad \cdot \left(F\tilde v_0\right) - \Delta F = 0, \\ 
F(0) & = 0, 
\end{align*} 
still denoting $\tilde v_j = B \ast \tilde w_j = B \ast (\bar w_i + R_i)$. 
From \eqref{ineq:2xExp} and \eqref{ineq:L2mInject}, additionally satisfies for all $\gamma > 1$: 
\begin{align} 
t^{1/4}\norm{F(t)}_{4/3} \lesssim_\gamma e^{-\frac{d^2}{256 \gamma}t^{-1}}, \label{ineq:Ffast}
\end{align} 
Re-writing $F$ in terms of Duhamel's formula using the linear propagator for \eqref{eq:SNDefinition},
\begin{align*}
F(t) & = -\int_0^t S_N(t,s) \grad \cdot \left( \sum_{i = 1}^N \frac{1}{t}G_{i}\left(\frac{x - z_i}{\sqrt{t}}\right) v^F \right) ds \\
& \quad - \int_0^t S_N(t,s) \grad \cdot \left(F \sum_{j=1}^N \alpha_j\frac{1}{\sqrt{t}} \tilde v_j\left(\log t, \frac{x-z_i}{\sqrt{t}}\right)\right) ds  \\
& \quad - \int_0^t S_N(t,s) \left[\grad \cdot \left( \sum_{i =1 }^N \frac{1}{t}\bar w_{i}\left(\frac{x - z_i}{\sqrt{t}}\right) v^F\right) + \grad \cdot \left(F\tilde v_0\right)\right]ds. 
\end{align*} 
By Proposition \ref{prop:SNProperties} (iii), 
\begin{align*} 
t^{1/4}\norm{F(t)}_{4/3} & \lesssim t^{1/4}\int_0^t \frac{1}{(t-s)^{3/4}}\left[\frac{t}{s}\right]^{1/2 + \gamma - \lambda_0}\norm{\sum_{i = 1}^N \frac{1}{t}G_{i}\left(\frac{x - z_i}{\sqrt{t}}\right) v^F}_{1}ds \\ 
& \quad + t^{1/4}\int_0^t \frac{1}{(t-s)^{3/4}}\left[\frac{t}{s}\right]^{1/2 + \gamma - \lambda_0} \norm{F(s)}_{4/3}\norm{\sum_{j=1}^N \alpha_j\frac{1}{\sqrt{s}} \tilde v_j\left(\log s, \frac{\cdot-z_i}{\sqrt{s}}\right)}_{4} ds \\ 
& \quad + t^{1/4}\int_0^t \frac{1}{(t-s)^{3/4}}\left[\frac{t}{s}\right]^{1/2 + \gamma - \lambda_0} \left( \norm{\sum_{i =1 }^N \frac{1}{s}\bar w_{i}\left(\frac{x - z_i}{\sqrt{s}}\right)}_{4/3}\norm{v^F(s)}_{4} + \norm{F(s)}_{4/3}\norm{v_0(s)}_4\right) ds. 
\end{align*}  
Using \eqref{ineq:v0limsup} and $\norm{\bar w_i(\tau)}_{L^2(m)} \rightarrow 0$ with \eqref{ineq:L2mInject} and \eqref{ineq:VelLp}, we can choose $t$ small enough to move the the latter three terms back to the left-hand side, deducing (also we used Proposition \ref{prop:Galpha}),
\begin{align*} 
t^{1/4}\norm{F(t)}_{4/3} & \lesssim t^{1/4}\int_0^t \frac{1}{s^{1/2}(t-s)^{3/4}}\left[\frac{t}{s}\right]^{1/2 + \gamma - \lambda_0} s^{1/4}\norm{F(s)}_{4/3} ds. 
\end{align*} 
Since $t^{1/4}\norm{F(t)}_{4/3}$ vanishes faster than any polynomial via \eqref{ineq:Ffast}, we can apply a variant of Lemma 5.4 in \cite{GallagherGallay05} to deduce that $F \equiv 0$.
\end{proof}

By applying Lemma \ref{lem:sumR}: 
\begin{align*}
\sum_{i = 1}^N w_i & = \sum_{i = 1}^N \frac{1}{t}G_{\alpha_i}\left(\frac{x-z_i}{\sqrt{t}}\right) + \frac{1}{t}\bar{w}_i\left(\log t,\frac{x-z_i}{\sqrt{t}}\right) \\ 
& = \sum_{i = 1}^N \frac{1}{t}G_{\alpha_i}\left(\frac{x-z_i}{\sqrt{t}}\right) + \frac{\alpha_i}{t}\tilde{w}_i\left(\log t, \frac{x- z_i}{\sqrt{t}} \right).
\end{align*}
Therefore, the above construction yields a proper decomposition of the \emph{original} mild solution $u(t,x)$. 
Moreover, $\tilde{w}_j$ satisfy \eqref{def:wiDuhamel} and for $T$ chosen sufficiently small, the perturbations $\tilde{w}_j$ and $\tilde{w}_0$ are inside the ball \eqref{def:BepT}.
This completes the proof of Proposition \ref{prop:EquivSolutions}. 

\subsection{Final Step of Theorem \ref{thm:PKSUnique}} \label{sec:finalStep}
\begin{proof}(Theorem \ref{thm:PKSUnique})
By Propositions \ref{prop:Fmapping} and \ref{prop:Fcontraction}, the contraction mapping theorem implies that there exists a unique solution to \eqref{def:w0Duhamel} and \eqref{def:wiDuhamel} which satisfies $M[\tilde{w}](t) < \epsilon$ for $t \leq T$ small.
This solution may be assembled into a function $u(t,x) \in L^\infty(0,T;L^1) \cap \set{\sup_{t \in (0,T)} t^{1/4}\norm{u(t)}_{4/3} < \infty}$ which satisfies \eqref{eq:MildSolutionDef} of Definition \ref{def:MildSolution} (the $L^1$ bound may be verified from the arguments of Proposition \ref{prop:Fmapping}). 
The proof of Theorem \ref{thm:Basics} does not require non-negativity, and hence $u(t,x)$ also satisfies $\norm{u(t,x)}_\infty \lesssim t^{-1}$ for short time and by bootstrapping parabolic regularity, after $t>0$ $u(t,x)$ is a classical solution to \eqref{def:PKS} on $(0,T]$.
In order to have a mild solution in the sense of Definition \ref{def:MildSolution}, we need now to verify two remaining properties: that $u(t,x) \in C_w([0,T];\mathcal{M}_+(\Real^2))$
 (and in particular is non-negative) and that it achieves the initial data in the weak$^\ast$ topology on measures, $u(t) \rightharpoonup^\star \mu$ as $t \searrow 0$. 

Since it is classical for $t > 0$, $u(t,x) \in C_w((0,T];\mathcal{M}(\Real^2))$, however since the integral in \eqref{eq:MildSolutionDef} is very singular it is not immediately clear in which norms it vanishes as $t \searrow 0$.
Indeed, as pointed out in Remark \ref{rmk:NonConverge}, the integral does \emph{not} generally vanish in the critical norm $t^{1/4}\norm{\cdot}_{4/3}$. 
Let $\phi \in C_c^\infty$ and consider $\lim_{t \rightarrow 0^+} u(t)$ in the sense of distributions. 
For $t > 0$, using \eqref{ineq:HeatLpLqEasy}, \eqref{ineq:VelLp},    
\begin{align*} 
\abs{\int \phi \int_0^t e^{(t-s)\Delta}\grad \cdot (u(s)\grad c(s)) ds dx} & = \abs{\int \grad \phi \int_0^t e^{(t-s)\Delta}\left( u(s)\grad c(s)\right) ds dx} \\ 
& \lesssim \norm{\grad \phi}_\infty \int_0^t \norm{u(s)}_{4/3}\norm{\grad c(s)}_{4} ds \lesssim \norm{\grad \phi}_\infty \int_0^t \frac{1}{s^{1/2}} ds. 
\end{align*} 
Hence, the integral $\lim_{t \rightarrow 0^+}\int_0^t e^{(t-s)\Delta} \grad \cdot (u(s) \grad c(s)) ds = 0$ in the sense of distributions as $t \rightarrow 0^+$. Together with \eqref{eq:MildSolutionDef} this implies that at least $\lim_{t \rightarrow 0^+} u(t) = \mu$ in the sense of distributions. 
In order to improve this to weak$^\star$ convergence we use the uniform $L^1$ bound and tightness in $\mathcal{M}(\Real^2)$. 
By construction, for $v = B \ast u$ we have the a priori estimate
\begin{align} 
\norm{v(t) - \sum_{i = 1}^N \frac{1}{\sqrt{t}}v^{G_i}\left(\frac{\cdot - z_i}{\sqrt{t}}\right)}_{4} \lesssim \epsilon t^{-1/4},   \label{ineq:vctrl}
\end{align} 
where for $t$ sufficiently small the implicit constant is independent of $N$ and $\epsilon$.  
This implies that $v(t)$ satisfies the tightness condition \eqref{def:vtight} in Lemma \ref{lem:GenLinear} (in \S\ref{subsec:SNExistence}), and hence we may apply the tightness argument in Lemma \ref{lem:GenLinear} to prove that $u(t,x)$ is tight in $L^1$ as $t \rightarrow 0^+$; in particular for all $\delta$, we can choose $\epsilon$ small and $R$ large so that
\begin{align} 
\limsup_{t \rightarrow 0^+}\norm{u(t)\mathbf{1}_{\Real^2\setminus B_R}}_1 \lesssim \delta.  \label{ineq:tightness}
\end{align} 
Now let $\psi \in C^0$ be an arbitrary bounded, continuous function and let $\psi_\delta \in C^\infty_c$ be such that $\sup_{x \in B_R}\abs{\psi - \psi_\delta} < \delta$. 
It follows that for $t > 0$ sufficiently small, using \eqref{ineq:tightness}, the convergence in distribution and the uniform $L^1$  bound on $u$,    
\begin{align*}
\abs{\int \psi(x)(u(t,x)-\mu) dx} & \leq \abs{\int_{B_R} (\psi(x) - \psi_\delta(x)) (u(t,x)-\mu) dx} + \abs{\int_{B_R} \psi_\delta(x)(u(t,x)-\mu)dx} \\ & \quad + \abs{\int_{\Real^2 \setminus B_R} \psi(x)(u(t,x)-\mu) dx} \\ 
 & \lesssim \delta. 
\end{align*}  
Since $\delta$ and $\psi \in C^0$ were arbitrary, we have that $u(t) \rightharpoonup^\star \mu$.

It remains to verify that $u(t)$ is non-negative. 
Consider now the \emph{linear} advection-diffusion initial-value problem with $v(t) = B \ast u(t)$ given by:  
\begin{subequations} \label{def:AdDifw}
\begin{align} 
\partial_t w + \grad \cdot (vw) & = \Delta w \\ 
w(0) & = \mu, 
\end{align} 
\end{subequations}
of which the constructed $u(t)$ is a mild solution. By Lemma \ref{lem:GenLinear}, there exists a non-negative mild solution to this problem which satisfies the same a priori estimates as $u(t)$, hence we need only verify that \eqref{def:AdDifw} admits only one mild solution.    
Any mild solution of \eqref{def:AdDifw} $w(t)$ can be written by (using the continuity properties of $S_N(t,s)$) 
\begin{align*} 
w(t) = S_N(t,0)\mu - \int_0^t S_N(t,s)\grad \cdot \left[w(s)\left(v(t) - \sum_{i = 1}^N \frac{1}{\sqrt{s}}v^{G_i}\left(\frac{\cdot - z_i}{\sqrt{s}}\right) \right)\right] ds.
\end{align*}
Therefore, the control in \eqref{ineq:vctrl} implies that for $\epsilon$ smaller than a universal constant, we may treat the \eqref{def:AdDifw} as a small perturbation of \eqref{eq:SNDefinition} and use a contraction mapping argument similar to that employed in the proof of \eqref{ineq:w0PropFmapping} to prove that in fact solutions to \eqref{def:AdDifw} (with this specific $v$) are unique and hence $u(t)$ agrees with the non-negative mild solution of \eqref{def:AdDifw} constructed in Lemma \ref{lem:GenLinear}.   
Hence we have a mild solution to \eqref{def:PKS} satisfying all of Definition \ref{def:MildSolution}. 

Finally, Proposition \ref{prop:EquivSolutions} implies that any other mild solution can also be written as a solution to \eqref{def:w0Duhamel} and \eqref{def:wiDuhamel}, and by the uniqueness implied by the contraction mapping theorem, must agree with the above constructed solution. This completes the proof of Theorem \ref{thm:PKSUnique}.  
\end{proof} 

\section{Lipschitz Dependence on Initial Data} \label{sec:Lipschitz}
Although the contraction mapping theorem generally provides Lipschitz dependence on initial data for free, the above argument does not. 
This is because all of the linear propagators used in the proof of Theorem \ref{thm:PKSUnique} depend on the initial data itself and the decomposition used to construct the solution. 
However, combining our decomposition \eqref{def:decompu} and the contraction mapping arguments of Section \S\ref{sec:ContractMap} with the main ideas of Section 5.3 of \cite{GallagherGallay05}, we are able to refine the results of Gallagher and Gallay to obtain Lipschitz dependence on initial data.
As the method applies equally well to the 2D Navier-Stokes equations in vorticity form \eqref{def:NSE}, we also prove that the solution map of the Navier-Stokes equations is locally Lipschitz. 
\subsection{Proof of Theorem \ref{thm:Lipschitz}}
\begin{proof} 
Let $\epsilon > 0$ be a small, fixed constant independent of $\delta$ to be chosen later. In particular, we can require $\delta$ to be small with respect to $\epsilon$. If $\delta < \epsilon$ then 
all the atoms of mass bigger than $\epsilon$ in $\mu^1$ and $\mu^2$ must be in the same location (this is the utility of the total variation norm as opposed to a weaker norm). Therefore, as in \cite{GallagherGallay05}, we may decompose the initial measures as 
\begin{equation*}
\mu^l = \sum_{j = 1}^N \alpha_j^l \delta_{z_j} + \mu_0^l, 
\end{equation*}
such that $\norm{\mu_0^l}_{pp} < \epsilon$ and
\begin{equation*}
\norm{\mu^1 - \mu^2}_{\mathcal{M}} = \norm{\mu_0^1 - \mu_0^2}_{\mathcal{M}} + \sum_{j = 1}^N \abs{\alpha_j^1 - \alpha_j^2} < \delta. 
\end{equation*}
Let $u^1$, $u^2$ be the unique mild solutions of \eqref{def:PKS} in $\chi_T$ associated to each of the respective initial measures constructed in Theorem \ref{thm:PKSUnique}. 
By Proposition \ref{prop:EquivSolutions}, we may decompose $u^l$ each into respective large atomic pieces and smaller perturbations as in \eqref{def:decompu}: 
\begin{align}
u^l(t,x) & = \tilde{w}_0^l(t,x) + \sum_{i = 1}^N \frac{1}{t} G_{\alpha^l_i}\left(\frac{x - z_i}{\sqrt{t}}\right) + \sum_{i = 1}^N \alpha^l_i \frac{1}{t}\tilde{w}_i^l\left(\log t,\frac{x - z_i}{\sqrt{t}}\right), \label{def:lipDecomp}
\end{align}  
where the perturbations $\tilde{w}^l_i$ satisfy the corresponding integral equations \eqref{def:w0Duhamel} and \eqref{def:wiDuhamel}.
Analogously to \cite{GallagherGallay05} we use the following quantity as the norm to measure the difference between the two solutions,
\begin{equation*}
\Delta(t) := \sup_{s \in (0,t)}\norm{\tilde{w}_0^1(s) - \tilde{w}_0^2(s)}_{1} + M[\tilde{w}^1 - \tilde{w}^2](t),  
\end{equation*}
where for future convenience we define 
\begin{equation*}
\Delta_0(t) := \sup_{s \in (0,t)}\norm{\tilde{w}_0^1(s) - \tilde{w}_0^2(s)}_{1} + K_0M_0[\tilde{w}^1_0 - \tilde{w}^2_0](t). 
\end{equation*}
By Proposition \ref{prop:Galpha} and the decomposition \eqref{def:lipDecomp}, Theorem \ref{thm:Lipschitz} is equivalent to: there exists a $T>0$ and $C_L < \infty$ (independent of $\delta$ and $\mu^2$), such that $\Delta(T) \leq C_L\delta$.  

The decompositions of $u^l$ define different linear propagators which are centered around the same points but have different masses in the concentrations, denoted $S_N^l(t,s)$ and $\mathcal{T}_{\alpha_i^l}(\tau)$. 
The difference between these linear propagators must be controlled, hence in order to continue we need the equivalents of Proposition 5.5 and 5.6 of \cite{GallagherGallay05}.
The proofs are a straightforward contraction mapping argument which follows by writing one linear propagator as a Duhamel integral equation involving the other and using (iv) and (v) of Proposition \ref{prop:Galpha} to estimate the terms in the integral. We omit the details of the contraction mapping, but the proofs of Proposition \ref{prop:Galpha} (iv) and (v), which are trivial for NSE but not for PKS, can be found in Appendix \S\ref{apx:PropSelfSim}. 
For $S_N(t,s)$ we need the following. 
\begin{proposition} \label{prop:SNPerturb}
There exists some $t_0$ sufficiently small ($t_0 \lesssim d^2$) such that the following holds (with implicit constants independent of $\delta$ and $\epsilon$): 
\begin{itemize}
\item[(i)] For all $p \in [1,\infty]$ and $\nu \in \mathcal{M}(\Real^2)$ we have 
\begin{equation}
\norm{\left(S^1_N(t,s) - S_N^2(t,s)\right)\nu}_{L^p} \lesssim \frac{\delta}{(t-s)^{1-1/p}}\norm{\nu}_{\mathcal{M}}, \;\;\; 0 \leq s < t < s+t_0. 
\end{equation}
\item[(ii)] There exists some $\lambda_0 \in (0,1/2)$ independent of $\epsilon$ such that the following holds: for all $\gamma > 0$ sufficiently small and for all $f \in L^1(\Real^2)$, 
\begin{equation}
\norm{ \left(S^1_N(t,s) - S_N^2(t,s)\right)\grad f}_{p} \lesssim \frac{\delta}{(t-s)^{3/2 - 1/p}}\left( \frac{t}{s} \right)^{\gamma + 1/2 - \lambda_0}\norm{f}_{1}, \;\;\; 0 < s < t < s + t_0. 
\end{equation}
\end{itemize}
\end{proposition}
Similarly we need the analogous estimate for $\mathcal{T}_\alpha$.
\begin{proposition} \label{prop:TPerturb}
Fix $\alpha \in (0,8\pi)$, $m > 2$. Then for some $\nu \in (0,1/2)$ (which depends on $\alpha$), all $q \in [1,2]$, all $\beta$ sufficiently small (in absolute value), and all $f \in L^q(m)$,
\begin{equation}
\norm{\left(\mathcal{T}_{\alpha + \beta}(\tau) - \mathcal{T}_{\alpha}(\tau)\right)\grad f}_{L^2(m)} \lesssim_{\alpha,q} \abs{\beta} \frac{e^{-\nu \tau}}{a(\tau)^{1/q}}\norm{f}_{L^q(m)}, \;\; \tau > 0, \label{ineq:TDiffGradDecay}
\end{equation}
where the implicit constant is independent of $\beta$. 
\end{proposition}  

Armed with these estimates we may now use arguments similar to those in Section \S\ref{sec:ContractMap} to estimate the norm $\Delta(t)$ using the integral equations satisfied by $\tilde{w}_i^l$. 
Consider first, 
\begin{align*}
\tilde{w}_0^1(t) - \tilde{w}_0^2(t) &  = \left(S_N^1(t,0) - S_N^2(t,0)\right)\mu_0^1 + S_N^2(t,0)(\mu_0^1 - \mu_0^2) \\ 
& \hspace{.5cm} -\int_0^t \left(S_N^1(t,s) - S_N^2(t,s)\right)\grad \cdot \left[ \left(\tilde{v}^1_0(s) + \sum_{j=1}^N \tilde{v}_j^1(s)\right) \tilde{w}_0^1(s) \right] ds \\ 
& \hspace{.5cm} -\int_0^t S_N^2(t,s)\grad \cdot \left[ \left( \tilde{v}_0^1 + \sum_{j = 1}^N v_j^1(s) \right) \tilde{w}_0^1(s) - \left( \tilde{v}_0^2 + \sum_{j = 1}^N v_j^2(s)\right)\tilde{w}_0^2(s) \right] ds. 
\end{align*}
An argument using the known a priori estimates on $u^l$ proves that (for $t_0$ sufficiently small so that Propositions \ref{prop:SNProperties} and \ref{prop:SNPerturb} hold)
\begin{align}
\Delta_0(t) \leq \delta K_1(1 + M[\tilde{w}^1](t)^2) + K_2\left(M[\tilde{w}^1](t) + M[\tilde{w}^2](t)\right)\Delta(t), \;\; 0 < t < t_0, \label{ineq:Delta0}
\end{align}
for constants $K_1,K_2$ independent of $\epsilon$. 
We now turn to the estimates for $\tilde{w}_i^l$, $i \in \set{1,...,N}$.
In what follows write $G_{i,k} = F_{i,k}^1 - F_{i,k}^2$ and 
\begin{align}
\tilde{v}_j^l(\tau,\xi) & := B \ast \tilde{w}_j^l, \nonumber \\ 
v^{W^l_i}(\tau,\xi) & := \sum_{j \neq i} \alpha_j^l \tilde{v}^l_j(\tau,\xi - (z_j - z_i)e^{-\tau/2}), \nonumber \\ 
v^{g_i^l}(\tau,\xi) & := \sum_{j \neq i} v^{G_{\alpha_j^l}}(\tau,\xi - (z_j - z_i)e^{-\tau/2}).  \label{def:vgl} 
\end{align}
Consider the first term, 
\begin{align*}
G_{i,1}(\tau) & = \int_{-\infty}^\tau \mathcal{T}_{\alpha_i^1}(\tau - \tau^\prime)\grad \cdot \left(\frac{G_{\alpha^1_i}}{\alpha_i^1}v^{g_i^1}\right) - \mathcal{T}_{\alpha_i^2}(\tau - \tau^\prime)\grad \cdot \left( \frac{G_{\alpha^2_i}}{\alpha_i^2}v^{g_i^2}\right) d\tau^\prime \\
& = \int_{-\infty}^\tau \left(\mathcal{T}_{\alpha_i^1} - \mathcal{T}_{\alpha_i^2}\right)(\tau-\tau^\prime)\grad \cdot \left(\frac{G_{\alpha^1_i}}{\alpha_i^1}v^{g_i^1} \right) d\tau^\prime \\ 
& \hspace{.5cm} - \int_{-\infty}^\tau \mathcal{T}_{\alpha_i^2}(\tau-\tau^\prime)\left(\grad \cdot \left( \frac{G_{\alpha^2_i}}{\alpha_i^2}v^{g_i^2}\right) - \grad \cdot \left( \frac{G_{\alpha^1_i}}{\alpha_i^1}v^{g_i^1}\right) \right) 
d\tau^\prime \\ 
& = \int_{-\infty}^\tau \left(\mathcal{T}_{\alpha_i^1} - \mathcal{T}_{\alpha_i^2}\right)(\tau-\tau^\prime)\grad \cdot \left(\frac{G_{\alpha^1_i}}{\alpha_i^1}v^{g_i^1} \right) d\tau^\prime \\ 
& \hspace{.5cm} + \int_{-\infty}^\tau \mathcal{T}_{\alpha_i^2}(\tau - \tau^\prime)\grad \cdot \left[\frac{G_{\alpha_i^1}}{\alpha_i^1}\left( v^{g_i^1} - v^{g_i^2} \right) - v^{g_i^2}\left(\frac{G_{\alpha_i^2}}{\alpha_i^2} - \frac{G_{\alpha_i^1}}{\alpha_i^1}\right)\right] d\tau^\prime.  
\end{align*}
By \eqref{ineq:TGradDecay} and \eqref{ineq:TDiffGradDecay}, 
\begin{align*}
\norm{G_{i,1}(\tau)}_{L^2(m)} & \lesssim \int_{-\infty}^\tau \delta \frac{e^{-\nu(\tau - \tau^\prime)}}{a(\tau - \tau^\prime)^{1/2}}\norm{\frac{G_{\alpha^1_i}}{\alpha_i^1}v^{g_i^1}}_{L^2(m)} d\tau^\prime \\  
& \hspace{.5cm} + \int_{-\infty}^\tau \frac{e^{-\nu(\tau - \tau^\prime)}}{a(\tau - \tau^\prime)^{1/2}}\norm{\frac{G_{\alpha_i^1}}{\alpha_i^1} \left(v^{g_i^1} - v^{g_i^2}\right)}_{L^2(m)} d\tau^\prime \\ 
& \hspace{.5cm} + \int_{-\infty}^\tau \frac{e^{-\nu(\tau - \tau^\prime)}}{a(\tau - \tau^\prime)^{1/2}}\norm{\left(\frac{G_{\alpha_i^2}}{\alpha_i^2} - \frac{G_{\alpha_i^1}}{\alpha_i^1}\right) v^{g_i^2}}_{L^2(m)} d\tau^\prime.
\end{align*} 
Notice that while the third term is always zero for the Navier-Stokes equations, it is non-zero for PKS due to the nonlinear nature of the self-similar solutions.  
The first term can be estimated as we estimated $F_{i,1}$ in Lemma \ref{lem:Fmap}, which gives the following 
with an implicit constant independent of $\epsilon$ and $\alpha_i$,  
\begin{equation*}
\int_{-\infty}^\tau \delta \frac{e^{-\nu(\tau - \tau^\prime)}}{a(\tau - \tau^\prime)^{1/2}}\norm{\frac{G_{\alpha^1_i}}{\alpha_i^1}v^{g_i^1}}_{L^2(m)} d\tau^\prime \lesssim \delta e^{\tau/2}. 
\end{equation*} 
The second term can be estimated in a similar fashion but now using Proposition \ref{prop:Galpha} and Proposition \ref{prop:Vel} to deduce the following with an implicit constant independent of $\epsilon$ and $\alpha_i$,
\begin{equation*}
\int_{-\infty}^\tau \frac{e^{-\nu(\tau - \tau^\prime)}}{a(\tau - \tau^\prime)^{1/2}}\norm{\frac{G_{\alpha_i^1}}{\alpha_i^1} \left(v^{g_i^1} - v^{g_i^2}\right)}_{L^2(m)} d\tau^\prime \lesssim e^{\tau/2}\sum_{j \neq i} \abs{\alpha_j^1 - \alpha_j^2} \lesssim e^{\tau/2}\delta. 
\end{equation*} 
Finally the third term can be estimated similarly with an implicit constant independent of $\epsilon$ and Proposition \ref{prop:Galpha}, 
\begin{align*}
\int_{-\infty}^\tau \frac{e^{-\nu(\tau - \tau^\prime)}}{a(\tau - \tau^\prime)^{1/2}}\norm{\left(\frac{G_{\alpha_i^2}}{\alpha_i^2} - \frac{G_{\alpha_i^1}}{\alpha_i^1}\right) v^{g_i^2}}_{L^2(m)} d\tau^\prime & \lesssim e^{\tau/2} \norm{\frac{G_{\alpha_i^2}}{\alpha_i^2} - \frac{G_{\alpha_i^1}}{\alpha_i^1}}_{L^2(m+1)} \\ 
   \lesssim \frac{e^{\tau/2}}{\alpha_i^2} \abs{\alpha_i^2 - \alpha_i^1} \lesssim \frac{e^{\tau/2}}{\alpha_i^2} \delta.
\end{align*}
Hence putting the three estimates together we have,  
\begin{equation*}
\norm{G_{i,1}}_{L^2(m)} \lesssim \delta e^{\tau/2}\left(1 + \frac{1}{\alpha_i^2}\right). 
\end{equation*} 
Dealing with the other terms does not really present any new real challenges, so we will include less details. 
The term $G_{i,2}$ is done in a similar way: first re-write as
\begin{align*}
G_{i,2}(\tau) & = \sum_{j \neq i} \int_{-\infty}^\tau \mathcal{T}_{\alpha_i^1}(\tau - \tau^\prime)\grad \cdot \left( \frac{G_{\alpha_i^1}}{\alpha_i^1} \alpha_j^1 \tilde{v}_{j}^1 \right) d\tau^\prime  + \sum_{j \neq i} \int_{-\infty}^\tau \mathcal{T}_{\alpha_i^2}(\tau - \tau^\prime)\grad \cdot \left( \frac{G_{\alpha_i^2}}{\alpha_i^2} \alpha_j^2 \tilde{v}_{j}^2 \right)  d\tau^\prime \\ 
& = \sum_{j \neq i} \int_{-\infty}^\tau \left(\mathcal{T}_{\alpha_i^1} - \mathcal{T}_{\alpha_i^2}\right)(\tau - \tau^\prime)\grad \cdot \left( \frac{G_{\alpha_i^1}}{\alpha_i^1} \alpha_j^1 \tilde{v}_{j}^1 \right) d\tau^\prime  \\ 
& \hspace{.5cm} + \sum_{j \neq i} \int_{-\infty}^\tau \mathcal{T}_{\alpha_i^2}(\tau-\tau^\prime)\grad \cdot \left(\frac{G_{\alpha_i^2}}{\alpha_i^2} \left(\alpha_j^1\tilde{v}_j^1 - \alpha_j^2 \tilde{v}_j^2\right) \right) d\tau^\prime \\ 
& \hspace{.5cm} + \sum_{j \neq i} \int_{-\infty}^\tau \mathcal{T}_{\alpha_i^2}(\tau - \tau^\prime)\grad \cdot \left[\left(\frac{G_{\alpha_i^1}}{\alpha_i^1} - \frac{G_{\alpha_i^2}}{\alpha_i^2}\right) \alpha_j^1 \tilde{v}_{j}^1\right]  d\tau^\prime. \\ 
\end{align*}
By estimates analogous to those used for $F_{i,2}$ in Proposition \ref{prop:Fmapping} combined with Propositions \ref{prop:Galpha} and \ref{prop:TPerturb},
\begin{equation*}
\norm{G_{i,2}(\tau)}_{L^2(m)} \lesssim \delta e^{\gamma/2 - 1/q} M[\tilde{w}^1](t) + \frac{\delta}{\alpha_i^2}e^{\gamma/2 - 1/q}M[\tilde{w}^1](t) + \Delta(t)e^{\gamma/2 - 1/q}, 
\end{equation*} 
where $\gamma$ and $q$ are as in the proof of Lemma \ref{lem:Fmap} and the implicit constant is independent of $\epsilon$ and $\delta$. 
Similarly, we re-write the third term as
\begin{align*}
G_{i,3} 
 & = \int_{-\infty}^\tau \left(\mathcal{T}_{\alpha_i^1} - \mathcal{T}_{\alpha_i^2}\right)(\tau - \tau^\prime)\grad \cdot \left( \frac{G_{\alpha_i^1}}{\alpha_i^1}e^{\tau^\prime/2} \tilde{v}_{0}^1(e^{\tau^\prime},\xi e^{\tau^\prime/2} + z_i) \right) d\tau^\prime  \\ 
& \hspace{.5cm} + \int_{-\infty}^\tau \mathcal{T}_{\alpha_i^2}(\tau-\tau^\prime)\grad \cdot \left[ \left(\frac{G_{\alpha_i^1}}{\alpha_i^1} - \frac{G_{\alpha_i^2}}{\alpha_i^2}\right) e^{\tau^\prime/2}\tilde{v}_{0}^1(e^{\tau^\prime},\xi e^{\tau^\prime/2} + z_i) \right]  d\tau^\prime \\ 
& \hspace{.5cm} + \int_{-\infty}^\tau \mathcal{T}_{\alpha_i^2}(\tau-\tau^\prime)\grad \cdot \left[\frac{G_{\alpha_i^2}}{\alpha_i^2} e^{\tau^\prime/2}\left((\tilde{v}_0^1 - \tilde{v}_0^2)(e^{\tau^\prime},\xi e^{\tau^\prime/2} + z_i) \right) \right] d\tau^\prime. 
\end{align*}
Using ideas from Lemma \ref{lem:Fmap} we deduce that there exists some $K_3$ independent of $\epsilon$ and $i$ such that,  
\begin{equation*}
\norm{G_{i,3}(\tau)}_{L^2(m)} \leq K_3 \delta M_0[\tilde{w}_0^1](t) + K_3 \frac{\delta}{\alpha_i^2} M_0[\tilde{w}_0^1](t) + K_3 M_0[\tilde{w}_0^1 - \tilde{w}_0^2](t).  
\end{equation*}
We may continue in a similar manner to handle the remaining terms and eventually prove an inequality of the form for some $\eta > 0$ using also that $M[\tilde{w}^l](t)$ is uniformly bounded, 
\begin{align}
\norm{\tilde{w}_i^1 - \tilde{w}_i^2}_{L^2(m)} & \leq K_4\left(1 + e^{\eta \tau} + \frac{e^{\eta \tau}}{\alpha^2_i}\right)\delta + K_3M_0[\tilde{w}_0^1 - \tilde{w}_0^2](t) + K_5\left(\delta + e^{\eta \tau} + M[\tilde{w}^1](t) + M[\tilde{w}^2](t) \right) \Delta(t) \nonumber \\ 
& \leq K_4\left(1 + e^{\eta \tau} + \frac{e^{\eta \tau}}{\inf \alpha_j^2}\right)\delta + \frac{K_3}{K_0}\Delta(t) + K_5\left(\delta + e^{\eta\tau} + M[\tilde{w}^1](t) + M[\tilde{w}^2](t) \right) \Delta(t), \label{ineq:Lipschtildew} 
\end{align}
where $K_4$, $K_3$ and $K_5$ are all independent of $\epsilon$ and $\delta$.  
Choosing $\epsilon > 0$ small depending only on absolute constants, $K_0 \geq \max(1,4K_3)$ and $t < t_0$ for $t_0$ sufficiently small depending on $\epsilon$. Note, none of these choices depend on $\delta$ as long as $\delta \leq \epsilon$.  
Moreover, we may $T$ small such that $M[\tilde{w}^l] < \epsilon$ by the work of Section \S\ref{sec:PKSUnique}.  
Taking the supremum over $i$, from \eqref{ineq:Delta0} and \eqref{ineq:Lipschtildew} we deduce that
\begin{equation*}
\Delta(t) \lesssim \delta, \;\;\; 0 < t < t_0, 
\end{equation*}
which as remarked above, implies Theorem \ref{thm:Lipschitz}.
\end{proof}

\subsubsection*{Acknowledgments} 
The authors would like to deeply thank the anonymous referee for his or her careful reading of the manuscript and their commitment to ensuring the correctness and completeness of the work.
Jacob Bedrossian would like to thank Nancy Rodr\'{i}guez for helping to proofread the manuscript and Jose A. Carrillo for relaying the results of Campos and Dolbeault to us. 
J. Bedrossian was partially supported by NSF Postdoctoral Fellowship in Mathematical Sciences DMS-1103765.  

\appendix
\section{Appendix: Linear Estimates} \label{apx:LinearEstimates}

\subsection{Spectral Gap Estimate} \label{apx:SpectralGap} 
In this Appendix we sketch an independent proof of a weaker version of a result due to J. Campos and J. Dolbeault \cite{CamposDolbeault12}.
In what follows define $L_0^2(G_\alpha^{-1}d\xi):=\set{f \in L^2(G_\alpha^{-1}d\xi): \int f d\xi = 0}$.  

\begin{proposition} \label{prop:OurSpectralGap}
Let $f \in L^2_0(G_\alpha^{-1} d\xi)$. Then for all $\alpha \in (0,8\pi)$ there exists some $K_\alpha > 0$ such that
\begin{equation}
\norm{\mathcal{T}_\alpha(\tau)f}_{L^2(G_\alpha^{-1} d\xi)} \lesssim_\alpha e^{-K_\alpha\tau}\norm{f}_{L^2(G_\alpha^{-1}d\xi)}, \label{Apx:SpecDecay}
\end{equation} 
where $K_\alpha$ and the implicit constant only depend  on $K < 8 \pi$ for all $\alpha \leq K$. 
\end{proposition}
\begin{remark} \label{rmk:controlledSpec} 
That Proposition \ref{prop:OurSpectralGap} holds for $\alpha$ sufficiently small with uniformly controlled implicit constant and $K_\alpha$ for $\alpha \searrow 0$ can be shown either by the work of \cite{BlanchetDEF10} or by an argument essentially the same as Lemma \ref{lem:S1alphaSmall} below, using Proposition \ref{prop:Galpha} (iv). Indeed, one can show $K_\alpha \approx 1/2$ for $\alpha$ small in the same sense as Lemma \ref{lem:S1alphaSmall}. 
\end{remark}
Let $\K(x) := \frac{1}{2\pi}\log \abs{x}$ be the fundamental solution for Poisson's equation in two dimensions. 
The following characterization of $G_\alpha$ is important for what follows: 
\begin{equation}
\grad\log G_\alpha = \grad c_\alpha - \frac{1}{2}\xi. \label{pde:logPDE}
\end{equation}

Formally linearizing $\mathcal{G}$ (defined in \eqref{def:Gssfree}) around the stationary point $G_\alpha$ yields (using \eqref{pde:logPDE} and $\int f d\xi = 0$), 
\begin{equation*}
\mathcal{G}(G_\alpha + \epsilon f) = \mathcal{G}(G_\alpha) + \frac{\epsilon^2}{2}\left[\int \frac{\abs{f}^2}{G_\alpha} d\xi - \int f \K \ast f d\xi \right] +\mathcal{O}(\epsilon^3), 
\end{equation*}
which suggests a natural Lyapunov function for the linearized problem. Hence, define 
\begin{equation*}
\tilde{F}(f) = \frac{1}{2}\int \frac{\abs{f}^2}{{G_\alpha}} d\xi - \frac{1}{2}\int f \K \ast f d\xi.   
\end{equation*} 
It will turn out that $\tilde{F}$ is convex and that the linear evolution is the corresponding gradient flow with the appropriate metric.
The first step is the following dissipation inequality. 
The proof is a direct computation using \eqref{pde:logPDE} which we omit for brevity. 
\begin{proposition} 
Let $f_0$ be mean zero with $\tilde{F}(f_0) < \infty$ and let $f(\tau) = e^{\tau(L-\Lambda_\alpha)}f_0 = \mathcal{T}_\alpha(\tau)f_0$.
Then, 
\begin{equation}
\frac{d}{d\tau}\tilde{F}(f(\tau)) = -\int G_\alpha \abs{\grad\left(\frac{f}{G_\alpha}\right) - \grad c}^2 d\xi := -D(f), \label{ineq:dissipationLinear}
\end{equation}
where $-\Delta c = f$. 
\end{proposition}

Drawing intuition from classical Bakry-Emery analysis and the more recent insights of entropy dissipation methods (see e.g. \cite{ArnoldMarkowichEtAl01,CarrilloEntDiss01,CarrilloMcCannVillani03}) it is expected that convexity and coercivity of $\tilde{F}$ are  equivalent to the decay estimate \eqref{Apx:SpecDecay}.
Since $\tilde{F}$ is quadratic, this is in turn equivalent to strict positivity. 
\begin{proposition}\label{prop:Positive}  
The energy $\tilde{F}$ is strictly positive, that is, 
if $f \in L_0^2(G_\alpha^{-1}d\xi)$ and $f \neq 0$  then $\tilde{F}(f) > 0$. 
This is equivalent to strict positivity of the dissipation:  
\begin{equation*}
D(f) = \int G_\alpha \abs{\grad\left(\frac{f}{G_\alpha}\right) - \grad c}^2 d\xi > 0. 
\end{equation*}
\end{proposition}
\begin{proof}
Linearization of $\mathcal{G}$ with smooth, compactly supported perturbations and passing to the limit shows that necessarily $\tilde{F}(f) \geq 0$.  
Now, consider the possibility that $f \in L_0^2(G_\alpha^{-1}d\xi)$ and $\tilde{F}(f) = 0$. 
The dissipation inequality \eqref{ineq:dissipationLinear} implies that
\begin{equation*}
\int G_\alpha \abs{ \grad \frac{f}{G_\alpha} - \grad c}^2 d\xi = 0,
\end{equation*}
which since $G_\alpha$ is strictly positive implies
\begin{equation}
\grad \frac{f}{G_\alpha} - \grad c = 0, \label{eq:gradfc}
\end{equation}
almost everywhere. 
A bootstrap argument using the smoothing effect of the nonlocal term shows that $f$ is necessarily smooth. 
Moreover, we also have, 
\begin{equation*}
\frac{f}{G_\alpha} - c = K,
\end{equation*}
for some constant $K$. 
Since $f \in L^2_0(G_\alpha^{-1}d\xi)$ we also have $c \in L^\infty$, and hence $f G_\alpha^{-1} \in L^\infty$. 
Taking the divergence of \eqref{eq:gradfc} implies that 
\begin{equation*}
\Delta \frac{f}{G_\alpha} + f = 0. 
\end{equation*}
Re-naming $h = f G_\alpha^{-1}$, we see that the question of whether or not $\tilde{F}$ is strictly positive reduces to whether or not there are any \emph{bounded}, finite energy $\int \abs{\grad h}^2 dx < \infty$, solutions to the elliptic PDE 
\begin{equation}
\Delta h + G_\alpha h = 0. \label{def:EvilPDE}
\end{equation}
Ruling out bounded solutions of \eqref{def:EvilPDE} turns out to be by far the most difficult step in the proof of 
Proposition \ref{prop:OurSpectralGap}.  
\begin{lemma}
The elliptic PDE \eqref{def:EvilPDE} admits no bounded, finite energy solutions.  
\end{lemma}
\begin{proof}
The proof requires a lengthy ODE argument and several lemmas. 
The first step is to consider any potential solution and decompose into the radial Fourier series (with $\xi = (r\sin \theta,r\cos \theta)^T$):
\begin{equation*}
h(\xi) = \sum_{n = -\infty}^\infty f_n(r)e^{in\theta}. 
\end{equation*}
Let us first rule out radially symmetric solutions $f(r):=f_0(r)$. In this case, \eqref{def:EvilPDE} becomes 
\begin{equation}
\frac{1}{r}(rf^\prime)^\prime + G_\alpha f = f^{\prime\prime} + \frac{1}{r}f^\prime + G_\alpha f = 0. \label{ode:radsym}
\end{equation}
The next lemma is a standard ODE result: 
\begin{lemma}
The ODE \eqref{ode:radsym} has two linearly independent solutions, and the possible behaviors at zero and infinity are $f(r) \sim K$ for some constant $K$ or $f(r) \sim \log r$. 
\end{lemma}
Hence it suffices to exhibit a solution to \eqref{ode:radsym} which is bounded at zero and unbounded at infinity.
Such a solution will be provided by the zero eigenfunction: 
\begin{equation}
E^0_\alpha(\xi) := \frac{d}{d\lambda}G_\lambda(\xi)|_{\lambda = \alpha}. %\label{def:E0alphaDerivative}
\end{equation}
Indeed, by \eqref{pde:logPDE}
\begin{equation}
\Delta \frac{E_\alpha^0}{G_\alpha} + E^0_\alpha = 0. \label{def:E0alphaPDE}
\end{equation}
By \eqref{def:E0alphaPDE}, $e(\abs{\xi}) := E_\alpha^0(\xi) G_\alpha^{-1}(\xi)$ solves \eqref{ode:radsym}. 
Moreover, since $\int E^0_\alpha(\xi) dx = 1$, $e(r)$ is necessarily logarithmically unbounded at infinity (by \eqref{def:E0alphaPDE}). 
\begin{lemma} \label{lem:EalphaControl}
The function $e(r)$ is bounded at zero. 
\end{lemma}
\begin{proof}
Define, 
\begin{equation*}
n(t) = \int_{\abs{x} \leq \sqrt{t}} E^0_\alpha(x) dx, 
\end{equation*}
which solves the ODE
\begin{equation*}
4n^{\prime\prime} + n^\prime + \frac{1}{\pi t}\left( n m^\prime + n^\prime m \right) = 0,
\end{equation*}  
with boundary conditions $n(0) = 0$, $n(\infty) = 1$, where $m(t):= \int_{\abs{x} \leq \sqrt{t}} G_\alpha(x) dx$.
Note that this ODE is linear in $n(t)$. 
From here one can apply an analysis similar to what is done in Lemma 4.1 in \cite{Biler06} to prove that $n^\prime(0)$ exists and is finite, and hence $E_\alpha$ and $e$ are bounded at zero. 
The argument in \cite{Biler06} is already localized to a small neighborhood of zero, which is necessary as $n(t)$ does not satisfy the same monotonicity properties that $m(t)$ does. 
That $m^\prime(t)$ is well-behaved and satisfies certain monotonicity properties is necessary for the proof.
Since the argument is a little technical and follows that of \cite{Biler06} very closely we omit it for brevity.  
\end{proof}
\begin{remark} 
By Lemma \ref{lem:EalphaControl}, $E_\alpha^0$ is bounded, hence from \eqref{def:E0alphaPDE} a bootstrap argument implies that $E_\alpha^0$ is smooth.
\end{remark}

Now we turn to the angular modes. In this case we get the ODE (which holds for the real and imaginary parts of the solutions)
\begin{equation*}
\frac{1}{r}(rf^\prime)^\prime - \frac{n^2}{r^2}f + G_\alpha f = 0, 
\end{equation*}
which we re-write as 
\begin{equation}
(rf^\prime)^\prime - \frac{n^2}{r}f  + r G_\alpha f = 0. \label{ode:angular}
\end{equation}
Of course we have the corresponding classical ODE result: 
\begin{lemma}
The ODE \eqref{ode:angular} has two linearly independent solutions, and the possible behaviors at zero and infinity are $\sim r^{-n}$ or $\sim r^{n}$. 
\end{lemma}
We first rule out bounded  solutions supported in the mode $n = 1$, namely, solutions that satisfy  $\sim r^{-1}$ when $r$ goes to infinity  and  $\sim r$ when $r$ goes to zero. 
For this,  we  use that
\begin{equation}
\Delta \frac{\partial_{\xi_1} G_\alpha}{G_\alpha} + \partial_{\xi_1}G_\alpha = 0,
\end{equation}
and hence $n_1 = \partial_{r}G_\alpha G_\alpha^{-1}$ is a solution to \eqref{ode:angular} with $n=1$.  
By definition $n_1(0) = 0$.
It also follows from Proposition \ref{prop:Galpha} that necessarily $n_1(r)$ is linearly unbounded at infinity. 
This rules out any bounded, non-zero solutions in the first angular mode. Moreover, from the monotonicity of $G_\alpha$, we get the important fact that $n_1(r)$ is \emph{strictly negative} for $r > 0$.

Now we confront $n \geq 2$. For this we will use second-order comparison principles against $n_1$, similar to, for example, Chapter 8 in \cite{CoddingtonLevinson}.  
Suppose we have a bounded solution $f(r)$ to \eqref{ode:angular} with $n\geq 2$. 
Therefore, near zero $f(r) \sim r^{n}$ and near infinity $f(r) \sim r^{-n}$. 
\begin{lemma} \label{lem:zerocrossings}
Let $f$ be a solution to \eqref{ode:angular} with $n \geq 2$ which is bounded. Then $f \equiv 0$.  
\end{lemma} 
\begin{proof}
If $f$ vanishes in an open neighborhood of zero then by unique continuation $f \equiv 0$, therefore since $f \sim r^{n}$ near zero, 
we can assume $f$ is strictly positive on some open set (replacing $f$ by $-f$ if necessary). 
Define $x_1 \in (0,\infty]$ by 
\begin{equation*}
x_1 := \sup \set{r >0: 0 < f(s), \;\; \forall s \in (0,r)}. 
\end{equation*}
If $f(r)$ crosses zero at a finite value of $r$ then $x_1$ is the location of the first zero of $f$. 
If $f$ remains positive for all time then $x_1 = \infty$.  
We will compare $f$ to the strictly positive solution $g(r) := -n_1(r)$ of \eqref{ode:angular} with $n = 1$ on the interval $(0,x_1)$.  
Multiplying the ODE satisfied by $g$ by $f$ and vice-versa and then subtracting gives
\begin{align*}
(rf^\prime)^\prime g - (r g^\prime)^\prime f - \frac{1}{r}\left(n^2 - 1\right)gf = 0. 
\end{align*}
Integrate now from $0$ to $x_1$ (both sides will turn out to be integrable in the case $x_1 = \infty$): 
\begin{align}
\int_0^{x_1} (rf^\prime)^\prime g - (r g^\prime)^\prime f dr = \int_0^{x_1} \frac{1}{r}\left(n^2 - 1\right)gf dr. \label{eq:intecomp}
\end{align}
First notice that
\begin{equation*}
\left(r f^\prime g - r g^\prime f\right)^\prime = (rf^\prime)^\prime g - (r g^\prime)^\prime f. 
\end{equation*} 
However $r f^\prime g - r g^\prime f$ is zero at $r = 0$, which implies \eqref{eq:intecomp} becomes
\begin{equation*}
\lim_{r \rightarrow x_1} rf^\prime(r) g(r) - rg^\prime(r) f(r) = \int_0^{x_1} \frac{n^2-1}{r} g(r)f(r) dr \geq 0. 
\end{equation*}
If $x_1 < \infty$ then $f^\prime(x_1) \leq 0$ and $f(x_1) = 0$ which implies that the integral on the RHS must be equal to zero, which implies $f \equiv 0$ on $(0,x_1)$. By unique continuation, $f \equiv 0$ on $[0,\infty)$. 
If $x_1 = \infty$ we have to first show that the RHS of \eqref{eq:intecomp} is integrable. 
This follows since $g \sim r$ and $f \sim r^{-n}$ as $r \rightarrow \infty$. 
Moreover,  
\begin{equation*}
\lim_{r \rightarrow \infty} rf^\prime(r) g(r) - rg^\prime(r) f(r) = 0,   
\end{equation*}
since $g^\prime(r) \sim 1$ and $f^\prime(r) \lesssim r^{-3}$. 
Hence in the case that $x_1= \infty$, \eqref{eq:intecomp} still implies $f \equiv 0$.  
\end{proof} 
This completes the proof that \eqref{def:EvilPDE} has no non-trivial bounded solutions on $\Real^2$.  
\end{proof}
This in turn, completes the proof that $\tilde{F}$ and $D(f)$ are positive. 
\end{proof}
Proposition \ref{prop:Positive} proves that $\tilde{F}$ is positive, which implies $\tilde{F}$ is convex since it is quadratic in $f$. 
The next proposition uses compactness arguments to confirm first that $\tilde{F}$ is coercive and next that $D(f)$ controls $\tilde{F}$.
\begin{proposition} \label{prop:CoerciveDissipation}
For all $\alpha \in (0,8\pi)$ the following holds. 
\begin{itemize} 
\item[(i)] There is a constant $C_\alpha \in (0,1)$ such that for all $f \in L^2_0(G_\alpha^{-1}d\xi)$,  
\begin{equation*}
 0 \leq \int f \K \ast f dx \leq C_\alpha \int \abs{f}^2 G_\alpha^{-1} dx.  
\end{equation*} 
In particular, $\tilde{F}$ is coercive:
\begin{equation*}
(1-C_\alpha)\int \abs{f}^2 G_\alpha^{-1} dx \leq 2  \tilde{F}(f),   
\end{equation*}
where $1-C_\alpha>0$.
\item[(ii)] There exists a constant $K_\alpha > 0$ such that for  all $f \in L^2_0(G_\alpha^{-1}d\xi)$,  
\begin{equation}
 K_\alpha \tilde{F}(f) \leq D(f). \label{ineq:DDineq}
\end{equation}
\end{itemize}
\end{proposition}
\begin{remark}
By Remark \ref{rmk:controlledSpec}, we do not need to worry about the behavior of these constants as $\alpha \searrow 0$.  
\end{remark}

First let us show why Proposition \ref{prop:CoerciveDissipation} completes the proof of Proposition \ref{prop:OurSpectralGap}. 
By  \eqref{ineq:DDineq} and \eqref{ineq:dissipationLinear} we get exponential decay of $\tilde{F}$, which combined with coercivity implies
\begin{equation*}
\norm{f}_{L_0^2(G_\alpha^{-1}d\xi)}^2 \lesssim \tilde{F}(f) \lesssim \tilde{F}(f_0)e^{-K_\alpha t} \lesssim \norm{f_0}_{L_0^2(G_\alpha^{-1}d\xi)}^2 e^{-K_\alpha t}.  
\end{equation*}
This completes the proof of Proposition \ref{prop:OurSpectralGap}. 

Now let us turn to the proof of Proposition \ref{prop:CoerciveDissipation}.  
\begin{proof}(Proposition \ref{prop:CoerciveDissipation})
Since $\tilde{F}(f) > 0$ we already have 
\begin{equation*}
\int f \K \ast f d\xi < \int \abs{f}^2 G_\alpha^{-1} d\xi,  
\end{equation*}
for all $f \neq 0$. 
Suppose there exists a sequence of $\set{f_k}_{k = 1}^\infty \subset L_0^2(G_\alpha^{-1}d\xi)$, such that $\int \abs{f_k}^2 G_\alpha^{-1} d\xi = 1$ (without loss of generality by homogeneity) and
\begin{equation*}
\lim_{k \rightarrow \infty}\int f_k \K \ast f_k d\xi = 1. 
\end{equation*}
By the boundedness of $f_k$ in $L_0^2(G_\alpha^{-1}d\xi)$ we may extract a subsequence (not relabeled) that weakly converges in $L^2_0(G_\alpha^{-1}d\xi)$ to some limit $h$ which by lower semicontinuity satisfies
\begin{equation*}
\int \abs{h}^2 G_\alpha^{-1} d\xi \leq 1. 
\end{equation*} 
Since $f_k$ can also be chosen to converge weakly in $L^1$, we have that $\int h d\xi = 0$. 
We claim that 
\begin{equation}
1= \lim_{k \rightarrow \infty}\int f_k \K \ast f_k d\xi = \int h \K \ast h d\xi, \label{eq:hKh} 
\end{equation}
which implies both that $h \neq 0$ and $\tilde{F}(h) \leq 0$, in contradiction with Proposition \ref{prop:Positive}.  
We now prove \eqref{eq:hKh}. 
First, define $K_\delta^R = \K(\abs{x-y})\mathbf{1}_{\delta < \abs{x-y} < R}$ and break up the convolution into 
\begin{align*}
\int f_k \K \ast f_k dx & = \int f_k K_\delta^R \ast f_k dx + \frac{1}{4\pi}\int_{\abs{x-y} < \delta} f_k(x)f_k(y) \log\abs{x-y} dx dy \\ & \;\;\;+ \frac{1}{4\pi}\int_{\abs{x-y} > R} f_k(x)f_k(y) \log\abs{x-y} dx dy \\
& = T1 + T2 + T3. 
\end{align*}
The $f_k$ are all uniformly well localized, hence the term $T3$ can be made arbitrarily small by choosing $R > 1$ large: 
\begin{align*}
\int_{\abs{x-y} > R} f_k(x)f_k(y) \log\abs{x-y} dx dy & \leq \frac{\log R}{R}\int_{\abs{x-y} > R}\abs{f_k(x)}\abs{f_k(y)}\abs{x-y} dx dy \\ 
& \leq \frac{\log R}{R}\int_{\abs{x-y} > R}\abs{f_k(x)}\abs{f_k(y)}(\abs{x} + \abs{y}) dx dy \\ 
& \lesssim \frac{\log R}{R} \norm{f_k}_{L^1} \norm{f_k(x)\abs{x}}_{L^1} \\ 
& \lesssim \frac{\log R}{R} \int \abs{f_k}^2 G_\alpha^{-1} dx. 
\end{align*}
Similarly, since the $f_k$ are uniformly bounded in $L^2$, the term $T2$ can be made arbitrarily small by choosing $\delta$ small: 
\begin{align*}
\int_{\abs{x-y} < \delta} f_k(x)f_k(y) \log\abs{x-y} dx dy \lesssim \norm{f_k}_{L^2}^2\norm{ \log\abs{x-y}\mathbf{1}_{\abs{x-y}<\delta}}_{L^1}.  
\end{align*} 
Notice that the exact same arguments apply to $h$. Hence, it suffices to prove that for all $\delta, R$ we have 
\begin{equation*}
\int f_k K_\delta^R\ast f_k dx   \rightarrow   \int h K_\delta^R \ast h dx. 
\end{equation*}
This convergence follows from classical weak convergence arguments, but we will 
still give a proof. 
Indeed, consider 
\begin{align*}
\int f_k K_\delta^R\ast f_k dx - \int h K_\delta^R \ast h dx & = \int (f_k - h) K_\delta^R \ast h dx + \int f_k K_\delta^R \ast (f_k - h). 
\end{align*}
Since $f_k \rightharpoonup h$ in $L^1$ the first term converges to zero, so we need only focus on the latter. 
Define 
\begin{equation*}
v_k(x) := \int (f_k(y) - h(y))K_\delta^R(x-y) dy = (f_k - h)\ast K_\delta^R. 
\end{equation*}
By weak convergence again, $v_k(x) \rightarrow 0$ pointwise a.e.. Let $\bar{v}(x) := \sup_k v_k(x)$.
Firstly, 
\begin{align*}
\abs{\bar{v}(x)} \lesssim \sup_k \norm{f_k - h}_{L^1(B_R(x))} \max(\abs{\log \delta}, \abs{\log R}).  
\end{align*}
Since $f_k$ and $h$ are exponentially localized in $L^1$ uniformly in $k$ (since $f_k \in L^2_0(G_\alpha^{-1}d\xi)$), it follows that $\bar{v}(x)$ is pointwise exponentially localized as well as bounded, and hence integrable.
By the dominated convergence theorem, it follows that $v_k \rightarrow 0$ strongly in $L^1$. By interpolation this implies $v_k \rightarrow 0$ in $L^p$ for all $p \in [1,\infty)$ from which it follows that $\int f_k K_\delta^R \ast (f_k - h) \rightarrow 0$.
Putting all of the estimates together, we deduce that
\begin{equation*}
\int f_k \K \ast f_k dx \rightarrow \int h \K \ast h dx, 
\end{equation*}
which is the desired contradiction. This completes the proof of (i). 

The proof of (ii) continues in a similar fashion. 
We suppose there exists a sequence $\set{f_k}_{k = 1}^\infty$ of functions $f_k \neq 0$ normalized such that $\tilde{F}(f_k) \equiv 1$ but $D(f_k) \rightarrow 0$ and derive a contradiction. By the coercivity estimate (i),
 this implies 
$ 2 \leq 
 \norm{f_k}^2_{L^2(G_\alpha^{-1}d\xi)} \leq 2  (1-C_\alpha)^{-1}$ and $\norm{\grad \K \ast f_k}_2 \lesssim 1$.  
Extracting a subsequence if necessary, we can assume that $f_k$ converges weakly 
to some $h$ in $L^2(G_\alpha^{-1} d\xi)$. Since $f_k$ can also be chosen to converge weakly in $L^1$, we have that $\int h d\xi = 0$ and hence $h \in L_0^2(G_\alpha^{-1} d\xi) $.  By lower semicontinuity of 
$D(f)$ with respect to weak convergence, we also deduce that $ D(h)=0$.
To get a contradiction, we have just to prove that $h \neq 0$.

It follows from the bound $\norm{\grad \K \ast f_k}_2 \lesssim 1$ and 
 the boundedness of $D(f_k)$ that $\int G_\alpha \abs{\grad \frac{f_k}{G_\alpha}}^2 d\xi$ is uniformly bounded.
  Define  $V= V (G_\alpha) = \{  F \, :  \norm{F}_{V}^2 = \int G_\alpha (\abs{\grad F}^2 
 + \abs{F}^2   ) d\xi < \infty       \}. $
As $V$ is compactly embedded in $L^2(G_\alpha d\xi)$  and since 
 $ \frac{f_k}{G_\alpha} $ is bounded in  $V$, it is relatively compact in $L^2(G_\alpha d\xi)$. 
Hence, $f_k$ is relatively compact in  $L^2(G_\alpha^{-1} d\xi)$ and extracting a subsequence, 
we deduce that   $f_k$  converges strongly to $h$  in  $L^2(G_\alpha^{-1} d\xi)$ and hence $h \neq 0$. 
This ends the proof of the Proposition.  

\end{proof}

\subsection{Proof of Proposition \ref{prop:SpecT}} 
\subsubsection{Proof of \textit{(i)}} \label{sec:SpecTi}
Fix $\alpha \in (0,8\pi)$. 
Consider $f$ which solves $\partial_\tau f = Lf - \Lambda_\alpha f$ with initial data $f_0$. Written with Duhamel's formula this is
\begin{equation*}
f(\tau) = S(\tau)f_0 - \int_0^\tau S(\tau - s) \Lambda_\alpha f(s) ds. 
\end{equation*}
A straightforward contraction mapping argument similar to the others employed in this work implies that $\mathcal{T}_\alpha(\tau) = e^{\tau(L - \Lambda_\alpha)}$ defines a strongly continuous semigroup on $L^2(m)$. 
However, an additional analysis must be done to ensure \eqref{ineq:TBounded} holds independently  of $\tau$. 
By linearity, it suffices to consider nonnegative $f_0$ with $\norm{f_0}_{L^2(m)} = 1$. As above, write $f(\tau) = \mathcal{T}_\alpha(\tau)f_0$. 
As $L^2(m) \hookrightarrow L^1$, $f(\tau) \in L^1$. Furthermore, $\mathcal{T}_\alpha(\tau)$ preserves non-negativity and is in divergence form, hence $\norm{f(\tau)}_1 = \norm{f_0}_1$ for all $\tau > 0$. 

We first show that $\norm{f(\tau)}_2 \lesssim \norm{f_0}_2$ independent of $\tau$.
Indeed, 
\begin{align*}
\frac{1}{2}\frac{d}{d\tau}\int \abs{f}^2 d\xi & = \int f(Lf - \Lambda_\alpha f) d\xi \\ 
& = -\int \abs{\grad f}^2 d\xi + \frac{1}{2}\int \abs{f}^2 d\xi + \frac{3}{2}\int G_\alpha \abs{f}^2 d\xi - \int f \grad G_\alpha \cdot \grad c d\xi \\ 
& \leq -\int \abs{\grad f}^2 d\xi + C\norm{f}_2^2 + C\norm{f}_{4/3}\norm{\grad c}_4 \\ 
& \leq -\int \abs{\grad f}^2 d\xi + C\norm{f}_2^2 + C\norm{f}_{4/3}^2 \\
& \leq  -\int \abs{\grad f}^2 d\xi + C\norm{f}_2^2 + C\norm{f}_{1}\norm{f}_2 \\ 
& \leq  -\int \abs{\grad f}^2 d\xi + C\norm{f}_2^2 + C\norm{f_0}_{1}^2, 
\end{align*}
for some constants $C$ which depend on $\alpha$ but whose precise values are not very relevant.
 Similar to the proof of Theorem \ref{thm:Basics} (i), we apply the Gagliardo-Nirenberg inequality $\norm{f}_3^3 \leq C_{GNS}\norm{\grad f}^{2}_2 \norm{f}_1$ and
\begin{equation*}
\frac{1}{2}\frac{d}{d\tau}\int \abs{f}^2 d\xi \leq -C_{GNS}\frac{\norm{f}_3^3}{\norm{f_0}_1} + C\norm{f}_2^2 + C\norm{f_0}_{1}^2, 
\end{equation*}
Using that for all $K > 0$, $\norm{f}_2^2 \leq \frac{1}{K}\norm{f}_3^3 + K\norm{f}_1$
the above differential inequality implies that $\norm{f(\tau)}_2$ is uniformly bounded by a constant.  

Now we may use the uniform bound on $\norm{f(\tau)}_2$ to control $\norm{\abs{\xi}^m f}_2$, similar to what is done in Theorem 3.1 in \cite{GallayWayne02}. 
Computing as there, 
\begin{align*}
\frac{1}{2}\frac{d}{dt}\int \abs{\xi}^{2m} f^2 d\xi & = \int \abs{\xi}^{2m}f(Lf - \Lambda_\alpha f) d\xi \\
& = -\int \abs{\xi}^{2m}\abs{\grad f}^2 d\xi - m\int \abs{\xi}^{2m} f^2 d\xi + 2m^2\int \abs{\xi}^{2m-2}f^2 d\xi - \int \abs{\xi}^{2m} f \Lambda_\alpha f d\xi. 
\end{align*}
The latter term expands to the following after a short computation (using that $-\Delta c = f$)
\begin{align*}
-\int \abs{\xi}^{2m}\Lambda_\alpha f d\xi & = \frac{3}{2}\int \abs{\xi}^{2m}G_\alpha f^2 d\xi - \int \abs{\xi}^{2m}f(\grad G_\alpha \cdot \grad c) d\xi + m \int \abs{\xi}^{2m-2}\xi \cdot \grad c_{\alpha} f^2 d\xi. 
\end{align*}
By the rapid decay of $G_\alpha$ the first term is uniformly bounded via 
\begin{align*}
\int \abs{\xi}^{2m}G_\alpha f^2 d\xi \lesssim \norm{f}_{L^2}^2. 
\end{align*}
Using H\"older's inequality and the $L^4$ estimate \eqref{ineq:VelLp} we control the second term   
\begin{align*}
-\int \abs{\xi}^{2m}f \grad G_\alpha \cdot \grad c d\xi & \leq \norm{\abs{\xi}^{2m}\grad G_\alpha}_4 \norm{f}_2 \norm{\grad c}_{4} \\ 
& \lesssim \norm{f}_2\norm{f}_{4/3} \lesssim \norm{f}_2\norm{f}_{L^2(m)}.  
\end{align*}
For all $\delta > 0$ there exists a constant $C_\delta > 0$ such that the following three inequalities all hold (using the spatial decay of $\grad c_\alpha$),
\begin{align*}
\norm{f}_2\norm{f}_{L^2(m)} & \leq \delta\norm{f}_{L^2(m)}^2 + C_\delta\norm{f}_2^2,  \\ 
2m^2 \int \abs{\xi}^{2m-2} f^2 d\xi & \leq \delta \int \abs{\xi}^{2m}f^2 d\xi + C_\delta \int f^2 d\xi, \\ 
m \int \abs{\xi}^{2m-2}\xi \cdot \grad c_{\alpha} f^2 d\xi & \leq \delta \int \abs{\xi}^{2m} f^2 d\xi + C_\delta \int f^2 d\xi.  
\end{align*}
Putting all of the estimates together we have, 
\begin{align*}
\frac{1}{2}\frac{d}{d\tau} \int \abs{\xi}^{2m} f^2 d\tau & \leq -\int \abs{\xi}^{2m} \abs{\grad f}^2 d\xi - m\int \abs{\xi}^{2m}f^2 d\xi + 3\delta \int \abs{\xi}^{2m} f^2 d\xi + C \int f^2  d\xi. 
\end{align*}
Therefore, since $\norm{f}_2$ is uniformly bounded, for $\delta$ chosen sufficiently small this inequality implies that
$\norm{\abs{\xi}^{2m}f}_2$ is also uniformly bounded. Since $\mathcal{T}_\alpha(\tau)$ is a linear operator the bound 
\eqref{ineq:TBounded} must hold. 

A contraction mapping argument similar to what is done in Lemma 2.1 in \cite{GallayWayne05} and Lemma \ref{lem:precompact} shows that the uniform bound \eqref{ineq:TBounded} implies the regularization estimate \eqref{ineq:TGradientHyper}.
Moreover, in both cases, one can verify that Proposition \ref{prop:Galpha} (iv) implies the implicit constants in \eqref{ineq:TBounded} and \eqref{ineq:TGradientHyper} only depend 
on $K < 8 \pi $  for $\alpha \leq K $. 
Note that 
 $\mathcal{T}_\alpha$ can be treated as a small perturbation of $S(\tau)$ for $\alpha$ small.   

\subsubsection{Proof of \textit{(ii)}}
Proposition \ref{prop:OurSpectralGap} implies the following bound on the point spectrum of $L - \Lambda_\alpha$ in $L^2_0(G_\alpha^{-1})$.  % (for example, by contradiction). 

\begin{theorem}[Spectral Gap] \label{thm:CD_SpectralGap}
For all $\alpha \in (0,8\pi)$, there exists some $K_\alpha \in (0,1/2]$ such that any eigenvalue $\lambda$ of $L - \Lambda_\alpha$ in $L_0^2(G_\alpha^{-1}d\xi)$ satisfies 
\begin{equation*}
\textup{Re}(\lambda) \leq -K_\alpha,
\end{equation*}
for some $K_\alpha > 0$ which is uniformly bounded for $\alpha$ small.  
\end{theorem}

We extend Theorem \ref{thm:CD_SpectralGap} to the polynomial-weighted spaces in a way analogous to \cite{GallayWayne05}. 
In the polynomial-weighted spaces $\sigma(L - \Lambda_\alpha)$, the spectrum of $L-\Lambda_\alpha$, will not be discrete, as can be expected by considering $\sigma(L)$ (see \cite{GallayWayne02}). 
In what follows, for any linear operator $A$, denote $\sigma_{ess}(A)$ the essential spectrum, defined here as the set of $\lambda \in \sigma(A)$ such that $\lambda I - A$ is not Fredholm. This is in slight contrast to \cite{GallayWayne05}, where it is defined as $\sigma(A)\setminus \sigma_{disc}(A)$ where $\sigma_{disc}(A)$ denotes the discrete spectrum, the set of isolated eigenvalues of finite multiplicity.   
Note that the essential spectrum as we have defined is contained in the essential spectrum as defined by \cite{GallayWayne05} (see e.g. \cite{GoldbergEtAl}).
The advantage is that with the convention we take, $\sigma_{ess}(A)$ is invariant under compact perturbations \cite{EngelNagel}, however the disadvantage is that there may be points in the spectrum which are neither in $\sigma_{ess}(A)$ or $\sigma_{disc}(A)$. 
 We denote the point spectrum, the set of all eigenvalues, as $P\sigma(A)$, which clearly contains the discrete spectrum, but in general is larger.  
The following proposition extends Theorem \ref{thm:CD_SpectralGap} to $L^2_0(m)$ by bounding the point spectrum of $L - \Lambda_\alpha$.

\begin{proposition} \label{prop:SpecGapLm}
Fix $m > 1$. Then any eigenvalue $\lambda$ of $L - \Lambda_\alpha$ in $L_0^2(m)$ satisfies
\begin{equation*}
\textup{Re}(\lambda) \leq \max\left(-K_\alpha, \frac{1-m}{2}\right). 
\end{equation*}
\end{proposition} 

\begin{proof}
As in section 4 of   \cite{GallayWayne05}, we use an ODE argument. 
The idea is to show that any eigenfunction of $L - \Lambda_\alpha$ in $L_0^2(m)$ actually lies in $L_0^2(G_\alpha^{-1}d\xi)$ and hence we may apply Theorem \ref{thm:CD_SpectralGap}.
The approach of Gallay and Wayne \cite{GallayWayne05} is to re-write the linear operator in radial variables and reduce the question to a statement about the asymptotic behavior of an ODE, for which classical results of Coddington and Levinson may be applied \cite{CoddingtonLevinson}. 
Hence the core of the argument is the following lemma. 
\begin{lemma} \label{lem:TexpoLoc}
Let $f \in L_0^2(m)$ satisfy $Lf - \Lambda_\alpha f = \mu f$ for $\textup{Re}(\mu) > \frac{1-m}{2}$. Then there exists $\gamma \geq 0$ such that 
\begin{equation*}
\abs{f(\xi)} \lesssim (1+\abs{\xi}^2)^\gamma e^{-\abs{\xi}^2/4}, \;\;\; \xi \in \Real^2. 
\end{equation*}  
\end{lemma} 
\begin{proof}
As in \cite{GallayWayne05},  we decompose the eigenfunction with a Fourier transform in the angular variables. Each mode decouples due to the radial symmetry of the coefficients. 
Define $-\Delta c = f$ and write $f = \sum_{n = -\infty}^\infty f_n(r) e^{in\theta}$ and $c = \sum_{n = -\infty}^\infty c_n(r) e^{in\theta}$. 
Written like this we have 
\begin{align*}
Lf - \Lambda_\alpha f & = \sum_{n = -\infty}^\infty e^{in\theta}\left[ \frac{1}{r}(rf_n^\prime)^\prime + \left(\frac{r}{2} - c_\alpha^\prime(r)\right)f_n^\prime + \left(1 + 2G_\alpha(r) - \frac{n^2}{r^2} \right)f_n - G_\alpha^\prime(r) c_n^\prime(r)    \right] \\ 
f_n & = -\frac{1}{r}(rc_n^\prime)^\prime + \frac{n^2}{r^2}c_n. 
\end{align*}
Due to this decoupling and linearity, we may assume without loss of generality that the eigenfunction is supported in only one mode, $n$, and writing $f(r) := f_n(r)$ and $c(r) := c_n(r)$ we have 
\begin{align*}
\frac{1}{r}(rf^\prime)^\prime + \left(\frac{r}{2} - c_\alpha^\prime(r)\right)f^\prime + \left(1 - \mu + 2G_\alpha(r) - \frac{n^2}{r^2} \right)f - G_\alpha^\prime(r) c^\prime(r) & = 0 \\ 
-\frac{1}{r}(rc^\prime)^\prime + \frac{n^2}{r^2}c & = f. 
\end{align*}
As in \cite{GallayWayne05} we will re-write this ODE in a form amenable to a classical result on ODEs (Theorem 8.1, pg 92 \cite{CoddingtonLevinson}, and following Gallay and Wayne, we introduce a new set of variables 
\begin{equation*}
t = \frac{r^2}{4}, \;\; f(r) = g(t).
\end{equation*}
In these new variables we have 
\begin{equation}
g^{\prime\prime}(t) + \left(1 + \frac{1}{t} - \frac{c_\alpha^\prime(2\sqrt{t})}{\sqrt{t}}\right)g^\prime(t) + \left( \frac{1-\mu}{t} - a(t)\right) g(t) = \frac{G_\alpha^\prime(2\sqrt{t})c^\prime(2\sqrt{t})}{t}, \label{eq:InhomogG}
\end{equation} 
where 
\begin{equation*}
a(t) = \frac{n^2}{4t^2} - \frac{2G_\alpha(2\sqrt{t})}{t}. 
\end{equation*}
We first analyze the linearly independent solutions of the homogeneous equation, 
\begin{equation}
\tilde{g}^{\prime\prime}(t) + \left(1 + \frac{1}{t} - \frac{c_\alpha^\prime(2\sqrt{t})}{\sqrt{t}}\right)\tilde{g}^\prime(t) + \left( \frac{1-\mu}{t} - a(t)\right) \tilde{g}(t) = 0. \label{eq:HomogG}
\end{equation} 
The primary way this ODE differs from the corresponding one for the Navier-Stokes equation is the term involving $c_\alpha^\prime$, which is only present for PKS. 
Note that by \eqref{ineq:GradcalphaAsymptotic} 
\begin{equation*}
\lim_{t \rightarrow \infty} 2\sqrt{t} c_\alpha^\prime(2\sqrt{t}) = -\frac{\alpha}{2\pi}.
\end{equation*}
Now, define $x(t) = (\tilde{g}(t), \tilde{g}^\prime(t))$ and re-write \eqref{eq:HomogG} as the system 
\begin{equation*}
x^\prime(t) = (A + V(t) + R(t))x(t),
\end{equation*}
with
\begin{equation*}
A = 
\begin{pmatrix}
0 & 1 \\ 
0 & -1
\end{pmatrix}
, \;\;\; 
V(t) = 
\begin{pmatrix}
0 & 0 \\ 
-\frac{1-\mu}{t} & -\frac{1}{t} + \frac{c_\alpha^\prime(2 \sqrt{t})}{\sqrt{t}}
\end{pmatrix}
, \;\;\; 
R(t) = 
\begin{pmatrix}
0 & 0 \\ 
a(t) & 0
\end{pmatrix}
.
\end{equation*}
Hence by Theorem 8.1, pg 92 of \cite{CoddingtonLevinson} we can get information about the decay of solutions by analyzing the eigenvalues of $A + V(t)$, given by
\begin{equation*}
\lambda_{\pm}(t) = -\frac{1}{2} - \frac{1}{2t} + \frac{c_\alpha^\prime(2\sqrt{t})}{2\sqrt{t}} \pm \frac{1}{2}\sqrt{\left(1 + \frac{1}{t} - \frac{c_\alpha^\prime(2\sqrt{t})}{\sqrt{t}}\right)^{2} - \frac{4 - 4\mu}{t}}. 
\end{equation*}
Hence $\lim_{t \rightarrow \infty} \lambda_+(t) = 0$ and $\lim_{t \rightarrow \infty} \lambda_-(t) = -1$, the eigenvalues of $A$.  
In order to use the result of Coddington and Levinson, we only need to compute the terms which are unbounded in $T$ in the time-integrals
\begin{equation*}
\int_1^T \lambda_\pm(t) dt,  
\end{equation*}
as the integrable terms can be absorbed into the constants. 
Using a Taylor expansion of the square root, dropping the terms which are integrable as $t \rightarrow \infty$ we get
\begin{align*}
\lambda_+ & \approx - \frac{1}{2t} + \frac{c_\alpha^\prime(2\sqrt{t})}{2\sqrt{t}} + \frac{1}{4}\left[ -\frac{2}{t} + \frac{4\mu}{t} - \frac{2c_\alpha^\prime(2\sqrt{t})}{\sqrt{t}} \right], \\ 
\lambda_- & \approx - 1 - \frac{1}{2t} + \frac{c_\alpha^\prime(2\sqrt{t})}{2\sqrt{t}} - \frac{1}{4}\left[ -\frac{2}{t} + \frac{4\mu}{t} - \frac{2c_\alpha^\prime(2\sqrt{t})}{\sqrt{t}}\right]. 
\end{align*}
Simplifying, 
\begin{align*}
\lambda_+ & \approx \frac{\mu - 1}{t},  \\ 
\lambda_- & \approx -1 - \frac{\mu}{t} + \frac{c_\alpha^\prime(2\sqrt{t})}{\sqrt{t}} \approx -1 - \frac{\mu}{t} - \frac{\alpha}{4\pi t}.
\end{align*}
The theorem of Coddington and Levinson then implies that there exists two linearly independent solutions to \eqref{eq:HomogG}, $\phi_1(t)$, $\phi_2(t)$, such that 
\begin{equation*}
\lim_{t \rightarrow \infty} t^{1-\mu} 
\begin{pmatrix} 
\phi_1(t) \\
\phi_1^\prime(t)
\end{pmatrix}
= 
\begin{pmatrix} 
1 \\ 
0
\end{pmatrix}
, \;\;\;\;\;
\lim_{t \rightarrow \infty} t^{\mu + \frac{\alpha}{4\pi}}e^{t} 
\begin{pmatrix} 
\phi_2(t) \\
\phi_2^\prime(t)
\end{pmatrix}
= 
\begin{pmatrix} 
1 \\ 
-1
\end{pmatrix}
. 
\end{equation*}
Returning now to inhomogeneous ODE \eqref{eq:InhomogG} we use the variation of constants formula, as in \cite{GallayWayne05}.
Note that the inhomogeneity here is 
\begin{equation*}
b(t) = \frac{G^\prime_\alpha(2\sqrt{t})c^\prime(2\sqrt{t})}{t}. 
\end{equation*}
Writing the solution $g(t)$ with the variation of constants formula gives
\begin{equation*}
g(t) = A(t)\phi_1(t) + B(t)\phi_2(t), 
\end{equation*}
with
\begin{equation*}
A(t) = A_1 - \int_1^t (W(s))^{-1} b(s)\phi_2(s) ds, \;\;\; B(t) = B_1 + \int_1^t (W(s))^{-1} b(s)\phi_1(s) ds, 
\end{equation*} 
where the Wronskian $W(t) = \phi_1(t)\phi_2^\prime(t) - \phi_2(t)\phi_1^\prime(t) \approx -t^{-1 - \frac{\alpha}{4\pi}}e^{-t}$ as $t \rightarrow \infty$.
Since $f$ is well-localized (since $f \in L^2(m)$) and is average zero, by a variant of \eqref{ineq:GradcalphaAsymptotic} we have
\begin{equation*}
\abs{c^\prime(r)} \lesssim \frac{1}{r^2}, \;\;\; r \rightarrow \infty. 
\end{equation*}
Hence by \eqref{eq:GradGalpha_limit} in Proposition \ref{prop:Galpha}, 
\begin{equation*} 
\abs{b(t)} \lesssim t^{-\frac{3}{2} - \frac{\alpha}{4\pi}}e^{-t}, \;\;\; t \rightarrow \infty, 
\end{equation*}
which implies that
\begin{equation*}
W(t)^{-1}\abs{b(t)} \lesssim t^{-\frac{1}{2}}, \;\;\; t \rightarrow \infty.  
\end{equation*} 
As $t \rightarrow \infty$, 
\begin{equation*}
W(t)^{-1}\abs{b(t)}\abs{\phi_2(t)} \lesssim t^{- \frac{1}{2} - \frac{\alpha}{4\pi} - \mu} e^{-t}, 
\end{equation*}
which is integrable, so $A(t) \rightarrow A_\infty$ for some constant. On the other hand as $t \rightarrow \infty$, 
\begin{equation*}
W(t)^{-1}\abs{b(t)}\abs{\phi_1(t)} \lesssim t^{\mu-\frac{3}{2}},
\end{equation*}
which implies
\begin{equation*}
\abs{B(t)} \lesssim (1+t)^{\mu - \frac{1}{2}}. 
\end{equation*}
By the asymptotic behavior of $\phi_1$, for large $t$ we have
\begin{equation*}
\abs{A(t) \phi_1(t)} \approx A_\infty t^{\mu-1}. 
\end{equation*}
Returning to the original variables shows that $A_\infty \neq 0$ would violate $f \in L^2(m)$ and hence must be zero. 
By the asymptotic decay of $\phi_2(t)$ together with the polynomial bound on $B(t)$ we have the claimed exponential localization. 
\end{proof} 

The lemma shows that any eigenfunction $f \in L_0^2(m)$ with eigenvalue $\textup{Re}\mu > (1-m)/2$ must in fact be in $L^2(G_\alpha^{-1}d\xi)$, and the result follows by Theorem \ref{thm:CD_SpectralGap}.
\end{proof} 

The next step is to prove the decay estimate for the time-dependent linear evolution equation, for which we more or less follow the arguments of section 4.2 in \cite{GallayWayne05}. 
With the convention we are taking, we have that in $L_0^2(m)$, $\sigma_{ess}(S(\tau)) \subset \set{\lambda \in \Complex: \abs{\lambda} \leq e^{\tau(1-m)/2}}$ \cite{GallayWayne02}. 
The following lemma shows that $\mathcal{T}_\alpha(\tau)$ is a compact perturbation of $S(\tau)$, which in turn controls $\sigma_{ess}(\mathcal{T}_\alpha(\tau))$.
\begin{lemma} \label{lem:compact_perturb}
Let $m > 1$. The linear operator $K(\tau) = \mathcal{T}_\alpha(\tau) - S(\tau)$ is compact in $L^2(m)$ for all $\tau > 0$.  
\end{lemma}
\begin{proof}
Let $f_0 \in L^2(m)$ and write
\begin{align*}
K(\tau)f_0 = -\int_0^\tau S(\tau-s) \Lambda_\alpha f(s) ds,
\end{align*}
where $f(\tau) := \mathcal{T}_\alpha(\tau) f_0$. 
Then by \eqref{ineq:gradSDecay},  
\begin{align*}
\norm{K(\tau)f_0}_{L^2(m + 1)} & \leq \int_0^\tau \norm{S(\tau - s)\Lambda_\alpha f(s)}_{L^2(m+1)} ds \\ 
 & \lesssim \int_0^\tau \frac{e^{-\frac{1}{2}(\tau - s)}}{a(\tau - s)^{1/2}}\left( \norm{v^{G_\alpha}f(s)}_{L^2(m+1)} + \norm{G_\alpha v^f(s)}_{L^2(m+1)} \right) ds,
\end{align*}
where $v^f = B \ast f$. To control the first term we use the spatial decay of $v^{G_\alpha}$,  
\begin{equation*}
\norm{v^{G_\alpha} f}_{L^2(m+1)} \leq \norm{v^{G_\alpha}(\xi) \jap{\xi}}_{L^\infty} \norm{f(s)}_{L^2(m)} \lesssim \norm{f(s)}_{L^2(m)}
\end{equation*}
The second term we use the $L^4$ estimate \eqref{ineq:VelLp}, 
\begin{equation*}
\norm{v^f G_\alpha}_{L^2(m+1)} \leq \norm{v^f}_{L^4}\norm{G_\alpha}_{L^{4}(m+1)} \lesssim \norm{f}_{L^{4/3}} \lesssim \norm{f}_{L^2(m)}.  
\end{equation*}
Hence, 
\begin{align*}
\norm{K(\tau)f_0}_{L^2(m + 1)} & \lesssim \int_0^\tau \frac{e^{-\frac{1}{2}(\tau - s)}}{a(\tau - s)^{1/2}}\norm{f(s)}_{L^2(m)} ds. 
\end{align*}
However, since $\mathcal{T}_\alpha(\tau)$ is bounded on $L^2(m)$ we have
\begin{equation*}
\norm{\mathcal{K}_\alpha(\tau)f_0}_{L^2(m+1)} \lesssim_\tau \norm{f_0}_{L^2(m)}. 
\end{equation*}
The estimate \eqref{ineq:TGradientHyper} similarly implies
\begin{equation*}
\norm{\mathcal{K}_\alpha(\tau)f_0}_{H^1(m)} \lesssim \norm{f_0}_{L^2(m)}. 
\end{equation*}
Compactness follows from the Rellich-Khondrashov embedding theorem.  
\end{proof}

We use this to prove 
\begin{proposition}
Let $\nu \in (0,K_\alpha)$ and $f_0 \in L^2_0(m)$ for any $m > 1+2\nu$. Then, 
\begin{equation*}
\norm{T_\alpha(\tau)f_0}_m \lesssim_\nu e^{-\nu \tau}\norm{f_0}_m. 
\end{equation*}
\end{proposition}
\begin{proof} 
Since $\mathcal{T}_\alpha(\tau)$ is a compact perturbation of $S(\tau)$ in $L^2_0(m)$ by Lemma \ref{lem:compact_perturb}
and since $\sigma_{ess}$ is invariant under compact perturbations \cite{EngelNagel}, 
$\sigma_{ess}(\mathcal{T}_\alpha(\tau)) = \sigma_{ess}(S(\tau)) \subset \set{\lambda \in \Complex: \abs{\lambda} \leq e^{\tau(1-m)/2}}$.
Since $\mathcal{T}_\alpha(\tau)$ is bounded in $L^2_0(m)$, it has a non-empty resolvent set (in particular the resolvent must contain some half-plane $\set{\lambda \in \Complex: \textup{Re}\lambda> C}$ for some $C \geq 1$),  
it follows from Theorem 2.1, Chapter XVII.2 \cite{GoldbergEtAl} that 
$\sigma(\mathcal{T}_\alpha(\tau)) \cap \set{\lambda \in \Complex : \abs{\lambda} > e^{\tau(1-m)/2}} \subset \sigma_{disc}(\mathcal{T}_\alpha(\tau))$.  
Since the discrete spectrum is contained in the point spectrum, to prove \eqref{ineq:SpecGapT} it suffices to control $P\sigma(\mathcal{T}_\alpha(\tau))$. 
For this we use the spectral mapping result, Theorem 3.7, Chapter IV \cite{EngelNagel}, which proves $P\sigma(\mathcal{T}_\alpha(\tau)) \setminus \set{0} = e^{\tau P\sigma(L - \Lambda_\alpha)}$. 
By Proposition \ref{prop:SpecGapLm}, if $\lambda \in P\sigma(L-\Lambda_\alpha)$ then, 
\begin{align*} 
\abs{e^{\tau \lambda}} = e^{\tau \textup{Re}(\lambda)} \leq \min(e^{-K_\alpha \tau},e^{\tau(1-m)/2}) .  
\end{align*}
Since everything except the discrete spectrum is contained in the ball of radius $e^{\tau(1-m)/2}$ already, it follows that the spectral radius
\begin{align*} 
r(\mathcal{T}_\alpha(\tau)) \leq \min(e^{\tau(1-m)/2},e^{-K_\alpha \tau}),
\end{align*}  
from which the proposition follows by Proposition 2.2, Chapter IV \cite{EngelNagel}. 
\end{proof} 

\subsubsection{Proof of \textit{(iii)}}
The proof of (iii) is analogous to Proposition 4.6(iii) in \cite{GallagherGallay05}, and one can easily check that the proof 
contained therein can be carried out in our case without any significant changes. 
However, notice that the spectral gap estimate deduced in (ii) is vitally important.  

\subsection{Proof of Proposition \ref{prop:SNProperties}} \label{sec:SNPrp}
\subsubsection{Existence of mild solutions} \label{subsec:SNExistence}
We prove the following general lemma. The assumptions could be weakened but this does not seem necessary.   
\begin{lemma}[Existence] \label{lem:GenLinear}
Suppose $T \in (0,\infty)$, $s \in [0,T)$ and let $v \in C^\infty( (0,T) \times \Real^2)$ be such that for all $q \in (2,\infty]$, $\norm{v(t)}_{q} \lesssim t^{\frac{1}{q} - \frac{1}{2}}$, for all $p \in [1,\infty]$, $\norm{\grad v(t)}_p \lesssim t^{\frac{1}{p} - 1}$ and for all $k \in \Naturals^2$ with $1 \leq \abs{k} \leq k_0$ for some $k_0$ we have $\norm{D^k v(t)}_{2} \lesssim t^{-\frac{\abs{k}}{2}}$ 
which additionally satisfies the tightness condition: for all $\epsilon > 0$ there exists an $R > 0$ such that 
\begin{align} 
\norm{\mathbf{1}_{\abs{x} > R}v(t)}_{4} < \epsilon t^{-1/4}. \label{def:vtight}
\end{align}
Then the linear PDE
\begin{subequations} \label{def:linAdVec}
\begin{align} 
\partial_t w + \grad \cdot (v w) & = \Delta w, \\ 
w(s) & = \mu, 
\end{align}
\end{subequations}
has a mild solution (in the sense of Remark \ref{def:MildSolnLinear}) on $(s,T)$ for all $\mu \in \mathcal{M}(\Real^2)$ which satisfies for all $p \in [1,\infty]$ and $k \in \Naturals^2$ with $1 \leq \abs{k} \leq k_0$, 
\begin{subequations} \label{ineq:linearHyper}
\begin{align} 
\norm{w(t)}_p & \lesssim (t-s)^{\frac{1}{p} - 1}\norm{\mu}_{\mathcal{M}(\Real^2)}, \label{ineq:linearHyperI} \\ 
\norm{D^k w(t)}_{2} & \lesssim (t-s)^{-\frac{1}{2}-\frac{\abs{k}}{2}}\norm{\mu}_{\mathcal{M}(\Real^2)}, \label{ineq:linearHyperII} \\
\limsup_{t \rightarrow s^+} (t-s)^{1-\frac{1}{p}}\norm{w(t)}_p & \lesssim \norm{\mu}_{pp}. \label{ineq:linearHyperLoc}
\end{align} 
\end{subequations}
Moreover, if $\mu \in \mathcal{M}_+(\Real^2)$ then we may take this mild solution to be strictly positive for $t > s$. 
Finally, if $\mu = \alpha \delta$ for $\alpha \in \Real$ then we may take this mild solution to satisfy the following localization: for all $\gamma > 1$, there exists a $c = c(\gamma) > 0$ such that,
\begin{align}
\int e^{\frac{\abs{x}^2}{4\gamma (t-s)}}\abs{w(t,x)} dx \lesssim_\gamma e^{c\left(\sup_{s^\prime \in (s,T)}s^{1/2}\norm{v(s^\prime)}_\infty\right)^2}.  \label{ineq:GaussLoc}
\end{align} 
\end{lemma}
\begin{remark} 
Theorem \ref{thm:Basics} (i) can be adapted to show that \eqref{ineq:linearHyperI} and \eqref{ineq:linearHyperII} hold for all mild solutions. 
However, it is not clear that \eqref{ineq:linearHyperLoc} or \eqref{ineq:GaussLoc} hold for all mild solutions with $\mu = \alpha\delta$.   
\end{remark} 
\begin{proof} 
Since \eqref{def:linAdVec} is linear, we may assume that $\norm{\mu}_{\mathcal{M}(\Real^2)} = 1$ (if $\mu = 0$ then we obviously take $w \equiv 0$ as our mild solution).  
For a standard mollifier $\rho_\epsilon(x) = \epsilon^{-2}\rho(\epsilon^{-1}x)$, we may define the regularized velocity field $v_\epsilon = \rho_\epsilon \ast v$ and
for all $\epsilon > 0$ we define $w_\epsilon$ as the classical solution to 
\begin{align*} 
\partial_t w_\epsilon + \grad \cdot (v_\epsilon w_\epsilon) = \Delta w_\epsilon \\ 
w_\epsilon(s) = \rho_\epsilon \ast \mu.  
\end{align*}
We next deduce some a priori estimates on $w_\epsilon$ to show we may extract a mild solution satisfying the appropriate analogue of Definition \ref{def:MildSolution}.  

By a straightforward contraction mapping argument (for example, one that could be used as the first step of constructing the classical solutions) it follows that for each $\epsilon > 0$ we may find a $t_\epsilon > s$ (uniformly in $s \geq 0$) such that 
$w_\epsilon$ satisfies \eqref{ineq:linearHyper} on $s < t < t_\epsilon$ with a constant which is independent of $\epsilon$ and $s$ (although $t_\epsilon \rightarrow s$ as $\epsilon \rightarrow 0$).  
Now we sketch how to extend $t_\epsilon$ independent of $\epsilon$.  
We re-write the problem in similarity variables, $\xi = x(t-s)^{-1/2}$ and $\tau = \log (t-s)$, $\tau \in (-\infty,\log (T-s)]$ and define 
\begin{align*}
\tilde v_\epsilon(\tau,\xi) & = e^{\tau/2}v_\epsilon(e^{\tau}+s,\xi e^{\tau/2}), \\ 
f_\epsilon(\tau,\xi) & = e^{\tau}w_\epsilon(e^{\tau}+s,\xi e^{\tau/2}), 
\end{align*} 
which together solve 
\begin{align*} 
\partial_\tau f_\epsilon + \grad \cdot (\tilde v_\epsilon f_\epsilon) = Lf_\epsilon. 
\end{align*} 
By the assumptions on $v$, $\tilde v_\epsilon$ satisfies $\norm{\tilde v_\epsilon(\tau)}_q \lesssim 1$ for all $q \in (2,\infty]$
and for all $k \in \Naturals^2$, $1 \leq \abs{k} \leq k_0$ and $\norm{D^k \tilde v_\epsilon}_2 \lesssim 1$ independent of $\epsilon$ and $s$. 
By the change of variables, we have that $f_\epsilon$ is uniformly bounded in $L^p$ for all $p \in [1,\infty]$ for $\tau < \log t_\epsilon$. 
Using standard parabolic regularity theory (for example, the Moser-Alikakos iteration), it follows by the uniform bounds on $\tilde v_\epsilon$ that in fact $f_\epsilon$ is bounded 
uniformly in $L^p$ and $H^k$ up until $\tau = \log (T-s)$ with bounds independent of $\epsilon$ and $s$.  
Upon passing back to the original variables, we have \eqref{ineq:linearHyper}.
By these controls, $\set{w_\epsilon(t)}_{\epsilon > 0}$ is locally precompact in $C_{loc}((s,T); H_{loc}^{k}(\Real^2))$ for all $1 \leq k \leq k_0-\delta$ (for any $\delta>0$).
From here we want to pass to the limit and extract a mild solution to \eqref{def:linAdVec} which satisfies the desired properties. 
Getting \eqref{ineq:linearHyper} follows immediately and \eqref{eq:MildSolutionDef} is not hard.
The property which is not so obvious is that the extracted solution $w(t)$ satisfies $w(t) \rightharpoonup^\star \mu$ as $t \rightarrow s^+$; for this we need a bit of regularity in time. 
Directly from the PDE and the uniform bound in $L^1$, we have the uniform bound in the negative order Sobolev space $W^{-2,1}$: 
\begin{align*} 
\norm{\partial_t w_\epsilon}_{W^{-2,1}} & \leq \norm{w_\epsilon v_\epsilon}_1 + \norm{w_\epsilon}_1  \lesssim \frac{1}{\sqrt{t}}.     
\end{align*} 
As the above is integrable, we have that $w_\epsilon(t)$ is uniformly equi-continuous in time with values in $W^{-2,1}$. 
Hence, as $\epsilon \rightarrow 0$ we may extract a subsequence $\epsilon_n$ such that $w_{\epsilon_n}$ 
converges to a function $w(t)$ in $C([s,T];W^{-2,1}(\Real^2))\cap C_{loc}((s,T);H_{loc}^{k}(\Real^2) \cap L^1(\Real^2))$ for all $1 \leq k \leq k_0-\delta$ (for any $\delta>0$). 
It follows that $w(t)$ is non-negative, satisfies \eqref{ineq:linearHyper} (by lower-semicontinuity) and for all $s^\prime > s$, 
\begin{align} 
w(t) = e^{(t-s^\prime)\Delta} w(s^\prime) - \int_{s^\prime}^t e^{(t-\tilde s)} \grad \cdot (v(\tilde s)w(\tilde s)) d\tilde s, \label{def:wtf}
\end{align}
and moreover for all $t > s$, we may pass to the limit $s^\prime \rightarrow s$ in \eqref{def:wtf} at least in the sense of distributions. 
By the equi-continuity in $W^{-2,1}$ we also know that $w(t) \rightarrow \mu$ in $W^{-2,1}$as $t \rightarrow s$. 
Hence $w(t)$ satisfies the initial value problem \eqref{def:linAdVec} in some sense but in order to satisfy the definition 
of mild solution we need to prove it takes the initial data in the weak$^\star$ topology, for which we need to use the $L^1$ bound and tightness in $L^1$ as $t \rightarrow s^+$. 
To see the latter, multiply \eqref{def:linAdVec} by a smooth, non-negative cut-off function $\chi$ which is equal to one for $\abs{x} > R$ and zero for $\abs{x} \leq R/2$. Define $w^R(t,x) = w(t,x)\chi(x)$ which solves (in a sense analogous to \eqref{def:wtf}), 
\begin{align*}
\partial_t w^R + \grad \cdot (vw^R) = \Delta w^R + \grad \chi \cdot v w + w \Delta \chi - 2\grad\cdot(w\grad \chi). 
\end{align*}  
By Duhamel's formula, followed by \eqref{ineq:HeatLpLqEasy} and \eqref{ineq:HeatLpLq}, H\"older's inequality and \eqref{ineq:linearHyperI}  
\begin{align*} 
\norm{w^R(t)}_1 & \leq \norm{e^{(t-s)\Delta}\mu \chi}_1 + \norm{\int_s^t e^{(t-t^\prime)\Delta}\left[\grad \cdot (vw^R(t^\prime)) + \grad \chi \cdot v w(t^\prime) + w(t^\prime) \Delta \chi - 2\grad\cdot(w(t^\prime)\grad \chi)\right] dt^\prime}_1 \\ 
& \lesssim \norm{\mu \chi}_1 + \int_s^t \frac{1}{(t-t^\prime)^{1/2}}\left(\norm{ \chi vw (t^\prime)}_1 + \frac{1}{R}\norm{w(t^\prime)}_1\right) + \frac{1}{R}\norm{\mathbf{1}_{\abs{x} \geq R/2} v w (t^\prime)}_1 + \frac{1}{R^2}\norm{w(t^\prime)}_1 dt^\prime \\ 
& \lesssim \norm{\mu \chi}_1 + \frac{(t-s)^{1/2}}{R} + \frac{(t-s)}{R^2} + \int_s^t \frac{1}{(t-t^\prime)^{1/2}}\norm{\mathbf{1}_{\abs{x}\geq R/2}v(t^\prime)}_{4} \norm{w (t^\prime)}_{4/3} \\ & \quad\quad + \int_s^t\frac{1}{R}\norm{\mathbf{1}_{\abs{x} \geq R/2} v(t^\prime)}_{4} \norm{w (t^\prime)}_{4/3} dt^\prime \\ 
& \lesssim \norm{\mu \chi}_1 + \frac{(t-s)^{1/2}}{R} + \frac{(t-s)}{R^2} + \int_s^t \frac{1}{(t-t^\prime)^{1/2}(t^\prime)^{1/4}}\norm{\mathbf{1}_{\abs{x}\geq R/2}v(t^\prime)}_{4} + \frac{1}{R(t^\prime)^{1/2}} dt^\prime \\ 
 & \lesssim \norm{\mu \chi}_1 + \frac{(t-s)^{1/2}}{R} + \frac{(t-s)}{R^2} + \int_s^t \frac{1}{(t-t^\prime)^{1/2}(t^\prime)^{1/4}}\norm{\mathbf{1}_{\abs{x}\geq R/2}v(t^\prime)}_{4}dt^\prime.  
\end{align*} 
For the last term we use \eqref{def:vtight}, which implies that we can choose $R$ sufficiently large so that: 
\begin{align*} 
\norm{w^R(t)}_1  
  & \lesssim \norm{\mu \chi}_1 + \frac{(t-s)^{1/2}}{R} + \frac{(t-s)}{R^2} + \epsilon. 
\end{align*}  
Then choosing $R$ such that $\mu(\Real^2 \setminus B_R) < \epsilon$ and $(t-s) < \epsilon R^2$ it follows that 
$\norm{w^R(t)}_1 \lesssim \epsilon$, which implies by definition that $w(t)$ is tight in $\mathcal{M}(\Real^2)$ as $t \rightarrow s^+$. 
As $w(t)$ is tight in $\mathcal{M}(\Real^2)$, uniformly bounded in total variation and converges in the sense of distributions to $\mu$ as $t \rightarrow s$ 
it follows that the convergence also holds in the weak$^\star$ topology (see e.g. \S\ref{sec:finalStep}). Therefore, by definition $w(t)$ is a mild solution. 

The Gaussian localization \eqref{ineq:GaussLoc} comes from adapting (2.6) in \cite{CarlenLoss92} which follows by identifying non-negative solutions to \eqref{def:linAdVec} as the law for the corresponding SDE, justified for the mild solution constructed above by approximation. 
\end{proof} 

\subsubsection{$N=1$: One Concentration}
We now prove Proposition \ref{prop:SNProperties} in the case of one concentration, which requires some known spectral gap properties of  Fokker-Planck operators.
We study, for fixed $\alpha \in (0,8\pi)$, the PDE 
\begin{equation}
\partial_t w + \grad \cdot \left( \frac{1}{\sqrt{t}}\grad c_\alpha\left( \frac{x}{\sqrt{t}} \right) w \right) = \Delta w.  \label{eq:propN1}
\end{equation}
In self-similar coordinates, $\xi = x t^{-1/2}$, $\tau = \log t$, $t w(t,x) = f(\tau,\xi)$, we may re-write this as
\begin{equation}
\partial_\tau f + \grad \cdot \left(f \grad c_\alpha \right) = Lf. \label{eq:FKP}
\end{equation}
First we prove that the mild solution constructed in Lemma \ref{lem:GenLinear} is unique, which requires a sequence of steps since the PDE is critically singular in the sense that a contraction mapping argument can only be applied if $\alpha$ is small.  
By linearity, it suffices to show that an arbitrary mild solution with zero initial data remains zero for positive times. 
The main idea is to `break scaling' and improve the sense in which the mild solution converges to zero as $t \rightarrow 0^+$ and then apply a straightforward duality argument.
In what follows, denote $v(t,x) = \frac{1}{\sqrt{t}}\grad c_\alpha\left(\frac{x}{\sqrt{t}}\right)$. 
We remark that our proof does not explicitly use any spectral properties of \eqref{eq:FKP} however it does rely on the monotonicity $v(t,x)\cdot x \leq 0$ (compare with \cite{JiaSverak13}). 

We begin by proving all mild solutions with zero initial data converge strongly to zero as $t \rightarrow 0^+$ everywhere except possibly near the origin, where the singularity in the velocity field is located. 

\begin{lemma} \label{lem:easyLoc}
Let $w(t,x)$ be a mild solution to \eqref{eq:propN1} with zero initial data. Then for all $\delta > 0$, 
\begin{align*} 
\lim_{t \rightarrow 0^+}\int_{\abs{x} > \delta}\abs{w(t,x)} dx = 0.
\end{align*}
\end{lemma}
\begin{proof} 
Let $\chi(x)$ be a smooth, non-negative function which satisfies $\chi(x) = 1$ for $\abs{x} \geq \delta$ and $\chi(x) = 0$ for $\abs{x} < \delta/2$.  
Define $w^\delta(t,x) = \chi w(t,x)$, which is a mild solution to the PDE 
\begin{align*} 
\partial_t w^\delta + \grad \cdot (v w^\delta) & = \Delta w^\delta + \grad \chi \cdot v w - w\Delta \chi  - 2\grad \chi \cdot \grad w \\
& = \Delta w^\delta + \grad \chi \cdot v w + w\Delta \chi  - 2 \grad\cdot (w \grad \chi).  
\end{align*} 
Using Duhamel's formula and the zero initial data assumption,
\begin{align*} 
w^\delta(t) = -\int_0^t e^{(t-s)\Delta}\left[\grad \cdot (\chi v(s) w(s)) + 2 \grad\cdot (w \grad \chi) - \grad \chi \cdot v w - w\Delta \chi\right] ds.
\end{align*} 
Note that for $t \lesssim \delta^2$, $\norm{\chi v(t)}_\infty \lesssim \delta^{-1}$. 
Hence, by the a priori estimates on $w$ from \eqref{ineq:linearHyper}, for $t$ sufficiently small (using also \eqref{ineq:HeatLpLqEasy} and \eqref{ineq:HeatLpLq})
\begin{align*} 
\norm{w^\delta(t)}_1 & \leq \int_0^t\frac{1}{(t-s)^{1/2}}\left[\norm{\chi v(s)}_{\infty}\norm{w(s)}_{1} + \frac{1}{\delta}\norm{w}_1\right] + \norm{\grad \chi \cdot v(s)}_{\infty}\norm{w(s)}_{1} + \frac{1}{\delta^2}\norm{w(s)}_{1} ds \\ 
& \lesssim \frac{t}{\delta^2} + \int_0^t\frac{1}{\delta (t-s)^{1/2}} ds \lesssim \frac{t}{\delta^2} + \frac{\sqrt{t}}{\delta}, 
\end{align*} 
from which we conclude. 
\end{proof}

The next step is to localize the mild solutions to a natural parabolic region in space-time which shows that any spurious information created at the origin cannot propagate away unphysically fast.
This is quantified by compactness of solutions to \eqref{eq:FKP} in the sense of the following lemma.
The hypothesis is a direct consequence of Lemma \ref{lem:easyLoc}.
\begin{lemma} \label{lem:compactL1}
Let $f(\tau,\xi)$ be a smooth, uniformly bounded $L^1$ solution to \eqref{eq:FKP} on $\tau \in (-\infty,T]$ such that for all $\epsilon > 0$ and  $\delta > 0$ there is a $\tau_{\epsilon,\delta}$ such that for $\tau < \tau_{\epsilon,\delta}$ the following holds: 
\begin{align} 
\int_{e^{\tau/2}\abs{\xi} > \delta} \abs{f(\tau,\xi)} d\xi < \epsilon. \label{lem:compcHypo}
\end{align}   
Then $f(\tau,\xi)$ is tight in $L^1$ in the sense that for all $\epsilon > 0$ there is an $R > 0$ such that the following holds for all $\tau < T$: 
\begin{align*} 
\int_{\abs{\xi} > R} \abs{f(\tau,\xi)} d\xi < \epsilon. 
\end{align*}   
\end{lemma} 
\begin{proof} 
By considering the positive and negative parts separately we see that if $\bar f$ solves \eqref{eq:FKP} with data
$\bar f(\tau^\prime) = \abs{f(\tau^\prime)}$ then for $\tau > \tau^\prime$, $\abs{f(\tau)} \leq \bar f(\tau)$ and hence we may assume without loss of generality that $f$ is non-negative. 

We will eventually compare solutions to \eqref{eq:FKP} with the corresponding PDE in the case $\alpha = 0$, for which there is an explicit expression for the solution. Indeed if $f(\tau)$ solves \eqref{eq:FKP} with $\alpha = 0$ then,  
\begin{align} 
f(\tau,\xi) = \frac{1}{4\pi a(\tau-\tau^\prime)} \int_\zeta e^{-\frac{1}{4a(\tau - \tau^\prime)}\abs{\xi - e^{\frac{\tau^\prime - \tau}{2}}\zeta}^2} f(\tau^\prime,\zeta) d\zeta. \label{def:explicit}
\end{align}
where as above $a(\tau-\tau^\prime) = 1 - e^{\tau^\prime-\tau}$. 
Let $\epsilon > 0$ and $\tau < T$ be arbitrary and choose $R > \epsilon^{-1/2}$.
Now fix $\delta < e^{\tau/2}R/4$ and let $\tau^\prime$ be sufficiently small such that 
$a(\tau-\tau^\prime) > \frac{1}{2}$ and (using \eqref{lem:compcHypo}), 
\begin{align} 
\int_{e^{\tau^\prime/2}\abs{\xi} > \delta} f(\tau^\prime,\xi) d\xi < \epsilon. \label{ineq:loc_tauprime}
\end{align}
We use the SDE representation of \eqref{eq:FKP} (see e.g. \cite{Gardiner,Oksendal}) and compare the solution with $\alpha > 0$ to the case $\alpha = 0$.
Define $\bar X_t$ as the unique non-anticipating solution to the Ito diffusion 
\begin{align} 
d\bar X_t & = -\frac{1}{2}\bar X_t dt + \grad c_\alpha(\bar X_t)dt + dW_t, \label{def:ItoBar}
\end{align}
with $\bar X_{\tau^\prime}$ distributed by the density $f(\tau^\prime,\xi)/\norm{f(\tau^\prime)}_1$.  
By \eqref{ineq:loc_tauprime}, 
\begin{align} 
\mathbb P\left( \abs{\bar X_\tau} > R^2\right) & =  \mathbb P\left( \abs{\bar X_\tau} > R^2 | \abs{\bar X_{\tau^\prime}} \leq \delta e^{-\tau^\prime/2}\right)\mathbb P\left(\abs{\bar X_{\tau^\prime}} \leq \delta e^{-\tau^\prime/2}\right) \\ & \quad +\mathbb P\left( \abs{\bar X_\tau} > R^2 | \abs{\bar X_{\tau^\prime}} > \delta e^{-\tau^\prime/2}\right) \mathbb P\left(\abs{\bar X_{\tau^\prime}} > \delta e^{-\tau^\prime/2}\right) \nonumber \\ 
& \leq \mathbb P\left( \abs{\bar X_\tau} > R^2 | \abs{\bar X_{\tau^\prime}} \leq \delta e^{-\tau^\prime/2}\right) + \mathbb P\left(\abs{\bar X_{\tau^\prime}} > \delta e^{-\tau^\prime/2}\right) \nonumber \\ 
& \lesssim \mathbb P\left( \abs{\bar X_\tau} > R^2 | \abs{\bar X_{\tau^\prime}} \leq \delta e^{-\tau^\prime/2}\right) + \epsilon. \label{ineq:Prob} 
\end{align}
Now define $\tilde X_t$ as the unique non-anticipating solution to the Ito diffusion \eqref{def:ItoBar} with  $\tilde X_{\tau^\prime}$ distributed by the law $\mathbf{1}_{\abs{\xi} \leq \delta e^{-\tau^\prime/2}} f(\tau^\prime,\xi)/\norm{\mathbf{1}_{\abs{\xi} \leq \delta e^{-\tau^\prime/2}} f(\tau^\prime)}_1$. 
By Markov's inequality, 
\begin{align} 
\mathbb P\left( \abs{\bar X_t} > R^2 | \abs{\bar X_{\tau^\prime}} \leq \delta e^{-\tau^\prime/2}\right) & = \mathbb P\left( \abs{\tilde X_t} > R^2\right) \leq \frac{1}{R^4}\mathbb E\abs{\tilde X_t}^2. \label{ineq:Markovs}
\end{align} 
It remains to estimate the expectation. 
Define $X_t$ to be the unique non-anticipating solution to the Ito diffusion
\begin{align*}
dX_t & = -\frac{1}{2}X_t dt + dW_t,
\end{align*}
with $X_{\tau^\prime}$ distributed by $\mathbf{1}_{\abs{\xi} \leq \delta e^{-\tau^\prime/2}} f(\tau^\prime,\xi)/\norm{\mathbf{1}_{\abs{\xi} \leq \delta e^{-\tau^\prime/2}} f(\tau^\prime)}_1$.
Consider the third Ito diffusion 
\begin{align*} 
Y_t = \abs{\tilde X_t}^2 - \abs{X_t}^2, 
\end{align*} 
with initial density $Y_{\tau^\prime} = \delta$ (the $\delta$ mass at the origin) to compare the two SDEs path-wise. 
By Ito's formula, $Y_t$ satisfies 
\begin{align*} 
dY_t = -Y_t dt +  2\grad c_\alpha(\tilde X_t)\cdot \tilde X_t dt + 2 \left(\tilde X_t - X_t\right)\cdot dW_t. 
\end{align*} 
Writing $y_t = Y_t e^{t}$ and applying the Ito formula again,
\begin{align*} 
y_t = 2\int_{\tau^\prime}^t e^s \grad c_\alpha(\tilde X_s)\cdot \tilde X_s ds +  2 \int_{\tau^\prime}^t e^s\left(\tilde X_s - X_s\right)\cdot dW_s,
\end{align*} 
from which it follows 
\begin{align*} 
\mathbb E y_t & = 2 \mathbb E\int_{\tau^\prime}^t e^s \grad c_\alpha(\tilde X_s)\cdot \tilde X_s ds +  2 \mathbb E \int_{\tau^\prime}^t e^s\left(\tilde X_s - X_s\right)\cdot dW_s \leq 0; 
\end{align*} 
the first term is negative due to the fact that $\grad c_\alpha(x) \cdot x \leq 0$ and the second expectation is zero due to the mean value formula for Ito integrals with non-anticipating integrands. 
Finally, from the definition of $Y_t$ it follows (since $X_{\tau^\prime}$ and $\tilde X_{\tau^\prime}$ are identically distributed), 
\begin{align} 
\mathbb E \abs{\tilde X_t}^2 \leq \mathbb E \abs{X_t}^2. \label{ineq:Markovs2}
\end{align} 
To compute the expectation on the RHS of \eqref{ineq:Markovs2},  we use the explicit formula \eqref{def:explicit} for solutions to \eqref{eq:FKP} in the case $\alpha = 0$ (using also \eqref{ineq:loc_tauprime} and our choice of $\tau^\prime$), 
\begin{align*} 
\mathbb E \abs{X_\tau}^2  &  \leq  \frac{1}{2\pi}\int_\xi \abs{\xi}^2 \int_\zeta e^{-\frac{1}{4}\abs{\xi - e^{\frac{\tau^\prime - \tau}{2}}\zeta}^2} \frac{f(\tau^\prime,\zeta)\mathbf{1}_{\abs{\zeta} \leq \delta e^{-\tau^\prime/2}}}{\norm{f(\tau^\prime,\cdot)\mathbf{1}_{\abs{\cdot} \leq \delta e^{-\tau^\prime/2}}}_1}d\zeta d\xi \\ 
& \lesssim  \frac{1}{1-\epsilon}\int_\xi \abs{\xi}^2 \int_\zeta e^{-\frac{1}{4}\abs{\xi - e^{\frac{\tau^\prime - \tau}{2}}\zeta}^2} f(\tau^\prime,\zeta)\mathbf{1}_{\abs{\zeta} \leq \delta e^{-\tau^\prime/2}} d\zeta d\xi \\ 
& \lesssim \int_\xi \abs{\xi}^2 \int_\zeta e^{-\frac{1}{4}\abs{\xi - e^{\frac{\tau^\prime - \tau}{2}}\zeta}^2} f(\tau^\prime,\zeta)\mathbf{1}_{\abs{\zeta} \leq \delta e^{-\tau^\prime/2}} \left( \mathbf{1}_{\abs{\xi - e^{\frac{\tau^\prime - \tau}{2}}\zeta} < \frac{\abs{\xi}}{2}} + \mathbf{1}_{\abs{\xi - e^{\frac{\tau^\prime - \tau}{2}}\zeta} \geq \frac{\abs{\xi}}{2}} \right) d\zeta d\xi \\ 
& \lesssim  \int_{\xi}\abs{\xi}^2 e^{-\frac{\abs{\xi}^2}{16}} d\xi + \int_\xi \abs{\xi}^2 \int_\zeta e^{-\frac{1}{2}\abs{\xi - e^{\frac{\tau^\prime - \tau}{2}}\zeta}^2} f(\tau^\prime,\zeta)\mathbf{1}_{\abs{\zeta} \leq \delta e^{-\tau^\prime/2}} \mathbf{1}_{\abs{\xi - e^{\frac{\tau^\prime - \tau}{2}}\zeta} < \frac{\abs{\xi}}{2}} d\zeta d\xi \\ 
& \lesssim 1 + \int_{\abs{\xi} > R} \abs{\xi}^2 \int_\zeta e^{-\frac{1}{2}\abs{\xi - e^{\frac{\tau^\prime - \tau}{2}}\zeta}^2} f(\tau^\prime,\zeta)\mathbf{1}_{\abs{\zeta} \leq \delta e^{-\tau^\prime/2}} \mathbf{1}_{\abs{\xi - e^{\frac{\tau^\prime - \tau}{2}}\zeta} < \frac{\abs{\xi}}{2}} d\zeta d\xi \\ 
& \quad + \int_{\abs{\xi} \leq R} \abs{\xi}^2 \int_\zeta e^{-\frac{1}{2}\abs{\xi - e^{\frac{\tau^\prime - \tau}{2}}\zeta}^2} f(\tau^\prime,\zeta)\mathbf{1}_{\abs{\zeta} \leq \delta e^{-\tau^\prime/2}} \mathbf{1}_{\abs{\xi - e^{\frac{\tau^\prime - \tau}{2}}\zeta} < \frac{\abs{\xi}}{2}} d\zeta d\xi.  
\end{align*} 
In fact the second term vanishes. Indeed on the support of the integrand, 
\begin{align*} 
e^{\frac{\tau - \tau^\prime}{2}} \abs{\xi} -\delta e^{-\tau^\prime/2}    < e^{\frac{\tau - \tau^\prime}{2}} \abs{\xi} - \abs{\zeta}   < \abs{e^{\frac{\tau - \tau^\prime}{2}} \xi - \zeta} < \frac{1}{2}e^{\frac{\tau - \tau^\prime}{2}} \abs{\xi}, 
\end{align*} 
which implies that 
\begin{align*} 
\abs{\xi} e^{\frac{\tau - \tau^\prime}{2}} < 2\delta e^{-\tau^\prime/2}. 
\end{align*}
However, since $\abs{\xi} > R$ and $\delta <  Re^{\tau/2}/4$ on the support of the integrand, this is a contradiction and in fact the support must vanish. 
Therefore, 
\begin{align*}
\mathbb E \abs{X_t}^2 & \lesssim 1 + \int_{\abs{\xi} \leq R} \abs{\xi}^2 \int_\zeta e^{-\frac{1}{2}\abs{\xi - e^{\frac{\tau^\prime - \tau}{2}}\zeta}^2} f(\tau^\prime,\zeta)\mathbf{1}_{\abs{\zeta} \leq \delta e^{-\tau^\prime/2}} \mathbf{1}_{\abs{\xi - e^{\frac{\tau^\prime - \tau}{2}}\zeta} < \frac{\abs{\xi}}{2}} d\zeta d\xi \\ 
& \lesssim 1 + R^2. 
\end{align*} 
Therefore with \eqref{ineq:Prob}, \eqref{ineq:Markovs} and \eqref{ineq:Markovs2} together with our choice of $R$, this implies 
\begin{align*} 
\mathbb P( \abs{\bar X_t} > R^2) \lesssim \frac{1}{R^2} + \epsilon \lesssim \epsilon.  
\end{align*} 
\end{proof}

Lemma \ref{lem:compactL1} provides the necessary compactness to improve Lemma \ref{lem:easyLoc} and prove that solutions with zero initial data converge to zero strongly in $L^1$ as $t \searrow 0$. 
The main tool is an $L^1$ decay result for advection-diffusion equations of Carlen and Loss \cite{CarlenLoss92} (also proved using stochastic techniques). 
\begin{lemma} \label{lem:cont}
All mild solutions to \eqref{eq:propN1} with zero initial data satisfy  %for all $R > 0$, 
\begin{align*} 
\lim_{t \rightarrow 0^+} \int \abs{w(t)} dx = 0. 
\end{align*}
\end{lemma} 
\begin{proof} 
Define the re-scaled solution 
\begin{align*} 
f(\tau,\xi) = e^{\tau}w(e^\tau, \xi e^{\tau/2}), 
\end{align*}
which solves \eqref{eq:FKP}. 
By adapting the arguments of Lemma \ref{lem:precompact} it follows that $\set{f(\tau)}_{\tau \in (-\infty,0]}$ is precompact in $C_{loc}((-\infty,0);H_{loc}^k(\Real^2))$. 
Moreover, by Lemma \ref{lem:compactL1} $\set{f(\tau)}_{\tau \in (-\infty,0]}$ is tight in $L^1$ as $\tau \rightarrow -\infty$. 
Therefore, for any sequence $\tau_k \rightarrow - \infty$, we may extract a subsequence (not relabeled) such that $f(\tau + \tau_k)$ converges strongly in $C((-\infty,0);L^1(\Real^2))$ and $C_{loc}((-\infty,0);H_{loc}^k(\Real^2))$ to some function $g(\tau)$ which solves \eqref{eq:FKP}.  
Denote the set of all limits of this type as $\mathcal{A}$. 
By weak lower semicontinuity, $\mathcal{A}$ consists of smooth, ancient solutions of \eqref{eq:FKP} which are uniformly bounded in $L^1 \cap L^\infty \cap H^k$ and by Lemma \ref{lem:compactL1} it follows that these solutions are precompact in $L^1$.
Next we claim that $\int g(\tau,\xi) d\xi = 0$ for all $g \in \mathcal{A}$. 
Suppose for contradiction that $a = \limsup_{\tau \rightarrow -\infty} \abs{\int f(\tau,\xi) d\xi} > 0$.
Let $\phi\in C_0$ be a bounded, non-negative continuous function with $\phi(x) = 1$ for $\abs{x} < 1$. 
Let $\tau_k \rightarrow -\infty$ be such that 
\begin{align*} 
\lim_{\tau_k \rightarrow -\infty} \abs{\int f(\tau_k,\xi) d\xi} = a.  
\end{align*} 
Now, 
\begin{align} 
\abs{\int w(e^{\tau_{k}},x) \phi(x) dx} & = \abs{\int_\xi f(\tau_{k},\xi)\phi\left(\xi e^{\tau_{k}/2}\right) d\xi} \nonumber \\ & \geq \abs{\int_\xi f(\tau_{k},\xi) d\xi} - \abs{\int_\xi f(\tau_{k},\xi)\left[1-\phi\left(\xi e^{\tau_{k}/2}\right)\right] d\xi}. \label{ineq:wphi}
\end{align}
Then by the support of the second integrand and Lemma \ref{lem:easyLoc}, for $k$ sufficiently large we can ensure 
\begin{align*}
\abs{\int_\xi f(\tau_{k},\xi)\left[1-\phi\left(\xi e^{\tau_{k}/2}\right)\right] d\xi} & \leq \norm{\phi}_\infty\int_{\abs{\xi} \geq e^{-\tau_{k}/2}} \abs{f(\tau_{k},\xi)} d\xi < a/4.  
\end{align*}  
It follows from \eqref{ineq:wphi} that for $k$ sufficiently large, 
\begin{align*} 
\abs{\int w(e^{\tau_{k}},x) \phi(x) dx} \geq \frac{3a}{4}, 
\end{align*} 
which contradicts the assumption on the initial data $\lim_{t \rightarrow 0^+}\int w(t,x) \phi(x) dx = 0$. 
Hence it follows that $\lim_{\tau \rightarrow -\infty}\abs{\int f(\tau,\xi) d\xi} = 0$, from which we have $\int g(\tau,\xi) d\xi = 0$ for all $g \in \mathcal{A}$. 

We have now shown that $g \in \mathcal{A}$ are ancient, mean-zero, uniformly bounded, $L^1$ precompact solutions and we would like to conclude that the only such solution is in fact $g \equiv 0$.  
The spectral gap \eqref{ineq:FKPSpecGap} does not quite apply since we do not have such strong control on the spatial decay of $g$. 
However, the proof of Theorem 7 in \cite{CarlenLoss92} only requires solutions to be tight in $L^1$, mean-zero and bounded.  
Hence we may conclude that $g \equiv 0$ for all $g \in \mathcal{A}$, which implies that $f(\tau)$ converges to zero in $L^1$ as $\tau \rightarrow -\infty$, which completes the lemma upon changing variables back.  
\end{proof}

Using the preceding lemma we may now use a standard duality argument to show that the mild solution is unique.
\begin{lemma} \label{lem:dual}
The only mild solution to \eqref{eq:propN1} with zero initial data is $w \equiv 0$.
\end{lemma} 
\begin{proof} 
Let $F \in C^\infty_0( (0,T) \times \Real^2)$ be arbitrary and let $\phi$ solve 
\begin{align*}
-\partial_t \phi  - \Delta \phi - v\cdot \grad \phi & = F, \\ 
\phi(T) & = 0. 
\end{align*} 
By the maximum principle, it follows that $\norm{\phi}_\infty \leq \int_0^T\norm{F(\tau)}_\infty d\tau$. 
Then for all $\epsilon > 0$, 
\begin{align*} 
\abs{\int_\epsilon^T \int w F dx dt} &  = \abs{\int_\epsilon^T \int w \left( -\partial_t \phi  - \Delta \phi - v\cdot \grad \phi \right)  dx dt} \\
& \leq \abs{\int w(\epsilon)\phi(\epsilon) dx} + \abs{\int_\epsilon^T \left(\partial_t w  -\Delta w + \grad \cdot (v w)\right)\phi dx dt} \\
& = \int \abs{w(\epsilon)\phi(\epsilon)} dx \leq \norm{\phi(\epsilon)}_\infty \int \abs{w(\epsilon)} dx.  
\end{align*} 
Taking $\epsilon \rightarrow 0$ and applying Lemma \ref{lem:cont} verifies that $\int_0^T \int w F dx = 0$. 
Since $F$ is arbitrary it follows that $w \equiv 0$. 
\end{proof}
From the uniqueness implied by Lemma \ref{lem:dual} and the properties of the solution constructed in Lemma \ref{lem:GenLinear}, 
we have that $S_1(t,s)$ is well defined and satisfies \eqref{ineq:SNHypercon} and \eqref{ineq:SNpp}. 
To complete the proof of Proposition \ref{prop:SNProperties} in the case $N=1$ we also have to justify that $S_1(t,s)$ is weak$^\star$ continuous in the sense of Remark \ref{def:weakstarcont}. 
Let $t_n,s_n$ and $\mu_n$ be as given there and define $w_n(t) = S_1(t+s_n,s_n)\mu_n$ which is a mild solution to
\begin{subequations} \label{def:wNSN}
\begin{align} 
\partial_t w_n(t) + \grad\cdot \left(w_n(t) \frac{1}{\sqrt{t + s_n}} v^{G_\alpha}\left(\frac{x}{\sqrt{t + s_n}}\right) \right) & = \Delta w_n(t) \\ 
w_n(0) = \mu_n. 
\end{align} 
\end{subequations}
As in the proof of Lemma \ref{lem:GenLinear} above, $w_n(t)$  satisfies \eqref{ineq:linearHyper} uniformly in $n$
and is uniformly equi-continuous in $W^{-2,1}$. 
Therefore, we may extract a subsequence $w_{n_j}(t)$ which converges in $C([0,T-\bar{s});W^{-2,1}(\Real^2)) \cap C_{loc}((0,T-\bar{s}); H_{loc}^k(\Real^2)\cap L^1(\Real^2))$ for all $k < \infty$ to some limit $w(t)$.
It now suffices to show that $w(t)$ solves \eqref{def:wNSN} with $s_{n_j}$ replaced by $\bar s$ and $w(0) = \mu$ (as a mild solution in the sense of Remark \ref{def:MildSolnLinear}). 
To see convergence of the Duhamel integral it suffices to apply \eqref{ineq:HeatLpLq} and note
\begin{align*}
\lim_{j \rightarrow \infty} t^{1/4}\norm{\frac{1}{\sqrt{t + s_{n_j}}} v^{G_\alpha}\left(\frac{x}{\sqrt{t + s_n}}\right) - \frac{1}{\sqrt{t + \bar s}} v^{G_\alpha}\left(\frac{x}{\sqrt{t + \bar s}}\right)}_4 & = 0,  \\ 
\lim_{j \rightarrow \infty} t^{1/4}\norm{w_{n_j}(t) - w(t)}_4 & = 0. 
\end{align*} 
By the continuity, $w(t) \rightarrow \mu$ in $W^{-2,1}$ as $t \rightarrow 0^+$. 
As in the proof of Lemma \ref{lem:GenLinear}, the uniform bound on the total variation and tightness in $\mathcal{M}(\Real^2)$ (one can use essentially the same argument as is used in Lemma \ref{lem:GenLinear}) imply that this convergence is also in the weak$^\star$ topology.  
Similarly, by uniform equi-continuity, we have $w_{n_j}(t_{n_j}) \rightarrow w(\bar t)$ in $W^{-2,1}$, which is improved to weak$^\star$ convergence by the total variation bound and tightness. 
This completes the proof of Proposition \ref{prop:SNProperties} (i) and (ii).

The remainder of the section is devoted to the proof of Proposition \ref{prop:SNProperties} (iii).
We proceed similar to Appendix 6.2 of \cite{GallagherGallay05}, however, the core of our argument uses the theory on the spectral gap of linear Fokker-Planck operators, due to the gradient nature of the nonlocal velocity law.    
As in Appendix 6.2 of \cite{GallagherGallay05}, it is easily seen by applying \eqref{ineq:SNHypercon} that it suffices to prove \eqref{ineq:SNgrad} for $p= 1$, so we concentrate on this.  

The PDE \eqref{eq:FKP}, for $\tau > 0$, is  a linear Fokker-Planck operator with confining potential 
\begin{equation*}
A(\xi) := \frac{1}{4}\abs{\xi}^2 - c_\alpha(\xi), 
\end{equation*}
and it is known that operators of this form admit a spectral gap in the appropriate weighted space, in this case $L^2(G_\alpha^{-1} d\xi)$. 
The easiest way to see this is to re-write \eqref{eq:FKP} as a self-adjoint Schr\"odinger operator for the variable $z = f G_{\alpha}^{-1/2}$ and apply a classical result that such operators has a pure discrete spectrum under certain conditions which are satisfied here (Theorem XIII.67 \cite{ReedSimonIV}). 
See e.g. \cite{ArnoldMarkowichEtAl01} for a more detailed explanation.
The existence of a spectral gap implies that there exists some $\lambda_\alpha > 0$ such that for all $f_0 \in L^2(G_\alpha^{-1}d\xi)$ with $\int f_0 d\xi = 0$,
\begin{equation}
\norm{f(\tau)}_{L^2(G_\alpha^{-1}dx)} \lesssim e^{-\lambda_\alpha \tau} \norm{f_0}_{L^2(G_\alpha^{-1}d\xi)}, 
 \quad \hbox{for} \, \tau \geq  0. \label{ineq:FKPSpecGap}
\end{equation}
We may assume without loss of generality that $\lambda_\alpha \in (0, 1/2)$.
We will show that \eqref{ineq:FKPSpecGap} ultimately implies
\begin{equation}
\norm{S_1(\tau)\grad_\xi f_0}_1 \lesssim \frac{e^{(-\lambda_\alpha + \gamma)\tau}}{a(\tau)^{1/2}}\norm{f_0}_1, \label{ineq:ssS1grad}
\end{equation}
for any $\gamma > 0$ sufficiently small.
When we transform back into physical coordinates, \eqref{ineq:ssS1grad} becomes
\begin{equation*}
\norm{S_1(t,s)\grad_x w}_1 \lesssim \frac{1}{(t-s)^{1/2}}\left[\frac{t}{s}\right]^{\gamma + 1/2 - \lambda_\alpha}  
 \norm{w}_1, 
\end{equation*}
noting carefully that we are taking $\tau = \log t - \log s$. 
To prove \eqref{ineq:ssS1grad} we proceed analogously to \cite{GallagherGallay05}. 
We define the Banach space $X$,
\begin{equation*}
X = \set{f \in L^1(\Real^2): f = \partial_{x_1} g_1 + \partial_{x_2}g_2, \;\; g_1,g_2 \in L^1(\Real^2)}, 
\end{equation*} 
equipped with the norm 
\begin{equation*}
\norm{f}_X = \norm{f}_{1} + \inf\set{\norm{g_1}_1 + \norm{g_2}_1 : f = \partial_{x_1}g_1 + \partial_{x_2}g_2}. 
\end{equation*}
As in \cite{GallagherGallay05}, consider the auxiliary equation for $g = (g_1,g_2)$
\begin{equation}
\partial_\tau g + \grad c_\alpha \grad \cdot g = \left( L - \frac{1}{2}\right)g, 
\end{equation}
and denote the associated linear propagator by $T_1(\tau)$, which is related to $S_1$ via
\begin{equation}
\grad \cdot ( T_1(\tau) g) = S_1(\tau) \grad \cdot g,  \label{eq:T1S1Relate}
\end{equation}
for $g \in (H^1(\Real^2))^2$.
A contraction mapping argument similar to Lemma 6.4 in \cite{GallagherGallay05} shows the following. 
\begin{lemma} \label{lem:T1Contract} 
$T_1(\tau)$ defines a strongly continuous semigroup on $L^1(\Real^2)^2$ and there exists some $\tau_0 > 0$ such that for all $g\in L^1(\Real^2)^2$, 
\begin{equation}
\norm{T_1(\tau)g}_1 \lesssim \norm{g}_1, \;\;\;\;\; \norm{\grad T_1(\tau) g}_1 \lesssim \frac{1}{a(\tau)^{1/2}}\norm{g}_1, \;\;\; \tau \in (0,\tau_0].
\end{equation}
\end{lemma}
By \eqref{eq:T1S1Relate}, Lemma \ref{lem:T1Contract} implies that for $\tau \in (0,\tau_0)$, 
\begin{equation}
\norm{S_1(\tau) \grad \cdot f}_{X} \lesssim \frac{1}{a(\tau)^{1/2}}\norm{f}_1. \label{ineq:S1XContraction}
\end{equation}  
To prove \eqref{ineq:ssS1grad}, it suffices to verify the spectral gap-type estimate 
\begin{equation}
\norm{S_1(\tau)f}_{X} \lesssim e^{(-\lambda_0 + \gamma)\tau}\norm{f}_X, \label{ineq:S1XSpecGap}
\end{equation}
since \eqref{ineq:S1XSpecGap} and \eqref{ineq:S1XContraction} together imply for $\tau > \tau_0$ (the difference between divergence and gradient is not relevant to the final estimate after adjusting the implicit constant),
\begin{align*}
\norm{S_1(\tau)\grad \cdot f}_1 \leq \norm{S_1(\tau)\grad \cdot f}_X & \lesssim e^{(-\lambda_\alpha + \gamma)(\tau - \tau_0)} \norm{S_1(\tau_0) \grad \cdot f}_{X} \\ 
& \lesssim_{\tau_0} \frac{e^{(-\lambda_\alpha + \gamma)\tau}}{a(\tau)^{1/2}} \norm{f}_{1}. 
\end{align*}
To prove \eqref{ineq:S1XSpecGap} we follow a procedure similar to what is carried out above to prove \eqref{ineq:SpecGapT}. 
By writing the solution $f(\tau) = S_1(\tau)f_0$ as the integral equation
\begin{equation*}
f(\tau) = S(\tau)f_0 - \int_0^\tau S(\tau-\tau^\prime)\grad \cdot (f(\tau^\prime) \grad c_\alpha) d\tau^\prime,
\end{equation*}
similar to Lemma \ref{lem:compact_perturb}, $S_1(\tau)$ can be shown to be a compact perturbation of $S(\tau)$, which has spectral radius $e^{-\tau/2}$ in $X$ \cite{GallagherGallay05}. It then suffices to show that all the eigenvalues of the Fokker-Planck operator $Lf - \grad\cdot( f\grad c_\alpha)$ have real part less than $-\lambda_\alpha$ (recall we are assuming without loss of generality that $\lambda_\alpha \in (0,1/2)$). 
Suppose that there exists some $w \in X$ with 
\begin{equation}
Lw - \grad \cdot( w \grad c_\alpha) = \mu w \label{eq:eigenvalS1}
\end{equation} 
 for some $\mu \in \Complex$ with $\textup{Re} \mu > -\lambda_\alpha$.  
As in the proof of \eqref{ineq:SpecGapT}, by the radial symmetry we may can assume $w(r\cos \theta, r\sin\theta)= f(r)e^{in\theta}$ for some $n \in \Integer$ and re-write \eqref{eq:eigenvalS1} as an ODE for $f(r)$. 
Similar to the argument in Lemma \ref{lem:TexpoLoc} (easier as the ODE can be treated as homogeneous), $w\in X$ satisfying \eqref{eq:eigenvalS1} implies 
\begin{equation*}
f(r) \approx K r^{2\mu - 2} + \mathcal{O}(r^{p}e^{-r^2/4}), \;\;\; r \rightarrow \infty, 
\end{equation*}
for some $p \geq 0$ and $K \in \Complex$. However since $w \in X$ and $\textup{Re} \mu > -\lambda_\alpha \geq -1/2$ it follows that $K = 0$. 
Therefore, $w \in L_0^2(G_\alpha^{-1} d\xi)$, but this would contradict the spectral gap \eqref{ineq:FKPSpecGap} for the Fokker-Planck operator. 
Hence, necessarily the spectrum of $Lf - \grad \cdot (f \grad c_\alpha)$ in $X$ must be contained in the set $\set{\lambda \in \Complex: \textup{Re} \lambda \leq -\lambda_\alpha}$. This in turn implies \eqref{ineq:S1XSpecGap} which completes the proof of \eqref{ineq:ssS1grad} and hence \eqref{ineq:SNgrad} in the special case $N = 1$.  

For the next section we also need the following perturbation lemma which shows that $\lambda_\alpha \approx 1/2$ for $\alpha$ small. This is important to ensure that $\lambda_0$ in \eqref{ineq:SNgrad} can be taken independent of $\epsilon$. 
\begin{lemma} \label{lem:S1alphaSmall}
For all $\delta > 0$ sufficiently small, there exists $\alpha_\delta$ such that if $\alpha < \alpha_\delta$ then $S_1(\tau)$ satisfies 
\begin{equation*}
\norm{S_1(\tau)f_0}_X \lesssim_\delta e^{(-1/2 + \delta)\tau}\norm{f_0}_X,
\end{equation*}
where the implicit constant is independent of $\alpha$. 
\end{lemma}
\begin{proof} 
Let $f_0 \in X$ and write $f(\tau) = S_1(\tau)f_0$ in the Duhamel integral form
\begin{equation*}
f(\tau) = S(\tau)f_0 - \int_0^\tau S(\tau - \tau^\prime) \grad \cdot (f(\tau^\prime) \grad c_\alpha) d\tau^\prime.  
\end{equation*}
Using the known spectral gap of $S(\tau)$ in $X$, and Proposition \ref{prop:Galpha} (iv), 
\begin{align*}
e^{\tau/2 -\delta \tau}\norm{f(\tau)}_{X} & \leq e^{\tau/2 - \delta \tau}\norm{S(\tau)f_0}_X + e^{\tau/2 - \delta \tau}\int_0^\tau \norm{S(\tau - \tau^\prime) \grad \cdot (f(\tau^\prime) \grad c_\alpha)}_X d\tau^\prime \\ 
&\lesssim e^{-\delta \tau}\norm{f_0}_X + e^{\tau/2 - \delta \tau} \int_0^\tau e^{-\frac{\tau - \tau^\prime}{2}}\norm{f(\tau^\prime) \grad c_\alpha)}_1 d\tau^\prime \\ 
&\lesssim e^{-\delta/2}\norm{f_0}_X + \alpha \left(\sup_{\tau^\prime \in (0,\tau)} e^{(1/2 - \delta)\tau}\norm{f(\tau^\prime)}_X \right) e^{- \delta \tau} \frac{1}{\delta}\left(e^{\delta \tau} - 1\right). 
\end{align*}
Taking the supremum in $\tau$ of both sides the lemma follows by choosing $\alpha < \delta$.
\end{proof}

\subsubsection{$N > 1$: Multiple Concentrations} \label{sec:multi}
As in \cite{GallagherGallay05}, to extend to multiple concentrations, we use the intuition that if $t/d^2$ is small, then separated concentrations should basically decouple. 
Introduce a nonnegative, smooth cutoff $\chi(x)$ which is one for $\abs{x} \leq 1/2$ and zero for $\abs{x} > 3/4$.  The localizations around the corresponding concentrations are $\chi_i(x)= \chi\left(\frac{x - z_i}{d}\right)$.  Define the opposite localization to be $\chi_0(x) = 1 - \sum_{i = 1}^N \chi_i(x)$.  Proceeding as in \cite{GallagherGallay05}, if $f(t,x)$ is a solution of \eqref{eq:SNDefinition}, define $f_i(t,x) = \chi_i(x)f(t,x)$ and note that for $i \in \set{1,..,N}$,
\begin{equation*}
\frac{df_i}{dt} + \grad \cdot \left(\frac{1}{\sqrt{t}}v^{G_i}\left(\frac{x-z_i}{\sqrt{t}}\right) f_i\right) = \Delta f_i + Q_i f - \grad \cdot (R_i f), 
\end{equation*}
with 
\begin{align*}
R_i(t,x) & = \sum_{j \neq i} \frac{1}{\sqrt{t}}v^{G_j}\left( \frac{x-z_j}{\sqrt{t}} \right) \chi_i(x) + 2 \grad \chi_i(x),  \\ 
Q_i(t,x) & =  \sum_{j = 1}^N \frac{1}{\sqrt{t}}v^{G_j}\left( \frac{x-z_j}{\sqrt{t}} \right) \cdot \grad \chi_i(x) + \Delta \chi_i(x).
\end{align*}
For $i = 0$ we have (with the remainders $Q_0$, $R_0$ defined similarly), 
\begin{equation*}
\frac{df_0}{dt} = \Delta f_0 + Q_0 f - \grad \cdot (R_0 f). 
\end{equation*}
Note that the a priori estimates on $f$, Proposition \ref{prop:Galpha} and the definition of the cutoffs imply
\begin{align*}
\norm{\sum_{i = 0}^N \abs{R_i(t)} }_\infty \lesssim \frac{1}{d}, \;\;\; \norm{\sum_{i = 0}^N \abs{Q_i(t)}}_\infty \lesssim \frac{1}{d^2},  
\end{align*}
for all $t \geq 0$. 
Denoting $S^i(t,s)$ the linear propagator associated with the concentration centered at $z_i$ we have, 
  for $t > s > 0$, 
\begin{equation}
f_i(t) = S^i(t,s)f_i(s) + \int_s^t S^i(t,t^\prime)\left[Q_i(t^\prime)f(t^\prime) - \grad \cdot (R_i(t^\prime)f(t^\prime))\right] dt^\prime. 
\end{equation}
By the previous section, we know that each $S_i(t,s)$ satisfies 
\begin{align*} 
\norm{S^i(t,s) w}_p & \lesssim \frac{1}{(t-s)^{1-\frac{1}{p}}}\norm{w}_{\mathcal{M}(\Real^2)}, \\ 
\limsup_{t \rightarrow 0^+} t^{1-\frac{1}{p}}\norm{S^i(t,0) w}_p & \lesssim \norm{w}_{pp}, \\ 
\norm{S^i(t,s)\grad w}_1 &\lesssim \frac{1}{(t-s)^{1/2}}\left[\frac{t}{s}\right]^{\gamma + 1/2 - \lambda_{\alpha_i}} \norm{w}_1. 
\end{align*} 
From here we may proceed as in Proposition 4.3 of \cite{GallagherGallay05} to finish the proof of the Proposition \ref{prop:SNProperties} with $\lambda_0 := \min_{1 \leq i \leq N}(\lambda_{\alpha_i})$.
Lemma \ref{lem:S1alphaSmall} shows that $\lambda_0 \in (0,1/2)$ uniformly in $N$, since only large values of $\alpha$ can have a relevant effect on $\lambda_{\alpha}$. 
From the proof we see that we need to choose $t_0$ in Proposition \ref{prop:SNProperties} such that 
\begin{equation*}
\frac{t_0}{d^2} \leq \min\left(1, K\right),
\end{equation*}
where $K$ is some constant which is independent of $\epsilon$, $N$ and $d$. 

\section{Appendix: Properties of Self-Similar Solutions} \label{apx:PropSelfSim}
\subsection{Sketch of Proposition \ref{prop:Galpha}}
\subsubsection{Part (i)} 
There are several methods for proving existence of a self-similar solution with finite energy, for example see \cite{BilerAccretion95,NaitoSuzuki04}. 
Another approach is to use the direct method of calculus of variations to produce a non-negative global minimizer to the self-similar free energy \eqref{def:Gssfree} which satisfies $G_\alpha = G_\alpha^\star$. 
As $G_\alpha$ is a finite energy solution to \eqref{def:resPKS} with $\norm{G_\alpha}_1 = \alpha < 8\pi$, it follows by standard iteration methods that $G_\alpha \in L^\infty$ (see e.g. \cite{CalvezCarrillo10}) 
with norm that depends on $\mathcal{G}(G_\alpha)$ and $\alpha$. 
By bootstrapping elliptic regularity it follows that $G_\alpha \in C^\infty$ and strictly positive. 
As $G_\alpha$ satisfies the Euler-Lagrange equation for \eqref{def:Gssfree}, for a given mass $\alpha \in (0,8\pi)$, it also follows that
\begin{equation}
G_{\alpha}(\xi) = \alpha \frac{e^{c_\alpha(\xi) - \abs{\xi}^2/4}}{\int_{\Real^2} e^{c_\alpha(\zeta) - \abs{\zeta}^2/4} d\zeta} = -\Delta c_\alpha. \label{eq:Gformula}
\end{equation}
By Lemma 4.3 in \cite{BlanchetEJDE06} it follows that for $\abs{\xi} \geq 1$, there is some constant $\bar C = \bar C(\alpha,\mathcal{G}(G_\alpha))$ such that
\begin{align} 
\abs{c_\alpha(\xi) + \frac{\alpha}{2\pi} \log \abs{\xi}} \leq \bar C. \label{ineq:BarC}
\end{align} 
Near the origin we have something better: 
\begin{align} 
\sup_{\abs{\zeta} \leq 1} \abs{c_\alpha(\zeta)} & \lesssim \sup_{\abs{\zeta} \leq 1} \left(  \abs{\int_{\abs{\eta -\zeta} > 2} \log\abs{\eta - \zeta} G_\alpha(\eta) d\eta} + \abs{\int_{\abs{\eta -\zeta} \leq 2} \log\abs{\eta - \zeta} G_\alpha(\eta) d\eta}  \right) \nonumber \\ 
& \lesssim \sup_{\abs{\zeta} \leq 1} \int_{\abs{\eta -\zeta} > 2} \frac{\log\abs{\eta - \zeta}}{\abs{\eta-\zeta}} \abs{\eta-\zeta} G_\alpha(\eta) d\eta + \norm{G_\alpha}_\infty \nonumber \\ 
& \lesssim \sup_{\abs{\zeta} \leq 1} \int_{\abs{\eta -\zeta} > 2} \left(1+\abs{\eta}\right)G_\alpha(\eta) d\eta + \norm{G_\alpha}_\infty \lesssim_{\bar C} 1,  \label{ineq:supC}
\end{align} 
from which it follows that 
\begin{align*} 
\int_{\Real^2} e^{c_\alpha(\zeta) - \abs{\zeta}^2/4} d\zeta & = \int_{\abs{\zeta} > 1} e^{c_\alpha(\zeta) - \abs{\zeta}^2/4} d\zeta + \int_{\abs{\zeta} \leq 1} e^{c_\alpha(\zeta) - \abs{\zeta}^2/4} d\zeta \\ 
& \leq e^{\bar C}\int_{\abs{\zeta} > 1} \abs{\zeta}^{\frac{-\alpha}{2\pi}}e^{- \abs{\zeta}^2/4} d\zeta
 + e^{\sup_{\abs{\zeta} \leq 1} \abs{c_\alpha(\zeta)}} \int_{\abs{\zeta} \leq 1} e^{-\abs{\zeta}^2/4} d\zeta \\ 
& \lesssim_{\bar C} 1,    
\end{align*} 
Similarly, 
\begin{align*} 
\int_{\Real^2} e^{c_\alpha(\zeta) - \abs{\zeta}^2/4} d\zeta & \geq e^{-\bar C}\int_{\abs{\zeta} > 1} \abs{\zeta}^{\frac{-\alpha}{2\pi}}e^{- \abs{\zeta}^2/4} d\zeta \gtrsim e^{-\bar C}. 
\end{align*} 
Therefore, for $\abs{\xi} \geq 1$, 
\begin{align*} 
\abs{\xi}^{-\frac{\alpha}{2\pi}}e^{-\abs{\xi}^2/4} \lesssim_{\bar C} G_\alpha(\xi) \lesssim_{\bar C} \abs{\xi}^{-\frac{\alpha}{2\pi}}e^{-\abs{\xi}^2/4}.
\end{align*} 
This allows one to improve Lemma 4.3 in \cite{BlanchetEJDE06} further to deduce the slightly more precise
\begin{align} 
\lim_{\xi \rightarrow \infty}\abs{c_\alpha(\xi) + \frac{\alpha}{2\pi} \log \abs{\xi}} = 0, \label{ineq:GcalphaAsymptotic}
\end{align} 
and hence \eqref{eq:Galpha_limit}. 
Using a similar technique, one also obtains the following for $\abs{\xi} \geq 1$, which implies \eqref{eq:GradGalpha_limit}  will be useful in several other places, 
\begin{align} 
\abs{\grad c_\alpha(\xi) + \frac{\alpha \xi}{2\pi \abs{\xi}^2}} \lesssim \frac{1}{\abs{\xi}^2}. \label{ineq:GradcalphaAsymptotic} 
\end{align} 
Moreover, by applying the above estimates along with Proposition \ref{prop:Vel} to \eqref{eq:Gformula}, one derives \eqref{ineq:GalphaBoundsLarge}. 
 
\subsubsection{Part (ii)}
In \cite{Biler06}, Biler et. al. show that $G_\alpha$ is the unique, radially symmetric, self-similar solution to the PKS (equivalent to fixed points of the self-similar PDE \eqref{def:resPKS} and hence critical points of $\mathcal{G}$ \eqref{def:Gssfree}). 
In \cite{NaitoSuzuki04,NaitoSuzukiYoshida02} (see also \cite{Naito01}) it is shown using the moving planes method that any self-similar solution to \eqref{def:PKS} must be radially symmetric, and hence $G_\alpha$ is the unique self-similar solution.  
One can prove the same result with a symmetrization argument as follows. 
Suppose that there existed a non-radially symmetric self-similar solution $u(\xi)$ with mass $\alpha \in (0,8\pi)$ and finite self-similar energy.  
Denote $\tilde{u}(\tau,\xi)$ the radially symmetric solution to \eqref{def:resPKS} with initial data $\tilde{u}(0,\xi) = u^\star(\xi)$, the Riesz symmetric decreasing rearrangement of $u$.
By the symmetrization inequalities in \cite{DiazNagai95,DiazNagaiRakotoson98}, 
we know that $\tilde{u}$ dominates $u$ in the sense of mass concentration:
\begin{equation}
u \prec \tilde{u}(\tau), \;\;\; \forall \, \tau \geq 0, \label{ineq:mass_concentration}  
\end{equation}
where if $f$, $g$ are two integrable functions, $f \prec g$ denotes
\begin{equation*}
\int_{\abs{x} < R} f^\star(x) dx \leq \int_{\abs{x} < R} g^\star(x) dx, \;\;\; \forall \, R>0. 
\end{equation*}
Since $G_\alpha$ is the unique radially symmetric stationary point of \eqref{def:resPKS}, a compactness argument using the energy dissipation inequality for \eqref{def:Gssfree} implies that $\tilde{u}(\tau) \rightarrow G_\alpha$ as $\tau \rightarrow \infty$ in $L^p$ for all $p \in [1,\infty)$. 
Passing to the limit in \eqref{ineq:mass_concentration} implies
\begin{equation*}
u \prec G_\alpha. 
\end{equation*}
Further, note that any stationary solution to \eqref{def:resPKS} with mass $\alpha$ satisfies the virial-type identity
\begin{equation*}
0 = \frac{d}{dt}\int \abs{\xi}^2 u(\xi) d\xi = 4\alpha\left( 1 - \frac{\alpha}{8\pi} \right) - \int \abs{\xi}^2 u(\xi) d\xi,  
\end{equation*}
hence both $G_\alpha$ and $u$ have the same second moment. 
An elementary lemma regarding the Riesz symmetric decreasing rearrangement shows that these facts together imply $u(\xi) = G_\alpha(\xi)$ \cite{GallagherGallayLions05}.     

\subsubsection{Part (iii)} 
Follows from part (ii) using an argument similar to what is employed in \cite{GallayWayne05}. 

\subsubsection{Part (iv)}
By the above, the $G_\alpha$ are the unique solutions to the system
\begin{align*}
\grad \cdot ( G_\alpha \grad c_\alpha) & = LG_\alpha, \\ 
-\Delta c_\alpha & = G_\alpha. 
\end{align*}
The operator $L$ can only be inverted up to the zero eigenmode, given by the Gaussian $G(\xi)$,
\begin{equation*}
G(\xi) = \frac{e^{-\abs{\xi}^2/4}}{(4\pi)^{1/2}}. 
\end{equation*}
Hence, we can write $G_\alpha$ as
\begin{equation}
G_\alpha = \alpha G + L^{-1}\grad \cdot (G_\alpha \grad c_\alpha), \label{eq:GalphaPerturb}
\end{equation}
which is amenable to contraction mapping arguments. 
We need the following lemma regarding the linear operator $L$. 
\begin{lemma} \label{lem:Lbounded} 
\begin{itemize}
\item[(i)] $L$ satisfies the following for all $p \in (1,2]$ and $m > 4$, 
\begin{equation}
\norm{L^{-1}\grad f}_{L^2(m)} \lesssim \norm{f}_{L^p(m)}. \label{ineq:LinvGrad}
\end{equation}
\item[(ii)] $L$ satisfies the following for all $p \in (1,2]$ and $m > 4$, 
\begin{equation}
\norm{\grad L^{-1}\grad f}_4 \lesssim \norm{f}_{4} + \norm{f}_{L^p(m)} \label{ineq:Linv2xGrad}
\end{equation} 
\end{itemize}
\end{lemma}
\begin{proof}
Observe the formula, 
\begin{equation*}
L^{-1}\grad f = -\int_0^\infty S(\tau- \tau^\prime) \grad f d\tau^\prime. 
\end{equation*}
Using a compactness argument and the estimates \eqref{ineq:gradSDecay}, \eqref{eq:SgradCommute}, one can justify the convergence of the integral for $f \in L^q(m)$, $q \in (1,2]$ and in particular we see that \eqref{ineq:gradSDecay} implies \eqref{ineq:LinvGrad}.  
Write $u = L^{-1}\grad f$ and note that since $\Delta u = \grad f - \frac{1}{2}\grad \cdot (x u)$ the Calder\'on-Zygmund inequality implies (or since this is $L^2$, really just the Fourier transform), 
\begin{equation*}
\norm{\grad u}_2 \lesssim \norm{f}_2 + \norm{x u}_{2} \leq \norm{f}_2 + \norm{u}_{L^2(m)}.
\end{equation*}
Hence, by Calder\'on-Zygmund followed by H\"older's inequality and Gagliardo-Nirenberg,  
\begin{align*}
\norm{\grad u}_{4} & \lesssim \norm{xu}_4 + \norm{f}_4 \\ 
& \lesssim \norm{u}_{L^2(m)}^{1/4}\norm{u}_6^{3/4} + \norm{f}_4 \\ 
& \lesssim \norm{u}_{L^2(m)}^{1/4}\norm{u}_2^{1/4}\norm{\grad u}_2^{1/2} + \norm{f}_4 \\ 
&  \lesssim \norm{u}_{L^2(m)}^{1/2}\norm{\grad u}_{2}^{1/2} + \norm{f}_4 \\ 
& \lesssim \norm{u}_{L^2(m)} + \norm{\grad u}_2 + \norm{f}_4 \\ 
& \lesssim \norm{u}_{L^2(m)} + \norm{f}_2 + \norm{f}_4 \\ 
 & \lesssim \norm{u}_{L^2(m)} + \norm{f}_1^{1/3}\norm{f}_4^{2/3} + \norm{f}_4 \\ 
 & \lesssim \norm{u}_{L^2(m)} + \norm{f}_1 + \norm{f}_4 \\ 
 & \lesssim \norm{u}_{L^2(m)} + \norm{f}_{L^p(m)} + \norm{f}_4, 
\end{align*}
where the last line followed from H\"older's inequality and $m > 2$. Hence \eqref{ineq:LinvGrad} implies \eqref{ineq:Linv2xGrad}.
\end{proof} 

Turning back to the proof of (iv),  we set up a contraction argument using the norm 
\begin{equation}
\norm{f}_X := \norm{f}_{L^2(m)} + \norm{\grad f}_{4}, \label{def:XPerturb}
\end{equation}
which by Gagliardo-Nirenberg and $m > 2$, embeds into every $L^p$ space for $p \in [1,\infty]$ as well as $L^q(m)$ for all $q \geq 2$.
Indeed, recall that
\begin{equation*}
\norm{f}_\infty \lesssim \norm{f}_{2}^{1/3}\norm{\grad f}^{2/3}_{4} \leq \norm{f}_X,  
\end{equation*}
and since $m > 2$, for $p \in [1,2)$,
\begin{equation*}
\norm{f}_{p} \lesssim \norm{f}_{L^2(m)} \leq \norm{f}_X. 
\end{equation*}
For $\rho \in X$, define the map $F\rho \rightarrow f$ by 
\begin{align*}
f & = \alpha G + L^{-1}\grad \cdot (\rho \grad c_\rho), \\ 
-\Delta c_\rho & = \rho. 
\end{align*}
Using \eqref{ineq:LinvGrad}, for any $p \in (1,2]$,  
\begin{align*}
\norm{f}_{L^2(m)} \lesssim \alpha \norm{G}_{L^2(m)} + \norm{\rho \grad c_\rho}_{L^p(m)}. 
\end{align*}
Using H\"older's inequality, \eqref{ineq:VelLp}, and $m > 2$,   
\begin{align*}
\norm{\rho \grad c_\rho} & \leq \norm{\rho}_{L^2(m)} \norm{\grad c_\rho}_{\frac{2p}{2-p}} \\ 
& \lesssim \norm{\rho}_{L^2(m)} \norm{\rho}_{p} \lesssim \norm{\rho}_{L^2(m)}^2 \leq \norm{\rho}_X^2. 
\end{align*}
Hence, 
\begin{equation*}
\norm{f}_{L^2(m)} \leq \alpha \norm{G}_{L^2(m)} + \norm{\rho}_{X}^2. 
\end{equation*}
A similar argument also shows that if $f_i = F[\rho_i]$, 
\begin{equation*}
\norm{f_1 - f_2}_{L^2(m)} \lesssim \left( \norm{\rho_1}_{L^2(m)} + \norm{\rho_2}_{L^2(m)}\right)\norm{\rho_1 - \rho_2}_{L^2(m)}. 
\end{equation*}
Hence for $m > 2$, using \eqref{ineq:Linv2xGrad}, \eqref{ineq:VelLp} and similar estimates to above,  
\begin{align*}
\norm{\grad f}_4 & \lesssim \alpha \norm{\grad G}_{4} + \norm{\rho \grad c_\rho}_{4} + \norm{\rho \grad c_\rho}_{L^p(m)} \\
 & \lesssim \alpha \norm{\grad G}_{4} + \norm{\rho}_\infty\norm{\rho}_{4/3} + \norm{\rho}_{L^2(m)}^2 \\ 
 & \lesssim \alpha \norm{\grad G}_{4} + \norm{\rho}_{X} \norm{\rho}_{L^2(m)} + \norm{\rho}_{L^2(m)}^2 \\ 
 & \lesssim \alpha \norm{\grad G}_{4} + \norm{\rho}_{X}^2,  
\end{align*}
and similarly, 
\begin{align*}
\norm{\grad f_1 - \grad f_2}_{4} & \lesssim \left( \norm{\rho_1}_X + \norm{\rho_2}_X\right)\norm{\rho_1 - \rho_2}_{X}. 
\end{align*}
Hence, we have the two estimates
\begin{align*} 
\norm{f}_X & \lesssim \alpha\norm{G}_X + \norm{\rho}_X^2 \\ 
\norm{f_1 - f_2}_X & \lesssim \left( \norm{\rho_1}_X + \norm{\rho_2}_X\right)\norm{\rho_1 - \rho_2}_{X},  
\end{align*} 
which together with the contraction mapping theorem (applied in a small ball defined by $\norm{\cdot}_X$) implies that for $\alpha$ sufficiently small, provided $m > 4$, $\norm{G_\alpha}_X \lesssim \alpha$, which in turn proves the following estimate: for all $1 \leq p \leq \infty$ and $m > 4$, 
\begin{align} 
\norm{G_\alpha}_{L^p} + \norm{G_\alpha}_{L^2(m)} \lesssim_{p,m} \alpha. \label{ineq:GalphaPerturb}
\end{align}
From \eqref{ineq:GalphaPerturb}, one can show that $\bar C$ in \eqref{ineq:BarC} satisfies $\bar C \lesssim \alpha$
and hence that for $\alpha$ sufficiently small, 
\begin{align*} 
1 \lesssim \int e^{c_\alpha(\zeta) - \abs{\zeta}^2/4} d\zeta \lesssim 1, 
\end{align*} 
independent of $\alpha$. Similarly, the bound in \eqref{ineq:supC} can be taken to be $O(\alpha)$ for $\alpha$ small. 
From this, \eqref{eq:Gformula}, \eqref{ineq:GalphaPerturb} and Proposition \ref{prop:Vel}, the bounds \eqref{ineq:GalphaBounds} follow.  

\subsubsection{Part (v)}
The argument for (v) is similar to that for (iv).  
For all $\alpha \in (0,8\pi)$, $G_\alpha$ is the unique solution of the system
\begin{align*}
\grad \cdot ( G_\alpha \grad c_\alpha) & = LG_\alpha \\ 
-\Delta c_\alpha & = G_\alpha. 
\end{align*}
Defining $f := G_\alpha - G_\beta$, subtracting the systems satisfied for each self-similar solution and rearranging we get the following elliptic system for $f$: 
\begin{align}
\grad \cdot (f \grad c) &= (L-\Lambda_\alpha)f \label{eq:PerturbG} \\ 
-\Delta c & = f \nonumber
\end{align} 
Let $E_\alpha^0$ be the zero eigenfunction of $L - \Lambda_\alpha$ in $L^2(G_\alpha^{-1} d\xi)$, defined above by \eqref{def:E0alphaPDE}.  First note that Proposition \ref{prop:SpecT} implies
 implies $L-\Lambda_\alpha$ can be inverted uniquely in $L^2_0(m)$ for $m>2$ by using \eqref{ineq:SpecGapT} to justify the formula
\begin{align*} 
\left(L-\Lambda_\alpha\right)^{-1}g = -\int_0^\infty \mathcal{T}_\alpha(\tau-\tau^\prime)g d\tau,  
\end{align*}  
for $g \in L_0^2(m)$ (note that $\mathcal{T}_\alpha$ preserves the mean zero property).
Therefore we may formally write the solution to \eqref{eq:PerturbG} as follows ($\int E_\alpha^0 d\xi = 1$; see \S\ref{apx:SpectralGap}), 
\begin{align*} 
f  =  \left(\int f d\xi \right)E_\alpha^0 + \left(L - \Lambda_\alpha\right)^{-1}\grad \cdot (f \grad c). 
\end{align*} 
With $X$ defined in \eqref{def:XPerturb}, one can bootstrap the result of Lemma \ref{lem:EalphaControl} (see \S\ref{apx:SpectralGap})  with \eqref{def:E0alphaPDE} and elliptic regularity to prove,   
\begin{equation*}
\norm{E_\alpha^0}_{X} \lesssim_\alpha 1,   
\end{equation*}
from which it follows that $\norm{\left(\int f d\xi \right)E_\alpha^0}_X \lesssim \abs{\alpha - \beta}$. 
Hence, the same contraction mapping argument used to deduce (iv) can be used here provided we have the analogous lemma.  
\begin{lemma} \label{lem:LalphaBounded} 
\begin{itemize}
\item[(i)] $L - \Lambda_\alpha$ satisfies the following for all $p \in (1,2]$ and $m > 2$, 
\begin{equation}
\norm{(L - \Lambda_\alpha)^{-1}\grad f}_{L^2(m)} \lesssim \norm{f}_{L^p(m)}. \label{ineq:LalphaGrad}
\end{equation}
\item[(ii)] $L - \Lambda_\alpha$ satisfies the following for all $p \in (4/3,2]$ and $m > 4$, 
\begin{equation}
\norm{\grad (L - \Lambda_\alpha)^{-1}\grad f}_4 \lesssim \norm{f}_{4} + \norm{f}_{L^p(m)}. \label{ineq:Lalpha2xGrad}
\end{equation} 
\end{itemize}
\end{lemma}
\begin{proof} 
Inequality \eqref{ineq:LalphaGrad} follows similarly to \eqref{ineq:LinvGrad}, using the estimates collected in Proposition \ref{prop:SpecT}.
Denote $u = (L-\Lambda_\alpha)^{-1}\grad f$ as the unique solution in $L_0^2(m)$ to
\begin{equation*}
Lu - \Lambda_\alpha u = \grad f, 
\end{equation*}
and denote $c$ the solution to $-\Delta c = u$. 
Similar to the proof of \eqref{ineq:Linv2xGrad} we use the Calder\'on-Zygmund inequality which implies (using also \eqref{ineq:VelLp}),  
\begin{align*}
\norm{\grad u}_2 & \lesssim \norm{f}_2 + \norm{xu}_2 + \norm{G_\alpha \grad c}_2 + \norm{u \grad c_\alpha}_2 \\ 
& \lesssim_\alpha \norm{f}_2 + \norm{u}_{L^2(m)} + \norm{\grad c}_4 + \norm{u}_2 \\
&  \lesssim_\alpha \norm{f}_2 + \norm{u}_{L^2(m)} + \norm{u}_{4/3} \\ 
& \lesssim_\alpha \norm{f}_2 + \norm{u}_{L^2(m)}, 
\end{align*}
where the last line followed from $m > 2$. 
From here, the corresponding $L^4$ estimate implies \eqref{ineq:Lalpha2xGrad} similar to above in the proof of \eqref{ineq:Linv2xGrad}.  
\end{proof}

\vfill\eject
\bibliographystyle{plain}
\bibliography{nonlocal_eqns}

\end{document}